\documentclass[a4paper,10pt,oneside,reqno]{amsart}

\usepackage[utf8]{inputenc}
\usepackage{url}
\usepackage[hidelinks]{hyperref}

\usepackage[DIV=10]{typearea}

\usepackage{microtype}
\usepackage{amssymb,amsmath,amsopn,amsxtra,amsthm,amsfonts}

\let\originallhook\lhook
\usepackage{MnSymbol} 
\newcommand{\lhook}{\mathrel{\raise.018ex\hbox{$\originallhook$}}}

\usepackage[mathcal]{euscript}
\let\mathscr\mathcal
\usepackage[bb=px]{mathalfa}

\usepackage{enumitem}
\setlist[enumerate,1]{label={(\arabic*)},itemsep=\parskip} 
\setlist[itemize,1]{itemsep=\parskip} 
\newlist{thmlist}{enumerate}{2}
\setlist[thmlist,1]{label={\em(\roman*)},ref={(\roman*)},%
  itemsep=\parskip,leftmargin=*,align=left}
\setlist[thmlist,2]{label={\em(\alph*)},ref={(\alph*)},%
  itemsep=\parskip,leftmargin=*,align=left,topsep=0.1cm}
\newlist{defnlist}{enumerate}{2}
\setlist[defnlist,1]{label={(\roman*)},ref={(\roman*)},itemsep=\parskip,%
  leftmargin=*,align=left}
\setlist[defnlist,2]{label={(\alph*)},ref={(\alph*)},itemsep=\parskip,%
  leftmargin=*,align=left,topsep=0.1cm}

\newtheoremstyle{plain}
  {.5\baselineskip\@plus.2\baselineskip\@minus.2\baselineskip}
  {0\baselineskip\@plus.2\baselineskip\@minus.2\baselineskip\@plus.5em}
  {\slshape}
  {}
  {\bfseries}
  {.}
  { }
  {}

\newtheoremstyle{definition}
  {.5\baselineskip\@plus.2\baselineskip\@minus.2\baselineskip}
  {0\baselineskip\@plus.2\baselineskip\@minus.2\baselineskip\@plus.5em}
  {}
  {}
  {\bfseries}
  {.}
  { }
  {}

\theoremstyle{plain}

\newtheorem{thmX}{Theorem}

\newtheorem{thm}[subsubsection]{Theorem}
\newtheorem{cor}[subsubsection]{Corollary}
\newtheorem{lem}[subsubsection]{Lemma}
\newtheorem{prop}[subsubsection]{Proposition}

\theoremstyle{definition}
\newtheorem{defn}[subsubsection]{Definition}
\newtheorem{rem}[subsubsection]{Remark}

\newtheorem{exam}[subsubsection]{Example}

\newtheorem{constr}[subsubsection]{Construction}
\newtheorem{notation}[subsubsection]{Notation}
\newtheorem{warn}[subsubsection]{Warning}

\newcommand{\constrref}[1]{Construction~\ref{#1}}
\newcommand{\thmref}[1]{Theorem~\ref{#1}}

\newcommand{\secref}[1]{Sect.~\ref{#1}}
\newcommand{\ssecref}[1]{Subsect. ~\ref{#1}}

\newcommand{\lemref}[1]{Lemma~\ref{#1}}
\newcommand{\propref}[1]{Proposition~\ref{#1}}
\newcommand{\corref}[1]{Corollary~\ref{#1}}

\newcommand{\remref}[1]{Remark~\ref{#1}}
\newcommand{\defref}[1]{Definition~\ref{#1}}

\renewcommand{\eqref}[1]{(\ref{#1})}

\newcommand{\examref}[1]{Example~\ref{#1}}

\numberwithin{equation}{subsection}

\usepackage{tikz}
\usetikzlibrary{matrix}
\usepackage{tikz-cd}
\tikzset{
  commutative diagrams/.cd,
  arrow style=tikz,
  diagrams={>=latex}}
\makeatletter
\tikzset{
  column sep/.code=\def\pgfmatrixcolumnsep{\pgf@matrix@xscale*(#1)},
  row sep/.code   =\def\pgfmatrixrowsep{\pgf@matrix@yscale*(#1)},
  matrix xscale/.code=%
    \pgfmathsetmacro\pgf@matrix@xscale{\pgf@matrix@xscale*(#1)},
  matrix yscale/.code=%
    \pgfmathsetmacro\pgf@matrix@yscale{\pgf@matrix@yscale*(#1)},
  matrix scale/.style={/tikz/matrix xscale={#1},/tikz/matrix yscale={#1}}}
\def\pgf@matrix@xscale{1}
\def\pgf@matrix@yscale{1}
\makeatother

\newcommand{\changelocaltocdepth}[1]{%
  \addtocontents{toc}{\protect\setcounter{tocdepth}{#1}}%
  \setcounter{tocdepth}{#1}}

\usepackage{xspace}

\usepackage{xifthen}

\usepackage{xparse}

\usepackage{mathtools}

\newcommand{\nc}{\newcommand}
\nc{\renc}{\renewcommand}
\nc{\ssec}{\subsection}
\nc{\sssec}{\subsubsection}
\nc{\on}{\operatorname}
\nc{\term}[1]{#1\xspace}

\nc{\sA}{\ensuremath{\mathcal{A}}\xspace}
\nc{\sB}{\ensuremath{\mathcal{B}}\xspace}
\nc{\sC}{\ensuremath{\mathcal{C}}\xspace}
\nc{\sD}{\ensuremath{\mathcal{D}}\xspace}
\nc{\sE}{\ensuremath{\mathcal{E}}\xspace}
\nc{\sF}{\ensuremath{\mathcal{F}}\xspace}
\nc{\sG}{\ensuremath{\mathcal{G}}\xspace}
\nc{\sH}{\ensuremath{\mathcal{H}}\xspace}
\nc{\sI}{\ensuremath{\mathcal{I}}\xspace}
\nc{\sJ}{\ensuremath{\mathcal{J}}\xspace}
\nc{\sK}{\ensuremath{\mathcal{K}}\xspace}
\nc{\sL}{\ensuremath{\mathcal{L}}\xspace}
\nc{\sM}{\ensuremath{\mathcal{M}}\xspace}
\nc{\sN}{\ensuremath{\mathcal{N}}\xspace}
\nc{\sO}{\ensuremath{\mathcal{O}}\xspace}
\nc{\sP}{\ensuremath{\mathcal{P}}\xspace}
\nc{\sQ}{\ensuremath{\mathcal{Q}}\xspace}
\nc{\sR}{\ensuremath{\mathcal{R}}\xspace}
\nc{\sS}{\ensuremath{\mathcal{S}}\xspace}
\nc{\sT}{\ensuremath{\mathcal{T}}\xspace}
\nc{\sU}{\ensuremath{\mathcal{U}}\xspace}
\nc{\sV}{\ensuremath{\mathcal{V}}\xspace}
\nc{\sW}{\ensuremath{\mathcal{W}}\xspace}
\nc{\sX}{\ensuremath{\mathcal{X}}\xspace}
\nc{\sY}{\ensuremath{\mathcal{Y}}\xspace}
\nc{\sZ}{\ensuremath{\mathcal{Z}}\xspace}

\nc{\bA}{\ensuremath{\mathbf{A}}\xspace}
\nc{\bB}{\ensuremath{\mathbf{B}}\xspace}
\nc{\bC}{\ensuremath{\mathbf{C}}\xspace}
\nc{\bD}{\ensuremath{\mathbf{D}}\xspace}
\nc{\bE}{\ensuremath{\mathbf{E}}\xspace}
\nc{\bF}{\ensuremath{\mathbf{F}}\xspace}
\nc{\bG}{\ensuremath{\mathbf{G}}\xspace}
\nc{\bH}{\ensuremath{\mathbf{H}}\xspace}
\nc{\bI}{\ensuremath{\mathbf{I}}\xspace}
\nc{\bJ}{\ensuremath{\mathbf{J}}\xspace}
\nc{\bK}{\ensuremath{\mathbf{K}}\xspace}
\nc{\bL}{\ensuremath{\mathbf{L}}\xspace}
\nc{\bM}{\ensuremath{\mathbf{M}}\xspace}
\nc{\bN}{\ensuremath{\mathbf{N}}\xspace}
\nc{\bO}{\ensuremath{\mathbf{O}}\xspace}
\nc{\bP}{\ensuremath{\mathbf{P}}\xspace}
\nc{\bQ}{\ensuremath{\mathbf{Q}}\xspace}
\nc{\bR}{\ensuremath{\mathbf{R}}\xspace}
\nc{\bS}{\ensuremath{\mathbf{S}}\xspace}
\nc{\bT}{\ensuremath{\mathbf{T}}\xspace}
\nc{\bU}{\ensuremath{\mathbf{U}}\xspace}
\nc{\bV}{\ensuremath{\mathbf{V}}\xspace}
\nc{\bW}{\ensuremath{\mathbf{W}}\xspace}
\nc{\bX}{\ensuremath{\mathbf{X}}\xspace}
\nc{\bY}{\ensuremath{\mathbf{Y}}\xspace}
\nc{\bZ}{\ensuremath{\mathbf{Z}}\xspace}

\nc{\dA}{\ensuremath{\mathds{A}}\xspace}
\nc{\dB}{\ensuremath{\mathds{B}}\xspace}
\nc{\dC}{\ensuremath{\mathds{C}}\xspace}
\nc{\dD}{\ensuremath{\mathds{D}}\xspace}
\nc{\dE}{\ensuremath{\mathds{E}}\xspace}
\nc{\dF}{\ensuremath{\mathds{F}}\xspace}
\nc{\dG}{\ensuremath{\mathds{G}}\xspace}
\nc{\dH}{\ensuremath{\mathds{H}}\xspace}
\nc{\dI}{\ensuremath{\mathds{I}}\xspace}
\nc{\dJ}{\ensuremath{\mathds{J}}\xspace}
\nc{\dK}{\ensuremath{\mathds{K}}\xspace}
\nc{\dL}{\ensuremath{\mathds{L}}\xspace}
\nc{\dM}{\ensuremath{\mathds{M}}\xspace}
\nc{\dN}{\ensuremath{\mathds{N}}\xspace}
\nc{\dO}{\ensuremath{\mathds{O}}\xspace}
\nc{\dP}{\ensuremath{\mathds{P}}\xspace}
\nc{\dQ}{\ensuremath{\mathds{Q}}\xspace}
\nc{\dR}{\ensuremath{\mathds{R}}\xspace}
\nc{\dS}{\ensuremath{\mathds{S}}\xspace}
\nc{\dT}{\ensuremath{\mathds{T}}\xspace}
\nc{\dU}{\ensuremath{\mathds{U}}\xspace}
\nc{\dV}{\ensuremath{\mathds{V}}\xspace}
\nc{\dW}{\ensuremath{\mathds{W}}\xspace}
\nc{\dX}{\ensuremath{\mathds{X}}\xspace}
\nc{\dY}{\ensuremath{\mathds{Y}}\xspace}
\nc{\dZ}{\ensuremath{\mathds{Z}}\xspace}

\nc{\bbA}{\ensuremath{\mathbb{A}}\xspace}
\nc{\bbB}{\ensuremath{\mathbb{B}}\xspace}
\nc{\bbC}{\ensuremath{\mathbb{C}}\xspace}
\nc{\bbD}{\ensuremath{\mathbb{D}}\xspace}
\nc{\bbE}{\ensuremath{\mathbb{E}}\xspace}
\nc{\bbF}{\ensuremath{\mathbb{F}}\xspace}
\nc{\bbG}{\ensuremath{\mathbb{G}}\xspace}
\nc{\bbH}{\ensuremath{\mathbb{H}}\xspace}
\nc{\bbI}{\ensuremath{\mathbb{I}}\xspace}
\nc{\bbJ}{\ensuremath{\mathbb{J}}\xspace}
\nc{\bbK}{\ensuremath{\mathbb{K}}\xspace}
\nc{\bbL}{\ensuremath{\mathbb{L}}\xspace}
\nc{\bbM}{\ensuremath{\mathbb{M}}\xspace}
\nc{\bbN}{\ensuremath{\mathbb{N}}\xspace}
\nc{\bbO}{\ensuremath{\mathbb{O}}\xspace}
\nc{\bbP}{\ensuremath{\mathbb{P}}\xspace}
\nc{\bbQ}{\ensuremath{\mathbb{Q}}\xspace}
\nc{\bbR}{\ensuremath{\mathbb{R}}\xspace}
\nc{\bbS}{\ensuremath{\mathbb{S}}\xspace}
\nc{\bbT}{\ensuremath{\mathbb{T}}\xspace}
\nc{\bbU}{\ensuremath{\mathbb{U}}\xspace}
\nc{\bbV}{\ensuremath{\mathbb{V}}\xspace}
\nc{\bbW}{\ensuremath{\mathbb{W}}\xspace}
\nc{\bbX}{\ensuremath{\mathbb{X}}\xspace}
\nc{\bbY}{\ensuremath{\mathbb{Y}}\xspace}
\nc{\bbZ}{\ensuremath{\mathbb{Z}}\xspace}

\nc{\mrm}[1]{\ensuremath{\mathrm{#1}}\xspace}
\nc{\mit}[1]{\ensuremath{\mathit{#1}}\xspace}
\nc{\mbf}[1]{\ensuremath{\mathbf{#1}}\xspace}
\nc{\mcal}[1]{\ensuremath{\mathcal{#1}}\xspace}
\nc{\msc}[1]{\ensuremath{\mathscr{#1}}\xspace}
\nc{\mfr}[1]{\ensuremath{\mathfrak{#1}}\xspace}

\renc{\bar}[1]{\overline{#1}}
\DeclarePairedDelimiter\abs{\lvert}{\rvert}%
\nc{\sub}{\subseteq}
\nc{\too}{\longrightarrow}
\nc{\hook}{\hookrightarrow}
\nc{\hooklongrightarrow}{\lhook\joinrel\longrightarrow}
\nc{\hooklong}{\hooklongrightarrow}
\nc{\twoheadlongrightarrow}{\relbar\joinrel\twoheadrightarrow}
\nc{\shiso}{\approx}
\nc{\isoto}{\stackrel{\sim}{\smash{\longrightarrow}\rule{0pt}{0.4ex}}}
\nc{\isofrom}{\xleftarrow{\sim}}
\renc{\ge}{\geqslant}
\renc{\le}{\leqslant}

\nc{\id}{\mathrm{id}}
\DeclareMathOperator{\Ker}{\on{Ker}}
\DeclareMathOperator{\Coker}{\on{Coker}}

\DeclareMathOperator{\Hom}{\on{Hom}}
\nc{\uHom}{\underline{\smash{\Hom}}}
\DeclareMathOperator{\Maps}{\on{Maps}}

\DeclareMathOperator{\End}{\on{End}}
\DeclareMathOperator{\Sym}{\on{Sym}}

\nc{\uEnd}{\underline{\smash{\End}}}

\nc{\colim}{\varinjlim}
\renc{\lim}{\varprojlim}
\nc{\Cofib}{\on{Cofib}}
\nc{\Fib}{\on{Fib}}
\nc{\initial}{\varnothing}
\nc{\op}{\mathrm{op}}

\DeclareMathOperator*{\fibprod}{\times}
\DeclareMathOperator*{\fibcoprod}{\operatorname{\sqcup}}

\nc{\Spc}{\mrm{Spc}}
\nc{\Spt}{\mrm{Spt}}
\nc{\Spec}{\on{Spec}}
\nc{\Stk}{\mrm{Stk}}
\nc{\Sch}{\mrm{Sch}}
\nc{\aff}{\mrm{aff}}
\nc{\A}{\mbf{A}}
\renc{\P}{\mbf{P}}
\nc{\cl}{{\mrm{cl}}}
\nc{\bDelta}{\mathbf{\Delta}}
\nc{\Tor}{\on{Tor}}
\nc{\Cech}{\textnormal{\v{C}}}
\nc{\Mod}{\mrm{Mod}}
\nc{\LMod}{\mrm{LMod}}
\nc{\D}{\on{\mbf{D}}}
\nc{\Qcoh}{\D}
\nc{\Perf}{\on{Perf}}
\nc{\Coh}{\on{Coh}^{\mrm{b}}}
\nc{\free}{\mrm{free}}
\nc{\perf}{\mrm{perf}}
\nc{\aperf}{\mrm{aperf}}
\nc{\coh}{\mrm{coh}}
\nc{\Einfty}{{\sE_\infty}}
\nc{\modmod}{/\!\!/}
\nc{\heart}{\heartsuit}
\nc{\proj}{\mrm{proj}}
\nc{\K}{\on{K}}
\nc{\G}{\bG}
\nc{\GL}{\on{GL}}
\nc{\BGL}{\on{BGL}}
\nc{\M}{\on{M}}
\nc{\KH}{\on{KH}}
\nc{\Alg}{\on{Alg}}
\nc{\CAlg}{\on{CAlg}}
\nc{\cn}{\mrm{cn}}
\nc{\hw}{\mrm{Hw}}
\nc{\htt}{\mrm{Ht}}
\nc{\Fun}{\on{Fun}}
\nc{\Funadd}{\on{Fun}_{\mrm{add}}}
\nc{\Funex}{\on{Fun}_{\mrm{ex}}}
\nc{\Ind}{\on{Ind}}
\nc{\Pro}{\on{Pro}}
\nc{\Kar}{\on{Kar}}
\nc{\Obj}{\on{Obj}}
\nc{\Ex}{\mrm{Ex}}
\nc{\E}{\on{E}}
\nc{\pr}{\mrm{pr}}
\nc{\red}{\mrm{red}}
\nc{\tsX}{\widetilde{\sX}}
\nc{\Pres}{\mrm{Pres}}
\nc{\Presc}{\Pres_\mrm{c}}
\nc{\sMaps}{\mathscr{M}\mathit{aps}} 
\nc{\tr}{\mrm{tr}}
\nc{\Nis}{\mrm{Nis}}
\DeclareMathOperator{\bldim}{\on{bl}\on{dim}}
\DeclareMathOperator{\covdimfppf}{\on{cov}\on{dim}_{\mrm{fppf}}}
\DeclareMathOperator{\covdimsm}{\on{cov}\on{dim}_{\mrm{sm}}}

\nc{\form}{\widehat}
\nc{\Qcohform}{\form{\mbf{D}}}

\nc{\scr}{\term{derived commutative ring}}
\nc{\scrs}{\term{derived commutative rings}}

\nc{\evcoconn}{\term{bounded}}

\nc{\Einfring}{\term{$\Einfty$-ring}}
\nc{\Einfrings}{\term{$\Einfty$-rings}}

\nc{\Ering}{\term{$\sE_1$-ring}}
\nc{\Erings}{\term{$\sE_1$-rings}}

\nc{\inftyCat}{\term{$\infty$-category}}
\nc{\inftyCats}{\term{$\infty$-categories}}

\nc{\pstab}{\term{presentable stable $\infty$-category}}
\nc{\pstabs}{\term{presentable stable $\infty$-categories}}

\nc{\inftyGrpd}{\term{$\infty$-groupoid}}
\nc{\inftyGrpds}{\term{$\infty$-groupoids}}

\title{Categorical Milnor~squares and K-theory of algebraic~stacks}

\author[T. Bachmann]{Tom Bachmann}
\address{Department Mathematik, LMU München, Theresienstr. 39, 80333 München, Germany}
\email{\href{mailto:tom.bachmann@zoho.com}{tom.bachmann@zoho.com}}

\author[A.\,A. Khan]{Adeel A. Khan}
\address{Institute of Mathematics, Academia Sinica, No. 1, Sec. 4, Roosevelt Road, Taipei 106319, Taiwan}
\email{\href{mailto:adeelkhan@gate.sinica.edu.tw}{adeelkhan@gate.sinica.edu.tw}}

\author[C. Ravi]{Charanya Ravi}
\address{Max-Planck-Institut für Mathematik, Vivatsgasse 7, 53111 Bonn, Germany}
\email{\href{mailto:ravi@mpim-bonn.mpg.de}{ravi@mpim-bonn.mpg.de}}

\author[V. Sosnilo]{Vladimir Sosnilo}
\address{Laboratory of Modern Algebra and Applications, St. Petersburg State University, 14th line, 29B, 199178 Saint Petersburg, Russia}
\address{St. Petersburg Department of Steklov Mathematical Institute of Russian Academy of Sciences, Fontanka, 27, 191023 Saint Petersburg, Russia}
\email{\href{mailto:vsosnilo@gmail.com}{vsosnilo@gmail.com}}

\begin{document}

\begin{abstract}
  We introduce a notion of Milnor square of stable \inftyCats and prove a criterion under which algebraic K-theory sends such a square to a cartesian square of spectra.
  We apply this to prove Milnor excision and proper excision theorems in the K-theory of algebraic stacks with affine diagonal and nice stabilizers.
  This yields a generalization of Weibel's conjecture on the vanishing of negative K-groups for this class of stacks.
\end{abstract}

\maketitle

\setcounter{tocdepth}{1}
\parskip 0pt
\tableofcontents

\setlength{\parindent}{0em}
\parskip 0.6em


\section*{Introduction}
\label{sec:intro}

A \emph{Milnor square} of rings, following \cite[Sect.~2]{Milnor}, is a cartesian square
\begin{equation}\label{eq:intro/affine Milnor square}
  \begin{tikzcd}[matrix scale=0.7]
    A \ar[twoheadrightarrow]{r}\ar{d}
    & A/I \ar{d}
    \\
    B \ar[twoheadrightarrow]{r}
    & B/J
  \end{tikzcd}
\end{equation}
where $I \sub A$ and $J\sub B$ are two-sided ideals.
The starting point of this paper is the following result of Land and Tamme, building on work of Morrow, Geisser and Hesselholt, and Suslin (see Corollaries~2.10 and 2.33 in \cite{LandTamme}, as well as \cite{Tamme_2018}, \cite{Morrow}, \cite{GeisserHesselholt} and \cite{Suslin}):

\begin{thm}\label{thm:intro/affine Milnor}\leavevmode
  \begin{thmlist}
    \item
    If the square \eqref{eq:intro/affine Milnor square} is Tor-independent, i.e., the group $\Tor_i^A(A/I, B)$ vanishes for all $i>0$, then the square
    \[ \begin{tikzcd}[matrix scale=0.7]
      \K(A) \ar{r}\ar{d}
      & \K(A/I) \ar{d}
      \\
      \K(B) \ar{r}
      & \K(B/J)
    \end{tikzcd} \]
    is cartesian.

    \item\label{item:uoubu}
    If the pro-system $\{\Tor_i^A(A/I^n, B)\}_{n>0}$ vanishes for all $i>0$, then the square of pro-spectra
    \[ \begin{tikzcd}[matrix scale=0.7]
      \{\K(A)\} \ar{r}\ar{d}
      & \{\K(A/I^n)\}_{n>0} \ar{d}
      \\
      \{\K(B)\} \ar{r}
      & \{\K(B/J^n)\}_{n>0}
    \end{tikzcd} \]
    is cartesian.
  \end{thmlist}
\end{thm}

\begin{rem}
  A sufficient condition for the pro-Tor-independence in \ref{item:uoubu} is pro-Tor-unitality of $I$ (see e.g. \cite[Lem.~2.14]{LandTamme}).
  As observed by Morrow, the latter condition is equivalent to the vanishing of $\{\Tor_i^A(A/I^n, A/I^n)\}_{n>0}$ for all $i>0$ (see \cite[Thm.~0.2]{Morrow}).
  Moreover, this is automatic in the case of noetherian commutative rings (see \cite[Thm.~0.3]{Morrow}).
\end{rem}

Our first goal in this paper is to prove a categorical version of \thmref{thm:intro/affine Milnor}.

\begin{defn}\leavevmode
  \begin{defnlist}
    \item
    Let $\Delta$ be a commutative square of \pstabs and colimit-preserving functors of the form
    \begin{equation}\label{eq:intro/cat Milnor}
      \begin{tikzcd}[matrix scale=0.7]
        \sA \ar{r}{f^*}\ar{d}{p^*}
        & \sB \ar{d}{q^*}
        \\
        \sA' \ar{r}{g^*}
        & \sB'.
      \end{tikzcd}
    \end{equation}
    \begin{defnlist}
      \item
      We say $\Delta$ is a \emph{precartesian} if the canonical functor $(p^*, f^*) : \sA \to \sA' \times_{\sB'} \sB$ is fully faithful.
      \item
      We say $\Delta$ is a \emph{Milnor square} if it is precartesian, and each of the functors $f^*$, $g^*$, $p^*$ and $q^*$ is compact and generates its codomain under colimits.
      \item
      We say $\Delta$ \emph{satisfies base change} if it is vertically right-adjointable; that is, the base change transformation
      $ f^*p_* \to q_*g^* $
      is invertible (where $p_*$ and $q_*$ are the right adjoints of $p^*$ and $q^*$, respectively).
    \end{defnlist}

    \item
    Let $\Delta$ be a commutative square of pro-systems in the \inftyCat of \pstabs and compact colimit-preserving functors.
    We say $\Delta$ is a \emph{pro-Milnor square}, resp. \emph{satisfies pro-base change}, if it can be represented by a cofiltered system $\{\Delta_n\}_n$ where each $\Delta_n$ is a Milnor square, resp. satisfies base change.
  \end{defnlist}
\end{defn}

Given a \pstab $\sA$, we write simply $\K(\sA)$ for the nonconnective algebraic K-theory spectrum of the full subcategory $\sA^\omega$ of compact objects.
Our first main result is as follows (see Theorems~\ref{thm:excision} and \ref{thm:pro excision}):

\begin{thmX}[Categorical Milnor excision]\label{thm:intro/cat}\leavevmode
  \begin{thmlist}
    \item
    Suppose $\Delta$ is a Milnor square of compactly generated stable \inftyCats.
    If $\Delta$ satisfies base change, then the induced square
    \[
      \begin{tikzcd}[matrix scale=0.7]
        \K(\sA) \ar{r}{f^*}\ar{d}{p^*}
        & \K(\sB) \ar{d}{q^*}
        \\
        \K(\sA') \ar{r}{g^*}
        & \K(\sB').
      \end{tikzcd}
    \]
    is cartesian.

    \item
    Suppose $\Delta$ is a pro-Milnor square satisfying the projective generation and boundedness hypotheses of \ref{thm:pro excision}.
    If $\Delta$ satisfies pro-base change, then the induced square of pro-spectra $\K(\Delta)$ is cartesian.
  \end{thmlist}
\end{thmX}

\begin{exam}
  Let $\D(A)$ denote the derived \inftyCat of left $A$-modules over a ring $A$.
  For any Milnor square as in \eqref{eq:intro/affine Milnor square}, the induced square
  \[ \begin{tikzcd}[matrix scale=0.7]
    \D(A) \ar{r}\ar{d}
    & \D(A/I) \ar{d}
    \\
    \D(B) \ar{r}
    & \D(B/J)
  \end{tikzcd} \]
  is a Milnor square of stable \inftyCats.
  It satisfies base change precisely when $\Tor^A_i(A/I, B) = 0$ for all $i>0$, and 
  the induced square of pro-\inftyCats (formed by taking the ideal $I^n$ for $n>0$) satisfies pro-base change precisely when $\{\Tor^A_i(A/I^n, B)\}_{n>0} = 0$ for all $i>0$.
  Thus \thmref{thm:intro/cat} may be regarded as a generalization of \thmref{thm:intro/affine Milnor}.
\end{exam}

\begin{rem}
  Land and Tamme prove \thmref{thm:intro/affine Milnor} as a consequence of a more general result \cite[Thm.~A]{LandTamme} which applies to cartesian squares of $\sE_1$-ring spectra.
  \thmref{thm:intro/cat} also recovers this result when applied to stable \inftyCats of module spectra.
\end{rem}

\begin{rem}
  In \thmref{thm:intro/cat}, nonconnective algebraic K-theory can be replaced by any localizing invariant of stable \inftyCats (which is not required to preserve filtered colimits).
\end{rem}

\thmref{thm:intro/cat} allows us to prove excision statements in the K-theory of algebraic stacks.
For a (quasi-compact quasi-separated) algebraic stack $\sX$, we let $\D(\sX)$ denote the derived \inftyCat of quasi-coherent sheaves on $\sX$, $\Perf(\sX) \sub \D(\sX)$ the full subcategory of perfect complexes, and $\K(\sX)$ the nonconnective algebraic K-theory spectrum of $\Perf(\sX)$.
If $\sZ$ is a closed substack, we regard the formal completion $\sX^\wedge_\sZ$ as an ind-algebraic stack so that there is naturally associated to it a ``continuous K-theory'' pro-spectrum $\form{\K}(\sX^\wedge_\sZ)$.
For example, for $I$ an ideal in a noetherian commutative ring $A$, the continuous K-theory of the formal completion of $\Spec(A)$ along the vanishing locus of $I$ is the pro-spectrum $\{\K(A/I^n)\}_n$.
We then have:

\begin{thmX}[Milnor excision]\label{thm:intro/Milnor stack}
  Let $\Delta$ be a commutative square
  \[ \begin{tikzcd}[matrix scale=0.7]
    \sZ' \ar{r}\ar{d}
    & \sX' \ar{d}{f}
    \\
    \sZ \ar{r}{i}
    & \sX
  \end{tikzcd} \]
  of noetherian \emph{ANS} stacks (algebraic stacks with affine diagonal and nice stabilizers).
  Assume that $\Delta$ is a \emph{Milnor square}: that is, it is cartesian and cocartesian, $f$ is an affine morphism, and $i$ is a closed immersion.
  \begin{thmlist}
    \item\label{item:intro/Milnor stack}
    If $\Delta$ is Tor-independent, then the square
    \[ \begin{tikzcd}[matrix scale=0.7]
      \K(\sX) \ar{r}\ar{d}
      & \K(\sZ) \ar{d}
      \\
      \K(\sX') \ar{r}
      & \K(\sZ')
    \end{tikzcd} \]
    is cartesian.

    \item\label{item:intro/Milnor stack pro}
    The induced square of pro-spectra
    \[ \begin{tikzcd}[matrix scale=0.7]
      \{\K(\sX)\} \ar{r}\ar{d}
      & \form{\K}(\sX^\wedge_\sZ) \ar{d}
      \\
      \{\K(\sX')\} \ar{r}
      & \form{\K}(\sX'^\wedge_{\sZ'})
    \end{tikzcd} \]
    is cartesian.
  \end{thmlist}
\end{thmX}

\begin{thmX}[Proper excision]\label{thm:intro/proper cdh}
  Consider a cartesian square of ANS noetherian algebraic stacks
  \[ \begin{tikzcd}[matrix scale=0.7]
    \sZ' \ar{r}\ar{d}
    & \sX' \ar{d}{f}
    \\
    \sZ \ar{r}{i}
    & \sX
  \end{tikzcd} \]
  where $i$ is a closed immersion and $f$ is a proper representable morphism which is an isomorphism away from $\sZ$.
  Then the induced square of pro-spectra
  \[ \begin{tikzcd}[matrix scale=0.7]
    \{K(\sX)\}\ar{r}\ar{d}
      & \form{K}(\sX^\wedge_{\sZ})\ar{d}
    \\
    \{K(\sX')\}\ar{r}
      & \form{K}(\sX'^\wedge_{\sZ'})
  \end{tikzcd} \]
  is cartesian.
\end{thmX}

\begin{rem}
  \thmref{thm:intro/proper cdh} can be viewed as a generalization of \cite[Thm.~A]{Kerz_2017}, with the minor caveat that a scheme is ANS if and only if it has affine diagonal.
  However, the statement for arbitrary noetherian schemes can be deduced from the case of schemes with affine diagonal using Zariski descent.
\end{rem}

\begin{rem}
  Theorems~\ref{thm:intro/Milnor stack} and \ref{thm:intro/proper cdh} also hold for arbitrary localizing invariants of stable \inftyCats.
  This is an improvement on documented results even in the case of schemes.
  In particular, we find that \thmref{thm:intro/affine Milnor}\ref{item:uoubu} generalizes to arbitrary localizing invariants.
\end{rem}

\begin{rem}
  The condition on nice stabilizers in Theorem~\ref{thm:intro/Milnor stack}, and in the finite case of \thmref{thm:intro/proper cdh}, can be relaxed.
  In fact, \thmref{thm:intro/Milnor stack}\ref{item:intro/Milnor stack} holds for all perfect stacks in the sense of \defref{defn:perfect stack}.
  In \thmref{thm:intro/Milnor stack}\ref{item:intro/Milnor stack pro} and the finite case of \thmref{thm:intro/proper cdh}, it suffices to assume that $\sX$ admits a scallop decomposition $(\sU_i, \sV_i, u_i)_i$ as in \cite[Def.~2.7]{sixstack} where $\sV_i$ are quotients of affines by actions of embeddable group schemes that are \emph{linearly reductive} (but not necessarily nice; compare \thmref{thm:ANS local structure}, \propref{prop:afsoub01}).
  The stronger condition of niceness of all stabilizers is important in our proof of the general case of \thmref{thm:intro/proper cdh} because it ensures that linear reductivity of the stabilizers is preserved under blow-ups.
\end{rem}

Finally, we apply \thmref{thm:intro/proper cdh} to generalize to stacks the proof in \cite{Kerz_2017} of Weibel's conjecture on negative K-theory (see \cite[2.9]{WeibelAnalytic}).

\begin{thmX}[Weibel's conjecture]\label{thm:intro/Weibel}
  Let $\sX$ be a noetherian ANS stack of covering dimension $d$.
  \begin{thmlist}
    \item
    The negative K-groups $\K_{-n}(\sX)$ vanish for all $n>d$.

    \item
    For every vector bundle $\pi : \sE \to \sX$, the map
    $$ \pi^*: \K_{-n}(\sX) \to \K_{-n}(\sE)$$
    is bijective for all $n \ge d$.
  \end{thmlist}
\end{thmX}

\begin{rem}
  For Deligne--Mumford stacks, the covering dimension coincides with the Krull dimension and with the usual dimension as in \cite[Tag~0AFL]{StacksProject}.
  In general, the latter can be negative and so is not suitable for the purposes of \thmref{thm:intro/Weibel}.
  See \ssecref{ssec:dim} for the definitions.
\end{rem}

\begin{exam}[Equivariant K-theory]
  Let $k$ be a field and $G$ a group scheme over $k$ acting on a noetherian finite-dimensional $k$-scheme $X$ with affine diagonal.
  If $G$ is finite of order prime to the characteristic of $k$, or if $G$ is a torus, then the quotient stack $\sX = [X/G]$ is noetherian and ANS (see \ssecref{ssec:ANS} for more examples) and $\K(\sX)$ is the algebraic K-theory $\K^G(X)$ of $G$-equivariant perfect complexes on $X$ as in \cite{Krishna_2018} (see also \cite{ThomasonGroup}).
  Thus \thmref{thm:intro/Weibel} implies that $\K^G_{-n}(X)$ vanishes for all $n > \dim(X)$.
\end{exam}

The following example was pointed out to us by B.~Antieau.

\begin{exam}[Twisted K-theory]
 Let $\sX$ be a noetherian ANS stack of covering dimension $d$ and let $\sY \to \sX$ be a $\G_m$-gerbe over $\sX$.
 Then $\sY$ is also of covering dimension $d$, since $\sY \to \sX$ is smooth and $\sY$ is étale-locally on $\sX$ a trivial $\G_m$-gerbe.
 Therefore by \thmref{thm:intro/Weibel}, $\K_{-n}(\sY)$ vanishes for all $n > d$.
 For any Azumaya algebra $\sA$ over a $d$-dimensional noetherian scheme $X$, we may apply this to the $\G_m$-gerbe $\sY$ over $X$ associated to the Brauer class representing the Azumaya algebra $\sA$.
 Then $\Perf(\sY)$ is equivalent to the stable \inftyCat of $\sA$-twisted perfect complexes on $X$, and in particular we find that the $\sA$-twisted K-groups $\K^{\sA}_{-n}(X)$ of $X$ vanish for all $n > d$.
 This recovers the main result of \cite{stapleton} for Azumaya algebras.
\end{exam}

\begin{rem}
  To our knowledge, \thmref{thm:intro/Weibel} is the first result in the literature about negative K-groups of singular stacks, or even negative equivariant K-groups of singular schemes with group action.
  Recall that in the nonsingular case, all the groups $\K_{-n}(\sX)$ vanish for $n>0$.
  See \cite{KerzICM} for some discussion about the relationship between negative K-groups and singularities in the setting of schemes.
\end{rem}

\ssec*{Outline}

We define categorical Milnor squares, and their pro-versions, in \secref{sec:milnor}.

In \secref{sec:qcoh} we give many examples of categorical Milnor and pro-Milnor squares of derived \inftyCats of (derived) algebraic stacks.
The appearance of pro-\inftyCats comes from the key observation that failure of the base change property can often be rectified by passing to formal completions.
Let $\Qcohform(\sX)$ denote the canonical pro-$\infty$-categorical refinement of the derived \inftyCat $\Qcoh(\sX)$ of a formal stack (or ind-algebraic stack) $\sX$.
Then we prove, for instance, that if
\[ \begin{tikzcd}[matrix scale=0.7]
  \sZ' \ar{r}\ar{d}
  & \sX' \ar{d}
  \\
  \sZ \ar{r}
  & \sX
\end{tikzcd} \]
is either a Milnor square or finite cdh square of noetherian algebraic stacks, then the induced square of pro-\inftyCats
\[ \begin{tikzcd}
  \{\Qcoh(\sX)\} \ar{r}\ar{d}
    & \Qcohform(\sX^\wedge_\sZ) \ar{d}
  \\
  \{\Qcoh(\sX')\} \ar{r}
    & \Qcohform(\sX'^\wedge_{\sZ'})
\end{tikzcd} \]
is not only a pro-Milnor square but also satisfies pro-base change (see Corollaries~\ref{cor:Qcohform Milnor} and \ref{cor:Qcohform finite cdh}).
These results can be viewed as pro-refinements of some results of Halpern-Leistner and Preygel \cite{HLP-h}.

\secref{sec:localizing_Milnor} deals with the proof of \thmref{thm:intro/cat}.
Our main tool is a categorical version of the \emph{$\odot$-construction} introduced in \cite{LandTamme}.
For every categorical Milnor square $\Delta$ of the form \eqref{eq:intro/cat Milnor}, we construct a new square $\Delta_0$
\[ \begin{tikzcd}[matrix scale=0.7]
  \sA \ar{r}{f^*}\ar{d}{p^*}
  & \sB \ar{d}{q_0^*}
  \\
  \sA' \ar{r}{g_0^*}
  & \sA' \odot_{\sA}^{\sB'} \sB,
\end{tikzcd} \]
which is isomorphic to $\Delta$ precisely when the latter satisfies base change (\thmref{thm:odot}).
This is the generalization of \cite[Main~Theorem]{LandTamme} in this setting.
The proof of the pro-variant is somewhat involved, as it requires a user-friendly criterion for a cofiltered system of compact colimit-preserving functors of \pstabs to induce a pro-equivalence (see \corref{cor:proequiv criterion stable}).
This question is surprisingly subtle and forces us to impose strong projective generation hypotheses on our categories, which roughly puts us in the situation of ``weighted'' \inftyCats (see \remref{rem:iohnpa}).
The reader willing to accept \thmref{thm:intro/cat} as a black box can safely skip this section.

Theorems~\ref{thm:intro/Milnor stack} and \ref{thm:intro/proper cdh} are proven in \secref{sec:stacky_applications}.
This involves three main steps:
\begin{itemize}
  \item
  Applying \thmref{thm:intro/cat} to the categorical pro-Milnor squares constructed in Corollaries~\ref{cor:Qcohform Milnor} and \ref{cor:Qcohform finite cdh}, we get \thmref{thm:intro/Milnor stack} as well as \thmref{thm:intro/proper cdh} in the case of finite morphisms.
  See Corollaries~\ref{cor:K formal Milnor excision} and \ref{cor:K formal finite excision}.

  \item
  Next we prove a formally completed version of \cite[Thm.~A]{khan2018algebraic}, which yields formal excision for blow-ups in quasi-smooth derived centres (\propref{prop:formal derived blow-up excision}).
  Combining this with finite excision yields the case of blow-ups (in arbitrary centres).

  \item
  For the general case of \thmref{thm:intro/proper cdh}, we use a generalization of the arguments of \cite{Kerz_2017} to reduce to the case of a blow-up.
  The key ingredient, as in \cite{hoyoiskrishna} and \cite[5.3.4]{khan2018algebraic}, is Rydh's stacky extension of the flatification theorem of Raynaud and Gruson.
  See \thmref{thm:K formal proper excision}.
\end{itemize}

The proof of \thmref{thm:intro/Weibel} is accomplished in \secref{sec:weightstructures}.
As in \cite{Kerz_2017}, the proof is a relatively straightforward consequence of proper excision (\thmref{thm:intro/proper cdh}) and a ``killing lemma'' that allows one to kill negative K-classes by blowing up (see \cite[Prop.~5]{KerzStrunkKH}, \cite[Prop.~7.3]{hoyoiskrishna} for stacks, and \propref{prop:fklkdajjk} for our version with supports).
To use the latter, we also need a nil-invariance result for low enough K-groups (\corref{nilinvariance_general}).

Appendices~\ref{sec:stacks} and \ref{sec:formal} develop some preliminary material on (derived) algebraic and formal stacks.
In appendix~\ref{sec:weak}, we study a ``weak'' version of categorical pro-Milnor squares, where pro-\inftyCats are not regarded up to isomorphism but rather only up to weak pro-equivalence.
This can be used to drop the boundedness condition imposed in \thmref{thm:intro/cat}, and hence also the boundedness conditions on the structure sheaves of derived stacks in some statements in Sects.~\ref{sec:qcoh} and \ref{sec:stacky_applications}.

\ssec*{Related work}

\thmref{thm:intro/cat} can be contrasted with a recent result of \cite[Thm.~18]{Tamme_2018} (cf. \cite[Cor.~13]{HoyoisEfimov}), which provides a different criterion for a square of compactly generated stable \inftyCats as in \eqref{eq:intro/cat Milnor} to induce a cartesian square in algebraic K-theory: namely, it suffices that the square be cartesian and the functor $p_*$ fully faithful.
Our base change condition can be regarded as a generalization of this, since for a Milnor square with $p_*$ fully faithful, the base change property is automatic (see \lemref{lem:jbnoauq}).
However, \thmref{thm:intro/cat} is not strictly more general than Tamme's criterion because the definition of Milnor square also requires that the functors $f^*$ and $g^*$ generate under colimits (or equivalently, that their right adjoints $f_*$ and $g_*$ are conservative).

The idea of passing to formal completions to prove excision statements for derived \inftyCats of quasi-coherent sheaves is present in the work of Halpern-Leistner and Preygel (compare \corref{cor:Qcohform finite cdh} with \cite[Lem.~3.3.4]{HLP-h}, for example).
To our knowledge, the more refined invariant $\Qcohform(\sX^\wedge_\sZ)$ (see \ssecref{ssec:formal/qcoh}) has not been considered in the literature before.

The analogue of Theorem~\ref{thm:intro/proper cdh} in homotopy invariant K-theory was obtained recently by Hoyois and Krishna \cite{hoyoiskrishna}, under slightly weaker hypotheses.
Another proof in homotopy invariant K-theory (that applies more generally to truncating invariants in the sense of \cite{LandTamme}) was given in \cite[Thm.~5.6, Rmk.~5.11(iii)]{khan2018algebraic}, modulo a derived invariance property that was later established independently in \cite{ElmantoSosnilo} and \cite[Cor.~F]{sixstack}.
Our proof combines the arguments of \cite{Kerz_2017} and \cite{khan2018algebraic}.

Hoyois and Krishna also proved a variant of \thmref{thm:intro/Weibel} in homotopy invariant K-theory (see \cite[Thm.~1.1]{hoyoiskrishna}).
Compared to their result, our \thmref{thm:intro/Weibel} applies to a larger class of stacks at the cost of a possibly less sharp bound (using covering dimension instead of blow-up dimension).

\ssec*{Notation and conventions}

We freely use the language of \inftyCats and derived algebraic geometry.
We generally follow the notation of \cite{HTT,HA-20170918}, \cite{GaitsgoryRozenblyum} and \cite{HLP-h}.
Some exceptions are as follows:
\begin{itemize}
  \item
  In an \inftyCat $\sC$, we write $\Maps_\sC(C, D)$ for the mapping space between any two objects $C$ and $D$.

  \item
  We write $\Spc$ and $\Spt$ for the \inftyCats of spaces and spectra, respectively.
  The \inftyCat of presentable \inftyCats and  colimit-preserving functors (see \cite[Defn.~5.5.3.1]{HTT}) will be denoted $\Pres$.
  The (non-full) subcategory of $\Pres$ where the morphisms are \emph{compact} colimit-preserving functors will be denoted $\Presc$.
  Recall that a colimit-preserving functor is compact if its right adjoint preserves filtered colimits (see e.g. \cite[Defn.~C.3.4.2]{SAG-20180204}).

  \item
  A \emph{derived commutative ring} is an object of the nonabelian derived \inftyCat of ordinary commutative rings.
  This \inftyCat can be realized as the localization of the category of simplicial commutative rings.
  See \cite[25.1]{SAG-20180204} for details.
  The reader may also choose to read the term ``derived commutative ring'' as ``connective $\sE_\infty$-ring'' as in \cite[Chap.~7]{HA-20170918}.

  \item
  We write $\LMod_R$ for the stable \inftyCat of left modules over an $\sE_1$-ring spectrum $R$.
  If $R$ is a derived commutative ring, we write $\Mod_R$ for the stable \inftyCat of modules over the underlying $\sE_\infty$-ring spectrum.

  \item
  Given a derived commutative ring $R$ and a collection of elements $f_1,\ldots,f_n \in \pi_0(R)$, we write $R\modmod(f_1,\ldots,f_n)$ for the derived commutative ring of functions on the derived zero locus of $f_1,\ldots,f_n$.
  That is, $R\modmod(f_1,\ldots,f_n)$ is the derived tensor product of $R$ and $\bZ$ over $\bZ[T_1,\ldots,T_n]$ as in \cite[2.3.1]{khan2018virtual}.

  \item
  A derived algebraic stack is a derived 1-Artin stack as in \cite[Chap.~2, 4.1]{GaitsgoryRozenblyum}.
  All derived algebraic stacks are assumed quasi-compact and quasi-separated.
  A derived algebraic stack $\sX$ is \emph{\evcoconn} if it admits a smooth surjection $\Spec(A) \twoheadrightarrow \sX$ where $A$ is a bounded \scr ($\pi_i(A) = 0$ for $i \gg 0$).

  \item
  We write $\Qcoh(\sX)$ for the derived \inftyCat of quasi-coherent sheaves on a derived algebraic stack $\sX$, defined as in \cite[Chap.~3, 1.1.4]{GaitsgoryRozenblyum}.
  When $\sX$ is a classical algebraic stack, this agrees with the derived \inftyCat of $\sO_\sX$-modules with quasi-coherent cohomology (see \cite[Prop.~1.3]{Hall_2017}).
\end{itemize}

\ssec*{Acknowledgments}

We would like to thank the organizers of the event ``Algebraic Groups and Motives'' in St.~Petersburg in September 2019, where the preliminary versions of these results were first obtained.
We thank David Rydh for helpful discussions about \cite{AlperHallRydhLocal} and \cite{ahhr}, and for extensive comments on a draft.
We thank Charles Weibel for his comments on a draft.
The third author thanks Benjamin Antieau for discussions related to the `many-objects version' of the Land--Tamme excision theorem.
The fourth author thanks Elden Elmanto for discussing the relations between various excision results.

The second author acknowledges partial support from SFB 1085 Higher Invariants, Universit\"at Regensburg, the Simons Collaboration on Homological Mirror Symmetry, and MOST 110-2115-M-001-016-MY3.
The third author also acknowledges support from SFB 1085 Higher Invariants.
The fourth author's work on sections 1, 3, 4, 5, and Appendix C was completed under the support of Russian Science Foundation grant 20-41-04401.

\changelocaltocdepth{2}

\section{Milnor squares of stable \texorpdfstring{$\infty$}{oo}-categories}
\label{sec:milnor}

\subsection{The base change property}

  \begin{notation}\label{notat:paisj}
    Recall that any colimit-preserving functor between presentable \inftyCats admits a right adjoint by the adjoint functor theorem.
    We will usually use a symbol of the form $f^*$ to denote such a functor, and $f_*$ for its right adjoint.
    We also write $\eta_f : \id \to f_*f^*$ and $\varepsilon_f : f^*f_* \to \id$ for the unit and co-unit transformations.
    This is just a notational device: we have not assigned any meaning to the symbol ``$f$'' itself.
  \end{notation}

  \begin{constr}[Exchange transformation]
    Let $\Delta$ be a commutative square
    \begin{equation*}
      \begin{tikzcd}
        \sA \arrow[r,"f^*"]\arrow[d,"p^*"] & \sB\arrow[d,"q^*"] \\
        \sA' \arrow[r, "g^*"]              & \sB'
      \end{tikzcd}
    \end{equation*}
    of presentable \inftyCats and colimit-preserving functors.
    The \emph{exchange transformation} associated to $\Delta$ is the canonical natural transformation
    \[
      \Ex_\Delta : f^*p_* \to q_*g^*
    \]
    of functors $\sA' \to \sB$ defined as the composite
    \[
      f^*p_* \xrightarrow{\eta_q} q_*q^*f^*p_* \simeq q_*g^*p^*p_* \xrightarrow{\varepsilon_p} q_*g^*.
    \]
  \end{constr}

  \begin{defn}
    When $\Ex_\Delta$ is invertible, we say that the square $\Delta$ \emph{satisfies base change}\footnote{It is also common to say that the square satisfies the \emph{Beck--Chevalley condition}, or that it is vertically \emph{right-adjointable}.}.
  \end{defn}

  \begin{exam}\label{ex:rings}
    Suppose given a commutative square
    \[
      \begin{tikzcd}
        A \ar{r}\ar{d}
          & B \ar{d}
        \\
        A' \ar{r}
          & B'
      \end{tikzcd}
    \]
    of rings or more generally of $\sE_1$-ring spectra.
    This induces a square of stable \inftyCats of module spectra
    \[
      \begin{tikzcd}
        \LMod_A \ar{r}\ar{d}
          & \LMod_{B} \ar{d}
        \\
        \LMod_{A'} \ar{r}
          & \LMod_{B'},
      \end{tikzcd}
    \]
    where the functors are each given by (derived) extension of scalars.
    Then the exchange transformation evaluated on any object $M' \in \LMod_{A'}$ is the canonical $B$-module morphism
    \[ B \otimes_A M' \to B' \otimes_{A'} M'. \]
    Since $\LMod_{A'}$ is generated under colimits by $A' \in \LMod_{A'}$, we see that this square satisfies base change if and only if the canonical morphism
    \[ B \otimes_A A' \to B' \]
    is invertible.
  \end{exam}

\subsection{Precartesian squares}

  \begin{defn}
    We say that a commutative square in $\Pres$
      \begin{equation*}
        \begin{tikzcd}
          \sA \arrow[r,"f^*"]\arrow[d,"p^*"] & \sB\arrow[d,"q^*"] \\
          \sA' \arrow[r, "g^*"]              & \sB'
        \end{tikzcd}
      \end{equation*}
    is \emph{precartesian} if the canonical functor
    \[ (p^*, f^*) : \sA \to \sA' \times_{\sB'} \sB \]
    is fully faithful.
    This is equivalent to the condition that the square of natural transformations
    \begin{equation}\label{eq:shfoin}
      \begin{tikzcd}
        \id \ar{r}{\eta_f}\ar{d}{\eta_p}
        & f_*f^* \ar{d}{\eta_q}
        \\
        p_*p^* \ar{r}{\eta_g}
        & f_*q_*q^*f^*
      \end{tikzcd}
    \end{equation}
    is cartesian.
  \end{defn}

  \begin{exam}
    In the situation of \examref{ex:rings}, the square is precartesian if and only if the original square of $\sE_1$-rings is cartesian.
  \end{exam}

  \begin{warn}
    A cartesian square of ordinary rings
    \[
      \begin{tikzcd}
        A \ar{r}\ar{d}
          & B \ar{d}
        \\
        A' \ar{r}
          & B'
      \end{tikzcd}
    \]
    is not necessarily cartesian as a square of $\sE_1$-rings, as the fibred product $A' \times_{B'} B$ may acquire negative homotopy groups when taken in the \inftyCat of $\sE_1$-rings.
    A sufficient condition is that $A' \to B'$ is surjective.
    For example, Milnor squares of rings give rise to cartesian squares of $\sE_1$-rings and hence to precartesian squares of \pstabs.
  \end{warn}

  \begin{exam}\label{exam:Ind cart}
    If
    \begin{equation*}
      \begin{tikzcd}
        \sA \ar{r}{f}\ar{d}{p}
        & \sB \ar{d}{q}
        \\
        \sA' \ar{r}{g}
        & \sB'
      \end{tikzcd}
    \end{equation*}
    is a cartesian square of small stable \inftyCats and exact functors, then the induced square
    \[ \begin{tikzcd}
      \Ind(\sA) \ar{r}{f^*}\ar{d}{p^*}
      & \Ind(\sB) \ar{d}{q^*}
      \\
      \Ind(\sA') \ar{r}{g^*}
      & \Ind(\sB')
    \end{tikzcd} \]
    is precartesian.
    (Indeed, the fully faithful functor
    \[ \sA \isoto \sA' \fibprod_{\sB'} \sB \hooklong \Ind(\sA') \fibprod_{\Ind(\sB')} \Ind(\sB) \]
    factors through the full subcategory of compact objects, since filtered colimits commute with finite limits of spaces.)
    However, it is cartesian only under additional hypotheses: for example, when $g_*$ is fully faithful \cite[Lem.~4]{HoyoisMilnor}.
  \end{exam}

\subsection{Milnor squares}

  \begin{lem}\label{lem:generate}
    Let $f^* : \sA \to \sB$ be a colimit-preserving functor between \pstabs.
    Then the following conditions are equivalent:
    \begin{thmlist}
      \item If $\sB_0 \subseteq \sB$ is a cocomplete stable subcategory containing the essential image of $f^* : \sA \to \sB$, then $\sB_0 = \sB$.
      \item The right orthogonal to the essential image of $f^* : \sA \to \sB$ is the zero subcategory.
      \item The right adjoint $f_* : \sB \to \sA$ is conservative.
    \end{thmlist}
  \end{lem}
  \begin{proof}
    See e.g. \cite[Lemma~7.6]{MNNNilpotence} for the equivalence of the first two conditions.
    The second and third are equivalent by adjunction.
  \end{proof}

  \begin{defn}
    If $f^* : \sA \to \sB$ satisfies the equivalent conditions of \lemref{lem:generate}, then we say \emph{$f^*$ generates $\sB$ under colimits}.
  \end{defn}

  \begin{defn}\label{defn:Milnor square}
    Let $\Delta$ be a commutative square in $\Pres$ of the form
      \begin{equation*}
        \begin{tikzcd}
          \sA \arrow[r,"f^*"]\arrow[d,"p^*"] & \sB\arrow[d,"q^*"] \\
          \sA' \arrow[r, "g^*"]              & \sB'.
        \end{tikzcd}
      \end{equation*}
    We say $\Delta$ is a \emph{Milnor square} if it is precartesian and each of the functors $f^*$, $g^*$, $p^*$ and $q^*$ is compact and generates its codomain under colimits.
  \end{defn}

  \begin{lem}\label{lem:jbnoauq}
    Suppose given a Milnor square $\Delta$ of \pstabs as above.
    If the functor $p^*$ is a localization, i.e., its right adjoint $p_*$ is fully faithful, then $\Delta$ satisfies base change.
  \end{lem}
  \begin{proof}
    The exchange transformation is by definition the composite
    \[
      f^*p_*
      \xrightarrow{\eta_q} q_*q^*f^*p_*
      \simeq q_* g^*p^* p_*
      \xrightarrow{\varepsilon_p} q_*g^*.
    \]
    The counit $\varepsilon_p$ is invertible by assumption, so it will suffice to show that the first arrow induced by the unit $\eta_q$ is invertible.
    Since $\Delta$ is precartesian, we have by precomposition of \eqref{eq:shfoin} with $p_*$ the cartesian square
    \[ \begin{tikzcd}
      p_* \ar{r}{\eta_f}\ar{d}{\eta_p}
      & f_*f^*p_* \ar{d}{\eta_q}
      \\
      p_*p^*p_* \ar{r}{\eta_g}
      & f_*q_*q^*f^*p_*.
    \end{tikzcd} \]
    Since $f_*$ is conservative (\lemref{lem:generate}), it will suffice to show that the right-hand arrow is invertible.
    The left-hand arrow in the square is invertible by the assumption that $\varepsilon_p$ is invertible and the adjunction identities.
    By stability, the claim follows.
  \end{proof}

\subsection{Pro-Milnor squares}

  We now define a pro-version of \defref{defn:Milnor square}.

  For a presentable $\infty$-category $\sC$, we let $\Pro(\sC)$ denote the \inftyCat of pro-objects in $\sC$ as in \cite[A.8.1]{SAG-20180204}.
  Any pro-object $X \in \Pro(\sC)$ can be represented (non-uniquely) by a cofiltered system $\{X_i\}_{i\in \sI}$, i.e., a diagram $\sI \to \sC$ from a cofiltered \inftyCat $\sI$.
  If $X \in \Pro(\sC)$ and $Y \in \Pro(\sC)$ are represented by cofiltered systems $\{X_i\}_i$ and $\{Y_j\}_j$, respectively, then the mapping space can be computed by the formula
  \[ \Maps(X, Y) \simeq \lim_j \colim_i \Maps_\sC(X_i, Y_j), \]
  see \cite[Rem.~A.8.1.5]{SAG-20180204}.
  From this one sees that the functor $\sC \to \Pro(\sC)$ sending an object $C \in \sC$ to the constant pro-system $\{C\}$ is fully faithful.
  It also implies that for any cofiltered \inftyCat $\sI$, the functor of ``passage to pro-objects''
  \[ \Fun(\sI, \sC) \to \Pro(\sC) \]
  commutes with finite limits and colimits, since filtered colimits of spaces commute with finite limits.

  \begin{defn}\label{defn:pro Milnor}
    A commutative square $\Delta$ in $\Pro(\Presc)$ is \emph{pro-precartesian}, \emph{pro-Milnor}, or \emph{satisfies pro-base change} if it can be represented by a cofiltered system $\{\Delta_n\}_n$ of commutative squares
    \[ \begin{tikzcd}
      \sA_n \ar{r}{f^*_n}\ar{d}{p^*_n}
        & \sB_n \ar{d}{q^*_n}
      \\
      \sA'_n \ar{r}{g^*_n}
        & \sB'_n
    \end{tikzcd} \]
    such that every $\Delta_n$ has the respective property.
  \end{defn}


\section{Quasi-coherent sheaves on algebraic stacks}
\label{sec:qcoh}

In this section, we give many examples of Milnor and pro-Milnor squares coming from squares of (derived) schemes and stacks.

\subsection{Nisnevich, cdh, and Milnor squares of stacks}

  \begin{defn}\label{defn:Nis}
    A \emph{Nisnevich square} of derived algebraic stacks is a cartesian square
    \begin{equation}\label{eq:Nis}
      \begin{tikzcd}
        \sU' \ar{r}\ar{d}
          & \sX' \ar{d}{f}
        \\
        \sU \ar{r}{j}
          & \sX,
      \end{tikzcd}
    \end{equation}
    where $j$ is a quasi-compact open immersion and $f$ is a representable étale morphism of finite presentation, inducing an isomorphism $f^{-1}(\sZ) \to \sZ$ for some closed immersion $i : \sZ \to \sX$ complementary to $j$.
    An \emph{affine Nisnevich square} is a Nisnevich square where $f$ is affine.
  \end{defn}

  \begin{defn}
    A \emph{Milnor square} of algebraic stacks is a commutative square
    \begin{equation}\label{eq:Milnor square of stacks}
      \begin{tikzcd}
      \sZ' \ar{r}{i'}\ar{d}{g}
      & \sX' \ar{d}{f}
      \\
      \sZ \ar{r}{i}
      & \sX
      \end{tikzcd}
    \end{equation}
    which is cartesian and cocartesian, where $f$ is an affine morphism and $i$ is a closed immersion with quasi-compact open complement.
  \end{defn}

  \begin{defn}\label{defn:cdh square}
    A \emph{proper cdh square} (or \emph{abstract blow-up square}) of derived algebraic stacks is a commutative square
    \begin{equation}\label{eq:abstract blow-up square}
      \begin{tikzcd}
        \sZ' \ar{r}{i'}\ar{d}{g}
          & \sX' \ar{d}{f}
        \\
        \sZ \ar{r}{i}
          & \sX
      \end{tikzcd}
    \end{equation}
    satisfying the following properties:
    \begin{defnlist}
      \item\label{item:paofpoql}
      The square is cartesian on classical truncations, i.e. $\sZ' \to \sZ \fibprod_\sX \sX'$ induces an isomorphism $\sZ'_\cl \simeq (\sZ \fibprod_\sX \sX')_\cl$.

      \item
      The morphism $f$ is representable and proper, and $i$ is a closed immersion with quasi-compact open complement.

      \item
      The induced map $f_\sU : \sU' \to \sU$ is invertible, where $\sU$ (resp. $\sU'$) is the open complement of $\sZ$ in $\sX$ (resp. of $\sZ'$ in $\sX'$).
    \end{defnlist}
    A \emph{finite cdh square} is a proper cdh square as above where the morphism $f$ is finite.
    The class of \emph{cdh squares} is the union of Nisnevich squares and proper cdh squares.
  \end{defn}

  \begin{exam}\label{exam:pinching}
    Given a diagram of algebraic spaces or stacks
    \[ \sX_0 \xleftarrow{i_0} \sX_{01} \xrightarrow{f} \sX_1, \]
    where $i_0$ is a closed immersion with quasi-compact open complement and $f$ is a finite morphism, the operation of forming the pushout $\sX = \sX_0 \fibcoprod_{\sX_{01}} \sX_1$ is often called ``pinching''.
    The square
    \[
      \begin{tikzcd}
        \sX_{01} \ar{r}{i_0}\ar{d}{f}
          & \sX_0 \ar{d}{f'}
        \\
        \sX_1 \ar{r}
          & \sX
      \end{tikzcd}
    \]
    is a Milnor square by \cite[Thm.~A.4]{RydhCompact}.
    See also \cite[Prop.~1]{Raoult} for the case of (separated noetherian) algebraic spaces.
    Note that, by \cite[Thm.~A.4]{RydhCompact}, the square is also almost a finite cdh square, except that $f'$ is only integral but not necessarily of finite type.
    If all stacks involved are of finite type over a noetherian base, then $f'$ is of finite type and hence finite.
  \end{exam}

  \begin{exam}
    In \examref{exam:pinching}, if the morphism $f$ is only affine, then the square is called a \emph{Ferrand pushout} after \cite{Ferrand}.
    See \cite[Thm.~6.1]{ArtinAlg} and \cite{ahhr} for the theory of Ferrand pushouts in the setting of algebraic spaces and stacks, respectively.
  \end{exam}

\subsection{Proto-excision statements}

  In this subsection we begin studying the behaviour of the functor $\sX \mapsto \Qcoh(\sX)$ with respect to various classes of squares of stacks.

  \begin{rem}\label{rem:f^* compact}
    Let $\sX$ and $\sY$ be derived algebraic stacks.
    For any representable morphism $f : \sX \to \sY$, the functor $f^* : \Qcoh(\sY) \to \Qcoh(\sX)$ is compact (see e.g. \cite[Prop.~A.1.5]{HLP-h}).
    If $f$ is affine (or quasi-affine), then $f^*$ generates under colimits.
  \end{rem}

  \begin{exam}[Base change]\label{exam:base change}
    Suppose given a commutative square of derived algebraic stacks
    \[ \begin{tikzcd}
      \sX' \ar{r}{g}\ar{d}{q}
      & \sY' \ar{d}{p}
      \\
      \sX \ar{r}{f}
      & \sY
    \end{tikzcd} \]
    where $p$ is representable.
    If the square is homotopy cartesian, then the induced square
    \[ \begin{tikzcd}
      \Qcoh(\sY) \ar{r}{f^*}\ar{d}{p^*}
        & \Qcoh(\sX) \ar{d}{q^*}
      \\
      \Qcoh(\sY') \ar{r}{g^*}
        & \Qcoh(\sX')
    \end{tikzcd} \]
    satisfies base change by \cite[Cor.~3.4.2.2]{SAG-20180204}.
    In fact, this remains valid for non-algebraic stacks (i.e., for arbitrary derived prestacks, see \cite[Prop.~A.1.5]{HLP-h}).
    Note moreover that this condition is both sufficient and necessary for base change.
  \end{exam}

  Affine Nisnevich squares give rise to categorical Milnor squares.
  More generally, we have:

  \begin{thm}[Excision]\label{thm:Qcoh excision}
    Let $f : \sX' \to \sX$ be a representable morphism of derived algebraic stacks.
    Suppose there exists a closed immersion $i : \sZ \to \sX$ with quasi-compact open complement such that $f$ is an isomorphism infinitely near $\sZ$, i.e., the induced morphism $\sX'^\wedge_{f^{-1}(\sZ)} \to \sX^\wedge_{\sZ}$ is invertible.
    Then the induced square
    \[ \begin{tikzcd}
      \Qcoh(\sX) \ar{r}{j^*}\ar{d}{f^*}
        & \Qcoh(\sX\setminus\sZ) \ar{d}
      \\
      \Qcoh(\sX') \ar{r}
        & \Qcoh(\sX'\setminus f^{-1}(\sZ))
    \end{tikzcd} \]
    is a cartesian square in $\Presc$ satisfying base change.
    If $f$ is affine, then it is moreover a Milnor square.
  \end{thm}
  \begin{proof}
    Since open immersions are representable, the base change formula holds by \examref{exam:base change} applied to the homotopy cartesian square
    \[
      \begin{tikzcd}
        \sX' \setminus f^{-1}(\sZ) \ar{r}{j'}\ar{d}{f'}
        & \sX' \ar{d}{f}
        \\
        \sX \setminus \sZ \ar{r}{j}
        & \sX.
      \end{tikzcd}
    \]
    By \remref{rem:f^* compact} it remains to show the cartesianness.
    For this, consider the adjunction
    \[ \Qcoh(\sX) \rightleftarrows \Qcoh(\sX\setminus\sZ) \fibprod_{\Qcoh(\sX'\setminus f^{-1}(\sZ))} \Qcoh(\sX'). \]
    It will suffice to show that the unit and counit are invertible.
    We will make use of the canonical exact triangle of endofunctors of $\D(\sX)$
    \begin{equation}\label{eq:sp0bf01b}
      \hat{i}_\sharp \hat{i}^* \to \id \to j_*j^*
    \end{equation}
    where $\hat{i} : \sX^\wedge_\sZ \hook \sX$ is the inclusion and $\hat{i}_\sharp$ is left adjoint to $\hat{i}^*$; see \cite[7.1]{GR-DGIndSch} or \cite[Thm.~2.2.3]{HLP-h}.

    The unit evaluated on $\sF \in \D(\sX)$ is the canonical morphism
    \begin{equation*}
      \sF \to f_*f^*(\sF) \fibprod_{g_*g^*(\sF)} j_*j^*(\sF),
    \end{equation*}
    where $g : \sX' \setminus f^{-1}(\sZ) \to \sX$.
    Using \eqref{eq:sp0bf01b}, it will suffice to show that it is invertible after applying either $j^*$ or $\hat{i}^*$.
    By the base change formula for $j_*$ and $f_*$ (\examref{exam:base change}), the map becomes the identity of $j^*(\sF)$ in the former case and the unit $\hat{i}^*(\sF) \to \hat{f}_*\hat{f}^*\hat{i}^*(\sF)$ in the latter, where $\hat{f} : \sX'^\wedge_{f^{-1}(\sZ)} \to \sX^\wedge_{\sZ}$ is the base change (which is invertible by assumption).
    
    For the counit, let $\sF_\sU \in \D(\sX\setminus\sZ)$, $\sF_{\sX'} \in \D(\sX')$, and $\sF_{\sU'} \in \D(\sX'\setminus f^{-1}(\sZ))$ such that there are isomorphisms $f'^*(\sF_\sU) \simeq \sF_{\sU'} \simeq j'^*(\sF_{\sX'})$.
    It will suffice to show that the canoincal morphisms
    \begin{align*}
      j^*(j_*\sF_\sU \fibprod_{g_*\sF_{\sU'}} f_*\sF_{\sX'})
      &\to \sF_{\sU}\\
      f^*(j_*\sF_\sU \fibprod_{g_*\sF_{\sU'}} f_*\sF_{\sX'})
      &\to \sF_{\sX'}
    \end{align*}
    are invertible.
    The first is straightforward using base change, and the second can be checked using \eqref{eq:sp0bf01b} again (applied to $\sX'$ and $\sZ'$ this time instead of $\sX$ and $\sZ$).
  \end{proof}

  Specializing to étale neighbourhoods, we get:

  \begin{cor}[Étale excision]\label{cor:Qcoh Nis}
    Suppose given a Nisnevich square of derived algebraic stacks of the form \eqref{eq:Nis}.
    Then the induced square
    \[ \begin{tikzcd}
      \Qcoh(\sX) \ar{r}{j^*}\ar{d}{f^*}
      & \Qcoh(\sU) \ar{d}
      \\
      \Qcoh(\sX') \ar{r}
      & \Qcoh(\sU')
      \end{tikzcd} \]
    is cartesian and satisfies base change.
    In particular, if \eqref{eq:Nis} is an affine Nisnevich square, then the above is a Milnor square satisfying base change.
  \end{cor}

  In contrast with the case of Nisnevich squares, Milnor squares of stacks are not generally homotopy cartesian.
  Therefore the induced square of stable \inftyCats does not usually satisfy base change.
  However, it will at least be precartesian in the following class of examples:

  \begin{thm}\label{thm:Qcoh Ferrand}
    Suppose given a Ferrand pushout square of derived algebraic stacks
    \begin{equation}\label{eq:pinapih}
      \begin{tikzcd}
        \sZ' \ar{d}{g}\ar{r}{i'}
        & \sX' \ar{d}{f}
        \\
        \sZ \ar{r}{i}
        & \sX.
      \end{tikzcd}
    \end{equation}
    That is, \eqref{eq:pinapih} is cocartesian, $i'$ is a closed immersion, and $g$ is affine.
    Then the induced square in $\Presc$
    \[
      \begin{tikzcd}
        \Qcoh(\sX) \ar{r}{i^*}\ar{d}{f^*}
          & \Qcoh(\sZ)\ar{d}{g^*}
        \\
        \Qcoh(\sX') \ar{r}{i'^*}
          & \Qcoh(\sZ')
      \end{tikzcd}
    \]
    is a Milnor square.
    It satisfies base change if and only if \eqref{eq:pinapih} is homotopy cartesian.
  \end{thm}
  \begin{proof}
    By \remref{rem:f^* compact} it is enough to show precartesianness.
    By fpqc descent we immediately reduce to the case where $\sX$ is affine, hence so are $\sX'$, $\sZ$ and $\sZ'$.
    Consider the corresponding cartesian square of \scrs
    \[ \begin{tikzcd}
      \sO_\sX \ar{r}\ar{d}
      & \sO_\sZ \ar{d}
      \\
      \sO_{\sX'} \ar{r}
      & \sO_{\sZ'}.
    \end{tikzcd} \]
    Since $i$ is a closed immersion, the homomorphism of \scrs $\sO_\sX \to \sO_\sZ$ is surjective on $\pi_0$, and thus the square is also cartesian on underlying (nonconnective) $\sE_1$-rings.
    Hence the claim follows from \cite[Thm.~16.2.0.2]{SAG-20180204}.
    The last part follows from \examref{exam:base change}.
  \end{proof}

\subsection{Formal completions of squares}

  \begin{constr}\label{constr:formal completion of square}
    Suppose given a commutative square of derived algebraic stacks
    \begin{equation*}
      \begin{tikzcd}
        \sZ' \ar{r}{i'}\ar{d}{g}
          & \sX' \ar{d}{f}
        \\
        \sZ \ar{r}{i}
          & \sX
      \end{tikzcd}
    \end{equation*}
   where $i$ and $i'$ are closed immersions with quasi-compact open complements.
    Formally completing the horizontal arrows (see \ssecref{ssec:formal}) gives rise to a commutative square
    \begin{equation}\label{eq:formal cdh square}
      \begin{tikzcd}
        \sX'^\wedge_{\sZ'} \ar[hookrightarrow]{r}\ar{d}{f^\wedge}
          & \sX' \ar{d}{f}
        \\
        \sX^{\wedge}_{\sZ} \ar[hookrightarrow]{r}
          & \sX.
      \end{tikzcd}
    \end{equation}
    If the original square is cartesian on underlying classical stacks as in \defref{defn:cdh square}\ref{item:paofpoql}, then the formally completed square is homotopy cartesian (\remref{rem:formal base change}).
    In particular, this holds for cdh squares and Milnor squares.
  \end{constr}

  We now show that the square \eqref{eq:formal cdh square} is also \emph{co}cartesian when the original square is a finite cdh square or a Milnor square (under mild finiteness assumptions).

  \begin{lem}\label{lem:finite cdh gives formal Milnor square}
    Suppose given a finite cdh square of \evcoconn noetherian derived algebraic stacks of the form \eqref{eq:abstract blow-up square}.
    Then the formally completed square \eqref{eq:formal cdh square} is cocartesian.
  \end{lem}
  \begin{proof}
    Since $\sX$ is algebraic, it will suffice to show that the square is cocartesian after smooth base change to any affine derived scheme.
    Since the latter operation preserves cocartesian squares, we may assume $\sX$ is affine.
    Let $\sX = \Spec(A)$, $\sX' = \Spec(A')$, and choose $f_1,\ldots,f_m \in \pi_0(A)$ such that $\sZ_\cl = \Spec(\pi_0(A)/(f_1,\ldots,f_m))$. Let $f'_1,\ldots,f'_m \in \pi_0(A')$
    denote their respective images.
    By \cite[7.2.4.31]{HA-20170918}, $A'$ is almost of finite presentation over $A$.
    By \examref{exam:oizchvnawl} it will suffice to show that the commutative squares
    \begin{equation*}
      \begin{tikzcd}
        A \ar{r}\ar{d}
          & A\modmod(f_1^n,\ldots,f_m^n)\ar{d}
        \\
        A' \ar{r}
          & A'\modmod({f'_1}^n,\ldots,{f'_m}^n)
      \end{tikzcd}
    \end{equation*}
    induce a cartesian square of pro-\scrs as $n$ varies.
    For this it is enough that the underlying square of pro-spectra is cartesian, or equivalently that the induced morphism $F \to F'$ on horizontal homotopy fibres induces isomorphisms of pro-abelian groups $\pi_k(F) \to \pi_k(F')$ for all $k\ge0$ (since the \scrs in the square are all bounded above).
    By the five lemma it suffices to show that the right- and left-hand vertical arrows are invertible in the commutative diagram of exact sequences
      \begin{equation*}
        \begin{tikzcd}
        0 \ar{r}
          & \Coker(\phi_{k+1})\ar{r}\ar{d}
          & \pi_k(F)\ar{r}\ar{d}
          & \Ker(\phi_{k})\ar{r}\ar{d}
          & 0
        \\
        0 \ar{r}
          & \Coker(\phi'_{k+1})\ar{r}
          & \pi_k(F')\ar{r}
          & \Ker(\phi'_{k})\ar{r}
          & 0
        \end{tikzcd}
      \end{equation*}
    where $\phi_k$ and $\phi'_k$ denote the induced morphisms $\{\pi_{k}(A)\} \to \{\pi_{k}(A\modmod(f_1^n,\ldots,f_m^n))\}_n$ and $\{\pi_{k}(A') \to \pi_{k}(A'\modmod({f'_1}^n,\ldots,{f'_m}^n))\}_n$, respectively.
    By \cite[Lem.~8.4.4.5]{SAG-20180204} or \cite[Lem.~4.10]{Kerz_2017} the canonical morphisms
      \begin{align*}
        \{\pi_k(A\modmod(f_1^n,\ldots,f_m^n))\}_n &\to \{\pi_k(A)/(f_1^n,\ldots,f_m^n)\}_n,\\
        \{\pi_k(A'\modmod({f'_1}^n,\ldots,{f'_m}^n))\}_n &\to \{\pi_k(A')/({f'_1}^n,\ldots,{f'_m}^n)\}_n
      \end{align*}
    are invertible for all $k\ge 0$.
    In particular, the morphisms $\phi_k$ and $\phi'_k$ are all surjective, with kernels $\{(f_1^n,\ldots,f_m^n) \cdot \pi_{k}(A)\}_n$ and $\{({f'_1}^n,\ldots,{f'_m}^n) \cdot \pi_{k}(A')\}_n$, respectively.
    Since $A'$ is noetherian and the morphism $f : \Spec(A') \to \Spec(A)$ is finite, the homotopy groups $\pi_k(A')$ are finitely generated over $\pi_0(A')$ and hence over $\pi_0(A)$.
    Now it is straightforward to check, using the assumption that $f: \Spec(A') \to\Spec(A)$ induces an isomorphism away from $Z$, that the canonical morphisms
      \begin{equation*}
        \{(f_1^n,\ldots,f_m^n)\cdot \pi_{k}(A)\}_n \to \{({f'_1}^n,\ldots,{f'_m}^n) \cdot\pi_{k}(A')\}_n,
      \end{equation*}
    are invertible.
    Indeed for each $i$, the morphism $\pi_{k}(A)[1/f_i] \to \pi_{k}(A')[1/f_i]$ is invertible by assumption. 
    This implies in particular that there exists an $m$ such that for all $m \geq n$, $(f_1^n,\ldots,f_m^n)$ annihilates the kernel (which is finitely generated by the noetherian assumption) and $(f_1^n,\ldots,f_m^n)\cdot \pi_{k}(A)\ \to \{({f'_1}^n,\ldots,{f'_m}^n)\pi_{k}(A')$ is surjective (since $\pi_k(A')$ is finitely generated over $\pi_0(A)$). 
    By the former claim and Artin-Rees lemma we can also ensure (by choosing a larger $m$ if necessary) that the morphism is injective.
    The claim follows.
  \end{proof}

  \begin{cor}\label{cor:nil formal Milnor square}
    Let $\sX$ be a \evcoconn noetherian derived algebraic stack.
    For any closed immersion $i : \sZ \to \sX$ with $0$-truncated quasi-compact open complement $\sX\setminus\sZ$, the square
    \[ \begin{tikzcd}
      (\sX_\cl)^\wedge_{\sZ_\cl} \ar[hookrightarrow]{r}\ar{d}
        & \sX_\cl \ar{d}
      \\
      \sX^{\wedge}_{\sZ} \ar[hookrightarrow]{r}
        & \sX
    \end{tikzcd} \]
    is cocartesian, where the right-hand vertical arrow is the inclusion of the classical truncation.
  \end{cor}

  \begin{lem}\label{lem:Milnor gives formal Milnor square}
    Suppose given a Milnor square of noetherian algebraic stacks of the form \eqref{eq:Milnor square of stacks}.
    Then the formally completed square \eqref{eq:formal cdh square} is cocartesian.
  \end{lem}
  \begin{proof}
    As in the proof of \lemref{lem:finite cdh gives formal Milnor square}, we may assume that $\sX$ (and hence $\sX'$) is affine.
    Let $\sX = \Spec(A)$, $\sX' = \Spec(A')$, and $f_1,\ldots,f_m \in A$ such that $\sZ = \Spec(A/(f_1,\ldots,f_m))$ and $\sZ' = \Spec(A'/(f_1,\ldots,f_m)A')$.
    By assumption, the square
    \[
      \begin{tikzcd}
        A \ar{r}\ar{d}
          & A/(f_1^n,\ldots,f_m^n) \ar{d}
        \\
        A' \ar{r}
          & A'/(f_1^n,\ldots,f_m^n)
      \end{tikzcd}
    \]
    is cartesian for $n=1$, i.e., $A \to A'$ sends the ideal $(f_1,\ldots,f_m)$ isomorphically onto $(f_1,\ldots,f_m)A'$.
    Then the same holds for all $n>0$ and hence by \examref{exam:oizchvnawl} it follows that the square
    \[
      \begin{tikzcd}
        \Spec(A')^\wedge \ar[hookrightarrow]{r}\ar{d}
          & \Spec(A') \ar{d}
        \\
        \Spec(A)^{\wedge} \ar[hookrightarrow]{r}
          & \Spec(A)
      \end{tikzcd}
    \]
    is cocartesian.
  \end{proof}

\subsection{Formal Milnor and finite excision}

  In this subsection we prove the following result, which shows that if we pass to formal completions in \thmref{thm:Qcoh Ferrand}, then we get a pro-Milnor square.

  \begin{thm}\label{thm:Qcohform Ferrand}
    Suppose given a square of \evcoconn noetherian derived algebraic stacks of the form
    \[
      \begin{tikzcd}
        \sZ' \ar{d}{g}\ar{r}{i'}
        & \sX' \ar{d}{f}
        \\
        \sZ \ar{r}{i}
        & \sX.
      \end{tikzcd}
    \]
    Assume that $i$ is a closed immersion, $f$ is affine, the square is cartesian on classical truncations, and the formally completed square \eqref{eq:formal cdh square} is cocartesian.
    Then the induced commutative square in $\Pro(\Presc)$
    \begin{equation}\label{eq:oasvnlq}
      \begin{tikzcd}
        \{\Qcoh(\sX)\} \ar{r}{\hat{i}^*}\ar{d}{f^*}
          & \Qcohform(\sX^\wedge) \ar{d}{(f^\wedge)^*}
        \\
        \{\Qcoh(\sX')\} \ar{r}{\hat{i'}^*}
          & \Qcohform(\sX'^\wedge).
      \end{tikzcd}
    \end{equation}
    is a pro-Milnor square satisfying pro-base change.
  \end{thm}

  \begin{proof}
    By \constrref{constr:formal completion of square} and the assumption, the square
    \begin{equation*}
      \begin{tikzcd}
        \sX'^\wedge_{\sZ'} \ar[hookrightarrow]{r}\ar{d}{f^\wedge}
          & \sX' \ar{d}{f}
        \\
        \sX^{\wedge}_{\sZ} \ar[hookrightarrow]{r}
          & \sX
      \end{tikzcd}
    \end{equation*}
    is both cartesian and cocartesian.
    Thus by \remref{rem:formal colim} it can be represented by either of the following filtered systems
    \begin{equation*}
      \begin{tikzcd}
        \widetilde{\sZ}' \ar[hookrightarrow]{r}\ar{d}
          & \sX' \ar{d}{f}
        \\
        \widetilde{\sZ} \ar[hookrightarrow]{r}
          & \sX,
      \end{tikzcd}
      \qquad
      \begin{tikzcd}
        \widetilde{\sZ}' \ar[hookrightarrow]{r}\ar{d}
          & \sX' \ar{d}{f}
        \\
        \widetilde{\sZ} \ar[hookrightarrow]{r}
          & \widetilde{\sZ}\fibcoprod_{\widetilde{\sZ}'} \sX',
      \end{tikzcd}
    \end{equation*}
    indexed by $\widetilde{\sZ}$ as in \remref{rem:formal colim}, where $\widetilde{\sZ}' = \widetilde{\sZ}\fibprod_\sX\sX'$.
    The left-hand squares are levelwise cartesian and the right-hand squares are levelwise cocartesian.
    Thus by \examref{exam:base change} and \thmref{thm:Qcoh Ferrand} the induced square \eqref{eq:oasvnlq} can be represented alternatively by squares that satisfy levelwise base change or are levelwise precartesian.
    In particular, it is a pro-Milnor square.
  \end{proof}

  Combining this with Lemmas~\ref{lem:finite cdh gives formal Milnor square} and \ref{lem:Milnor gives formal Milnor square} yields:

  \begin{cor}[Formal Milnor excision]\label{cor:Qcohform Milnor}
    Suppose given a Milnor square of noetherian algebraic stacks of the form \eqref{eq:Milnor square of stacks}.
    Then the induced commutative square in $\Pro(\Presc)$
    \begin{equation*}
      \begin{tikzcd}
        \{\Qcoh(\sX)\} \ar{r}{\hat{i}^*}\ar{d}{f^*}
          & \Qcohform(\sX^\wedge) \ar{d}{(f^\wedge)^*}
        \\
        \{\Qcoh(\sX')\} \ar{r}{\hat{i'}^*}
          & \Qcohform(\sX'^\wedge).
      \end{tikzcd}
    \end{equation*}
    is a pro-Milnor square satisfying pro-base change.
  \end{cor}

  \begin{cor}[Formal finite excision]\label{cor:Qcohform finite cdh}
    Suppose given a finite cdh square of \evcoconn noetherian derived algebraic stacks of the form \eqref{eq:abstract blow-up square}.
    Then the induced commutative square in $\Pro(\Presc)$
    \begin{equation*}
      \begin{tikzcd}
        \{\Qcoh(\sX)\}\ar{r}\ar{d}
          & \Qcohform(\sX^\wedge_{\sZ})\ar{d}
        \\
        \{\Qcoh(\sX')\}\ar{r}
          & \Qcohform(\sX'^\wedge_{\sZ'})
      \end{tikzcd}
    \end{equation*}
    is a pro-Milnor square satisfying pro-base change.
  \end{cor}

  \begin{cor}[Formal nil-excision]\label{cor:Qcohform nil}
    Let $\sX$ be a \evcoconn noetherian derived algebraic stack with classical truncation $\sX_\cl$.
    Then for any closed immersion $\sZ \hook \sX$ with $0$-truncated quasi-compact open complement $\sX\setminus\sZ$, the induced commutative square in $\Pro(\Presc)$
      \begin{equation*}
        \begin{tikzcd}
          \{\Qcoh(\sX)\}\ar{r}\ar{d}
            & \Qcohform(\sX^\wedge_\sZ)\ar{d}
          \\
          \{\Qcoh(\sX_\cl)\}\ar{r}
            & \Qcohform((\sX_\cl)^\wedge_{\sZ_\cl})
        \end{tikzcd}
      \end{equation*}
    is a pro-Milnor square satisfying pro-base change.
  \end{cor}


\section{Categorical Milnor excision}
\label{sec:localizing_Milnor}

In this section we prove categorical Milnor excision (\thmref{thm:intro/cat}).
We will first prove a characterization of Milnor squares in terms of the ``$\odot$-construction'' (see \thmref{thm:odot}).

\subsection{Adjustments of squares}

  Of particular interest for us will be the behaviour of Milnor squares as $\sB'$ is allowed to vary (especially among categories of non-geometric origin).
  For convenience, we refer to the operation of extending the square $\Delta$ along a colimit-preserving functor $a^* : \sB' \to \sB''$ as ``adjustment''.

  \begin{defn}
    Let $\Delta$ be a commutative square in $\Pres$ of the form
    \begin{equation}\label{eq:square1}
      \begin{tikzcd}
        \sA \arrow[r,"f^*"]\arrow[d,"p^*"] & \sB\arrow[d,"q^*"] \\
        \sA' \arrow[r, "g^*"]              & \sB'.
      \end{tikzcd}
    \end{equation}
    Any choice of another presentable \inftyCat $\sB''$ and a colimit-preserving functor $a^* : \sB' \to \sB''$ determines a new square $\Delta_a$ of the form
    \begin{equation*}
      \begin{tikzcd}
        \sA \arrow[r,"f^*"]\arrow[d,"p^*"] & \sB\arrow[d,"q_a^*"] \\
        \sA' \arrow[r, "g_a^*"]           & \sB'',
      \end{tikzcd}
    \end{equation*}
    where $q_a^* = a^* \circ q^*$ and $g_a^* = a^* \circ g^*$.
    We call $\Delta_a$ the \emph{adjustment of $\Delta$ by $a^*$}.
  \end{defn}

  \begin{lem}\label{lem:base change adjustment}
    Let $\Delta$ be a commutative square in $\Pres$ of the form \eqref{eq:square1}.
    Let $\Delta_a$ be the adjustment of $\Delta$ by a colimit-preserving functor $a^* : \sB' \to \sB''$.
    Then the exchange transformations for the respective squares fit into a canonical commutative square
    \[
      \begin{tikzcd}
        f^*p_* \ar{r}{\Ex_\Delta}\ar[equals]{d}
          & q_*g^* \ar{d}
        \\
        f^*p_* \ar[swap]{r}{\Ex_{\Delta_a}}
          & q_{a,*}g_a^*
      \end{tikzcd}
    \]
    where the vertical arrow is the morphism $q_*g^* \to q_*a_*a^*g^* \simeq q_{a,*}g_a^*$ induced by the unit $\eta_a$.
  \end{lem}
  \begin{proof}
    The diagram in question can be subdivided as follows
    \[
      \begin{tikzcd}
        f^*p_* \ar{r}{\eta_q}\ar[equals]{d}
          & q_*q^*f^*p_* \ar{r}{\sim}\ar{d}{\eta_a}
          & q_*g^*p^*p_* \ar{r}{\varepsilon_p}\ar{d}{\eta_a}
          & q_*g^*\ar{d}{\eta_a}\\
        f^*p_* \ar{r}{\eta_{aq}}
          & q_*a_*a^*q^*f^*p_* \ar{r}{\sim}
          & q_*a_*a^*g^*p^*p_* \ar{r}{\varepsilon_p}
          & q_*a_*a^*g^*
      \end{tikzcd}
    \]
    where each square is tautologically commutative.
  \end{proof}

  \begin{cor}\label{cor:base change adjustment}
    Let $\Delta$ be a commutative square of \pstabs and colimit-preserving functors of the form \eqref{eq:square1}.
    Suppose that $\Delta$ satisfies base change.
    For any \pstab $\sB''$ and colimit-preserving functor $a^* : \sB' \to \sB''$, consider the following conditions:
    \begin{thmlist}
      \item The functor $a^* : \sB' \to \sB''$ is fully faithful.
      \item The adjustment $\Delta_a$ satisfies base change.
    \end{thmlist}
    The implication (i) $\implies$ (ii) always holds, and the converse holds if $g^*$ and $q^*$ generate $\sB'$ under colimits.
  \end{cor}
  \begin{proof}
    Recall that the first condition is equivalent to invertibility of the unit $\eta_a : \id \to a_*a^*$.
    Therefore the claim follows immediately from \lemref{lem:base change adjustment}.
  \end{proof}

  \begin{lem}\label{lem:precartesian adjustment}
    Let $\Delta$ be a commutative square in $\Pres$ of the form \eqref{eq:square1}.
    If a colimit-preserving functor $a^* : \sB' \to \sB''$ induces monomorphisms on mapping spaces (e.g. it is fully faithful, or more generally if it is a monomorphism of \inftyCats), the following conditions are equivalent:
    \begin{thmlist}
      \item The square $\Delta$ is precartesian.
      \item The adjustment $\Delta_a$ is precartesian.
    \end{thmlist}
  \end{lem}
  \begin{proof}
    This follows from the following basic fact: given a diagram of spaces $X \to Z \gets Y$ and a map $f : Z \to Z'$, the induced map
    \[ X\fibprod_ZY \to X\fibprod_{Z'}Y \]
    is invertible when $Z \to Z'$ is a monomorphism.
  \end{proof}

\subsection{The \texorpdfstring{$\odot$}{odot}-construction}

  In this subsection we introduce the main tool in our analysis of categorical Milnor squares (see \constrref{constr:odot}).
  This material closely follows \cite{LandTamme}, where it is developed in the special case of \inftyCats of module spectra.

  Fix a commutative square $\Delta$ in $\Presc$ of the form
  \begin{equation*}
    \begin{tikzcd}
      \sA \arrow[r,"f^*"]\arrow[d,"p^*"] & \sB\arrow[d,"q^*"] \\
      \sA' \arrow[r, "g^*"]              & \sB'.
    \end{tikzcd}
  \end{equation*}
  We begin with the following preliminary construction.

  \begin{constr}
    Consider the \inftyCat $\sC$ of triples $(X', Y, \theta)$, where $X' \in \sA'$ and $Y \in \sB$ are objects and $\theta : g^*(X') \to q^*(Y)$ is a morphism in $\sB'$.
    A morphism $(X'_1, Y_1, \theta_1) \to (X'_2, Y_2, \theta_2)$ in $\sC$ is the data of a morphism $\alpha : X'_1 \to X'_2$ in $\sA'$, a morphism $\beta : Y_1 \to Y_2$ in $\sB$, and a commutative square
    \[
      \begin{tikzcd}
        g^*(X'_1) \ar{r}{g^*(\alpha)}\ar{d}{\theta_1}
          & g^*(X'_2) \ar{d}{\theta_2}
        \\
        q^*(Y_1) \ar{r}{q^*(\beta)}
          & q^*(Y_2)
      \end{tikzcd}
    \]
    in $\sB'$.
    More precisely, $\sC$ may be defined by the cartesian square
    \[
      \begin{tikzcd}
        \sC \arrow[d]\arrow[r]
          & \Fun(\Delta^1,\sB') \arrow[d]\\
        \sA' \times \sB \arrow[r]
          & \sB' \times \sB'.
      \end{tikzcd}
    \]
    This construction is called the \emph{lax pullback} or \emph{oriented fibre product}.
    Note that the usual fibre product can be identified with the full subcategory of $\sC$ spanned by objects $(X', Y, \theta)$ such that $\theta$ is an isomorphism.
  \end{constr}

  \begin{rem}\label{rem:C=<B,A'>}
    By construction, the \inftyCat $\sC$ is stable and presentable (see \cite[Lemma~8]{Tamme_2018}).
    The full subcategory of $\sC$ spanned by objects of the form $(X', 0, 0)$ is canonically identified with $\sA'$, and similarly $\sB$ is identified with the full subcategory spanned by objects of the form $(0, Y, 0)$.
    Every object $(X', Y, \theta) \in \sC$ fits functorially into a cofibre sequence
    \begin{equation}\label{eq:decomp C}
      (0, Y, 0) \to (X', Y, \theta) \to (X', 0, 0).
    \end{equation}
    In fact, there is a semi-orthogonal decomposition $\sC = \langle\sB,\sA'\rangle$, see \cite[Proposition~10]{Tamme_2018}.
  \end{rem}

  \begin{rem}\label{rem:d_*}
    Consider the functor $d_* : \sA \to \sC$ defined as the composite
    \[
      (p^*,f^*) : \sA \to \sA' \times_{\sB'} \sB \subseteq \sC.
    \]
    whose essential image is spanned by objects of the form $(p^*(X), f^*(X), \theta_\Delta) \in \sC$, where $X \in \sA$ and $\theta_\Delta : g^*p^*(X) \simeq q^*f^*(X)$ is the isomorphism determined by the commutative square $\Delta$.
    Note that $d_*$ is fully faithful if and only if the square $\Delta$ is precartesian.
    Since $d_*$ preserves colimits it admits a right adjoint\footnote{%
      The departure from our usual notation (\ref{notat:paisj}) is because of the semi-orthogonal decomposition we will see below (\remref{rem:C=<A,Q>}).
    } $d^! : \sC \to \sA$ whose value on any object $(X',Y,\theta)\in\sC$ can be computed by the cartesian square
    \[
      \begin{tikzcd}
        d^!(X', Y, \theta) \ar{rr}\ar{dd}
        & & f_*(Y) \ar{d}{\eta_q}
      \\
        & & f_*q_*q^*(Y) \ar[equals]{d}
      \\
        p_*(X') \ar{r}{\eta_g} & p_*g_*g^*(X') \ar{r}{\theta} & p_*g_*q^*(Y).
      \end{tikzcd}
    \]
  \end{rem}

  \begin{rem}
    For every object $X \in \sA$, the cofibre sequence \eqref{eq:decomp C} may be evaluated at the object $d_*(X) = (p^*(X),f^*(X),\theta_\Delta) \in \sC$ to get a cofibre sequence
    \[
      (0, f^*(X), 0) \to d_*(X) \to (p^*(X), 0, 0)
    \]
    with boundary map
    \[
      \partial : (p^*(X), 0, 0) \to \Sigma(0, f^*(X), 0) \simeq (0, \Sigma f^*(X), 0).
    \]
    As the construction is functorial in $X$, $\partial$ defines a natural transformation of functors $\sA \to \sC$, up to which the square
    \begin{equation}\label{eq:Delta_C}
      \begin{tikzcd}
        \sA \ar{rr}{f^*}\ar{d}{p^*}
          & & \sB \ar{d}{(0, \Sigma, 0)}
        \\
        \sA' \ar[swap]{rr}{(\id, 0, 0)}\ar[urr, phantom, "\rotatebox{30}{\(\Rightarrow \)}"]
          & & \sC.
      \end{tikzcd}
    \end{equation}
    commutes.
    Note that the adjustment of this square by the canonical colimit-preserving functor $c^* : \sC \to \sB'$, sending $(X', Y, \theta)$ to the fibre of $\theta$, is canonically identified with our original square $\Delta$.
  \end{rem}

  The $\odot$-construction is obtained by forcing the square \eqref{eq:Delta_C} to become commutative, which amounts to killing the essential image of $d_* : \sA \to \sC$.

  \begin{constr}[$\odot$-Construction]\label{constr:odot}
    We let $\sA' \odot_{\sA}^{\sB'} \sB$ denote the Verdier quotient of the \pstab $\sC$ by the essential image of the functor $d_* : \sA \to \sC$.
    We write $a^* : \sC \to \sA' \odot_{\sA}^{\sB'} \sB$ for the quotient functor.
    Thus we have a canonical commutative square $\Delta_0$
    \begin{equation}\label{eq:Delta_0}
      \begin{tikzcd}
        \sA \ar{rr}{f^*}\ar{d}{p^*}
          & & \sB \ar{d}{q_0^*}
        \\
        \sA' \ar{rr}{g_0^*}
          & & \sA' \odot_{\sA}^{\sB'} \sB,
      \end{tikzcd}
    \end{equation}
    where $g_0^* = a^*(\id, 0, 0)$ and $q_0^* = \Sigma a^* (0, \id, 0)$.
    Since the composite 
    \[
      \sA \xrightarrow{d_*} \sC \xrightarrow{c^*} \sB'
    \]
    is canonically null-homotopic, the colimit-preserving functor $c^* : \sC \to \sB'$ descends to a canonical colimit-preserving functor
    \[
      b^* : \sA' \odot_{\sA}^{\sB'} \sB \to \sB'.
    \]
    By construction, the adjustment of $\Delta_0$ by $b^*$ is canonically identified with $\Delta$.
  \end{constr}

  \begin{rem}\label{rem:a^* splitting}
    Upon application of the quotient functor $a^* : \sC \to \sA' \odot_{\sA}^{\sB'} \sB$, the cofibre sequences of \remref{rem:C=<B,A'>} induce cofibre sequences of the form
    \[
      \Omega q_0^*(Y) \to a^*(X', Y, \theta) \to g_0^*(X')
    \]
    for every object $(X', Y, \theta) \in \sC$.
  \end{rem}

  \begin{rem}\label{rem:C=<A,Q>}
    By construction, we have a semi-orthogonal decomposition $\sC = \langle \sA, \sA' \odot_{\sA}^{\sB'} \sB \rangle$.
    In particular there is a cofibre sequence
    \begin{equation*}
      d_*d^! \xrightarrow{\varepsilon_d} \id \xrightarrow{\eta_a} a_*a^*
    \end{equation*}
    of natural transformations $\sC \to \sC$.
  \end{rem}

  \begin{lem}\label{lem:q_0 and g_0 generation}\leavevmode
    \begin{thmlist}
      \item If $p^* : \sA \to \sA'$ generates its codomain under colimits, then so does $q_0^* : \sB \to \sA' \odot_{\sA}^{\sB'} \sB$.
      \item If $f^* : \sA \to \sB$ generates its codomain under colimits, then so does $g_0^* : \sA' \to \sA' \odot_{\sA}^{\sB'} \sB$.
      \item If $q^* : \sB \to \sB'$ generates its codomain under colimits, then so does $b^* : \sA' \odot_{\sA}^{\sB'} \sB \to \sB'$.
      \item If $p^*$, $f^*$ and $q^*f^* \simeq g^*p^*$ are compact, then so is $a^*$.
      \item If $q^*$ is compact, then so is $c^*$.
      \item If $f^*$, $p^*$, and $q^*$ are compact, then so is $b^*$.
      \item If $p^*$ and $f^*$ are compact, then so are $q_0^*$ and $g_0^*$.
    \end{thmlist}
  \end{lem}
  \begin{proof}\leavevmode
    \begin{defnlist}
      \item
      If the essential image of $p^*$ generates $\sA'$ under colimits, then it follows from \remref{rem:C=<B,A'>} that $\sC$ is generated under colimits by objects of the form $(p^*(X), 0, 0)$ and $(0, Y, 0)$ as $X$ and $Y$ vary among objects of $\sA$ and $\sB$, respectively.
      In particular, the Verdier quotient $\sA' \odot_{\sA}^{\sB'} \sB$ is generated under colimits by the images of such objects under the quotient functor $a^* : \sC \to \sA' \odot_{\sA}^{\sB'} \sB$, i.e., by objects of the forms $g_0^*p^*(X)$ and $q_0^*(Y)$.
      But the commutativity of $\Delta_0$ \eqref{eq:Delta_0} provides isomorphisms
      \[ g_0^*p^*(X) \simeq q_0^*f^*(X) \]
      in $\sA' \odot_{\sA}^{\sB'} \sB$ for every $X \in \sA$.

      \item
      Similar to (i).

      \item
      By assumption, $\sB'$ is generated under colimits by the essential image of $q^*\simeq b^*q_0^*$.
      Hence $b^*$ also generates $\sB'$ under colimits.

      \item
      Since $a^*$ generates under colimits, it will suffice to show that $a_*a^*$ preserves filtered colimits.
      Using the cofibre sequence of \remref{rem:C=<A,Q>}, we reduce to showing that $d^!$ preserves filtered colimits.
      This follows immediately from the description of \remref{rem:d_*} (in view of the assumptions).

      \item
      Note that the right adjoint $c_* : \sB' \to \sC$ is given by the formula $Y' \mapsto (0, \Sigma q_*(Y'), 0)$.

      \item
      Recall that $a_* : \sA' \odot_\sA^{\sB'} \sB \to \sC$ is conservative (as $a^*$ is essentially surjective) and colimit-preserving by (iii).
      Therefore $b_*$ preserves colimits if and only if $a_*b_* \simeq c_*$ does.
      This is true under the assumptions by (iv).

      \item
      By the definition of $q_0^*$ and $g_0^*$, this follows immediately from (iii).
    \end{defnlist}
  \end{proof}

\subsection{Characterization of Milnor squares}

  The following result shows that every Milnor square $\Delta$ satisfying base change is isomorphic to one coming from the $\odot$-construction.
  More precisely, it is an adjustment of the square $\Delta_0$ \eqref{eq:Delta_0} along the canonical functor $b^* : \sA' \odot_{\sA}^{\sB'} \sB \to \sB'$, which is an equivalence.

  \begin{thm}\label{thm:odot}
    Let $\Delta$ be a Milnor square of \pstabs of the form
    \[ \begin{tikzcd}
      \sA \arrow[r,"f^*"]\arrow[d,"p^*"] & \sB\arrow[d,"q^*"] \\
      \sA' \arrow[r, "g^*"]              & \sB'.
    \end{tikzcd} \]
    Then we have:
    \begin{thmlist}
      \item\label{item:nalnfiowa}
      The square $\Delta_0$ \eqref{eq:Delta_0} satisfies base change.

      \item
      If $\Delta$ satisfies base change, then the canonical functor
      \[ b^* : \sA' \odot_{\sA}^{\sB'} \sB \to \sB'\]
      is an equivalence.
      In particular, there is an isomorphism of squares $\Delta \simeq \Delta_0$.
    \end{thmlist}
  \end{thm}

  \begin{lem}\label{lem:phi_q0}
    Let $\Delta$ be a commutative square of \pstabs and colimit-preserving functors as above.
    Let $\phi_{q_0}$ denote the fibre of the unit $\eta_{q_0}$, and similarly for $\phi_q$.
    Then there is a canonical isomorphism $f^*f_* \phi_q \simeq \phi_{q_0}$ such that the diagram
    \[
      \begin{tikzcd}
        f^*f_* \phi_q \ar{r}{f^*f_* \mrm{can}}\ar[equals]{d}
          & f^*f_* \ar{d}{\varepsilon_f}
        \\
        \phi_{q_0} \ar{r}{\mrm{can}}
          & \id
      \end{tikzcd}
    \]
    commutes, where $\mrm{can}$ denotes the canonical inclusion of the respective fibre.
  \end{lem}
  \begin{proof}
    Precomposing the cofibre sequence of \remref{rem:C=<A,Q>} with the inclusion $(0, \id, 0) : \sB \to \sC$ and projecting the result back to $\sB$ yields the cofibre sequence
    \begin{equation*}
      f^* d^!(0, \id, 0) \to \id \xrightarrow{\eta_{q_0}} q_{0,*} q_0^*
    \end{equation*}
    of natural transformations $\sB \to \sB$.
    Now apply the description of $d^!$ given in \remref{rem:d_*}.
  \end{proof}

  \begin{lem}\label{lem:base change Delta_0}
    Let $\Delta$ be a commutative square of \pstabs and colimit-preserving functors as above.
    If $\Delta$ is precartesian, then the exchange transformation for $\Delta_0$ induces an isomorphism
    \[
      f^*p_*p^* \xrightarrow{\Ex_{\Delta_0} p^*} q_{0,*} g_0^* p^* \simeq q_{0,*} q_0^* f^*.
    \]
    In particular, if $\Delta$ is a Milnor square, then $\Delta_0$ satisfies base change.
  \end{lem}
  \begin{proof}
    Note that the second claim follows from the first in view of the fact that $p^*$ generates $\sA'$ under colimits and $p^*$ and $q_0^*$ are compact (\lemref{lem:q_0 and g_0 generation}).
    For the main claim, by construction of the exchange transformation, we have a commutative square
    \begin{equation}\label{eq:ldsjf}
      \begin{tikzcd}
        f^* \ar{r}{\eta_p}\ar[equals]{d}
          & f^* p_* p^* \ar{d}
        \\
        f^* \ar{r}{\eta_{q_0}}
          & q_{0,*} q_0^* f^*.
      \end{tikzcd}
    \end{equation}
    Taking fibres horizontally, we obtain by \lemref{lem:phi_q0} the following commutative diagram:
    \[
      \begin{tikzcd}
        f^*\phi_p \ar{r}{\mrm{can}}\ar{d}
          & f^* \ar{r}{\eta_p}\ar{d}{f^*\eta_f}
          & f^*p_*p^* \ar{dd}
        \\
        f^*f_*\phi_q f^* \ar{r}{\mrm{can}}\ar[equals]{d}
          & f^*f_*f^* \ar{d}{\varepsilon_f f^*}
          &
        \\
        \phi_{q_0} f^* \ar{r}{\mrm{can}}
          & f^* \ar{r}{\eta_{q_0}}
          & q_{0,*}q_0^*f^*
      \end{tikzcd}
    \]
    The upper vertical left-hand arrow is invertible as soon as $\Delta$ is precartesian, i.e., when the right-hand square below is cartesian:
    \[
      \begin{tikzcd}
        \phi_p \ar{d}\ar{r}{\mrm{can}}
          & \id \ar{d}{\eta_f}\ar{r}{\eta_p}
          & p_*p^* \ar{d}{\eta_g}
        \\
        f_* \phi_q f^* \ar{r}{\mrm{can}}
          & f_*f^* \ar{r}{\eta_q}
          & f_*q_*q^*f^* \simeq p_*g_*g^*p^*
      \end{tikzcd}
    \]
    Thus in that case we find that \eqref{eq:ldsjf} is cartesian.
  \end{proof}

  \begin{proof}[Proof of \thmref{thm:odot}]
    The first claim follows from \lemref{lem:base change Delta_0}.
    For the second, assume that $\Delta$ satisfies base change.
    Since $\Delta$ is the adjustment of $\Delta_0$ by $b^*$, and since $b^*$ is compact by \lemref{lem:q_0 and g_0 generation}, this implies by \corref{cor:base change adjustment} that $b^*$ is fully faithful.
    By \lemref{lem:q_0 and g_0 generation}, $b^*$ generates $\sB'$ under colimits.
    Since $\sA' \odot_{\sA}^{\sB'} \sB$ is cocomplete, it now follows that $b^*$ is an equivalence.
  \end{proof}

\subsection{Categorical Milnor excision}

  Let $E$ be a functor on the \inftyCat of small stable \inftyCats with values in a stable \inftyCat.
  Recall that $E$ is called a \emph{localizing invariant} if it sends exact sequences to exact triangles.
  See e.g. \cite{BlumbergGepnerTabuada}, except that we do not assume that $E$ commutes with filtered colimits.

  \begin{exam}
    The main example we have in mind is nonconnective algebraic K-theory, as defined in \cite[Sect.~9]{BlumbergGepnerTabuada} or via the generalized Bass--Thomason--Trobaugh construction as in \cite[Thm.~C]{CisinskiKhanKH}.
    We denote it simply by $\K$.
  \end{exam}

  \begin{notation}
    Let $E$ be a localizing invariant.
    For any compactly generated stable \inftyCat $\sC$, we set for convenience
    \[ E(\sC) := E(\sC^\omega) \]
    where $\sC^\omega$ denotes the full subcategory of compact objects.
  \end{notation}

  \begin{thm}\label{thm:excision}
    Let $E$ be a localizing invariant.
    Let $\Delta$ be a Milnor square of compactly generated stable \inftyCats of the form
    \begin{equation*}
      \begin{tikzcd}
        \sA \arrow[r,"f^*"]\arrow[d,"p^*"] & \sB\arrow[d,"q^*"] \\
        \sA' \arrow[r, "g^*"]              & \sB'.
      \end{tikzcd}
    \end{equation*}
    If $\Delta$ satisfies base change, then the induced square $E(\Delta)$
    \[
      \begin{tikzcd}
        \E(\sA) \ar{r}{f^*}\ar{d}{p^*}
          & \E(\sB) \ar{d}{q^*}
        \\
        \E(\sA') \ar{r}{g^*}
          & \E(\sB')
      \end{tikzcd}
    \]
    is cartesian.
  \end{thm}

  \thmref{thm:excision} immediately follows, in view of \thmref{thm:odot}, from the following statement:

  \begin{prop}\label{prop:E(Delta_0)}
    Let the notation be as in \thmref{thm:excision}.
    Then the induced square $E(\Delta_0)$
    \[
      \begin{tikzcd}
        \E(\sA) \ar{r}{f^*}\ar{d}{p^*}
          & \E(\sB) \ar{d}{q_0^*}
        \\
        \E(\sA') \ar{r}{g_0^*}
          & \E(\sA' \odot_{\sA}^{\sB'} \sB')
      \end{tikzcd}
    \]
    is cartesian.
  \end{prop}

  \begin{proof}
    First note that the Verdier localization sequence (see \remref{rem:C=<A,Q>})
    \[
      \sA \xrightarrow{d_*} \sC \xrightarrow{a^*} \sA' \odot_{\sA}^{\sB'} \sB
    \]
    induces a cofibre sequence
    \[
      \E(\sA) \xrightarrow{d_*} \E(\sC) \xrightarrow{a^*} \E(\sA' \odot_{\sA}^{\sB'} \sB).
    \]
    Similarly, $\E$ sends the Verdier localization sequence $\sB \to \sC \to \sA'$ (\remref{rem:C=<B,A'>}) to a cofibre sequence.
    Moreover, both functors in this sequence have \emph{compact} right adjoints which hence yield a splitting
    \[
      \E(\sC) \simeq \E(\sA') \oplus \E(\sB)
    \]
    given by the projections $\sC \to \sA'$ and $\sC \to \sB$.
    Under this isomorphism, the map $d_* : \E(\sA) \to \E(\sC)$ is induced by the maps $p^* : \E(\sA) \to \E(\sA')$ and $f^* : \E(\sA) \to \E(\sB)$, and the map $a^* : \E(\sC) \to \E(\sA' \odot_{\sA}^{\sB'} \sB)$ splits thanks to \remref{rem:a^* splitting} as
    \[ a^* : \E(\sC) \simeq \E(\sA') \oplus \E(\sB) \xrightarrow{g_0^*-q_0^*} \E(\sA' \odot_{\sA}^{\sB'} \sB). \]
    Thus we end up with a cofibre sequence of the form
    \[
      \E(\sA)
        \xrightarrow{(p^*,f^*)} \E(\sA') \oplus \E(\sB)
        \xrightarrow{g_0^* - q_0^*} \E(\sA' \odot_{\sA}^{\sB'} \sB).
    \]
    Recall that this is equivalent to the assertion that $E(\Delta_0)$ is cartesian.
  \end{proof}

\subsection{Categorical pro-Milnor excision}

  We now turn our attention to the behaviour of localizing invariants with respect to pro-Milnor squares (\defref{defn:pro Milnor}).
  Under some strong assumptions, we will be able to show a pro-variant of \thmref{thm:excision} (see \thmref{thm:pro excision}).

  We begin with some preliminary considerations on isomorphisms of pro-systems in $\Pro(\Presc)$.

  \begin{defn}
    Let $\{f_n^* : \sC_n \to \sD_n \}_n$ be a cofiltered system in $\Presc$.
    We say that $\{f_n^*\}_n$ is a \emph{pro-equivalence} if it induces an isomorphism in $\Pro(\Presc)$.
  \end{defn}

  While a general criterion for pro-equivalences seems out of reach, we will show that there is a large enough supply for our purposes.
  First note the following class of examples:

  \begin{exam}\label{exam:pro-equiv monogenic}
    Let $\{\phi_n : A_n \to B_n\}_n$ be a cofiltered system of homomorphisms of connective bounded above $\sE_1$-rings.
    If $\{\phi_n\}_n$ induces an isomorphism of underlying pro-spaces, then it follows from \cite[Lem.~2.28]{LandTamme} that it induces an isomorphism of pro-objects in the \inftyCat of $\sE_1$-rings.
    In particular, by functoriality of the construction $R \mapsto \LMod_R$ (as a functor from the \inftyCat of $\sE_1$-rings to $\Presc$), we deduce that the extension of scalars functors \[ \{\phi_n^* : \LMod_{A_n} \to \LMod_{B_n}\}_n \] define a pro-equivalence.
  \end{exam}

  \begin{exam}\label{exam:proequiv lim}
    Since passage to underlying pro-objects preserves finite limits and colimits, it follows that pro-equivalences are closed under finite (co)limits.
  \end{exam}

  We also have closure under certain colocalizations (and a dual statement for localizations that we leave to the reader to formulate):

  \begin{lem}\label{lem:proequiv coloc}\leavevmode
    Suppose given a cofiltered system of commutative squares in $\Presc$
    \[ \begin{tikzcd}
      \sC'_n \ar{r}{i_{n,*}}\ar[swap]{d}{g_n^*}
        & \sC_n \ar{d}{f_n^*}
      \\
      \sD'_n \ar{r}{i_{n,*}}
        & \sD_n
    \end{tikzcd} \]
    where $i_{n,*}$ is fully faithful for every $n$, with right adjoint $i_n^!$.
    Assume that for every $n$, the square is horizontally right-adjointable: that is, the base change transformation
    \[ g_n^* i_n^! \to i_n^! f_n^* \]
    is invertible.
    If $f^*$ is a pro-equivalence, then so is $g^*$.
  \end{lem}

  \begin{proof}
    By assumption, there exists a morphism $v^* : \{\sD_n\}_n \to \{\sC_n\}_n$ in $\Pro(\Presc)$ which is inverse to $f^* = \{f_n^*\}_n$.
    Let $u^* : \{\sD'_n\}_n \to \{\sC'_n\}_n$ be the morphism in $\Pro(\Pres)$ defined as the composite $u^* = i^! v^* i_*$.
    More explicitly, we may represent $v^*$ by the following data:
    \begin{enumerate}
      \item[($\ast$)]
      For every index $n$, there exists an index $n' > n$, a compact colimit-preserving functor $v^*_{n} : \sD_{n'} \to \sC_{n}$, isomorphisms
      \[
        f_{n}^* v^*_{n} \simeq \tr^* : \sD_{n'} \to \sD_{n},
        \qquad 
        v^*_{n} f_{n'}^* \simeq \tr^* : \sC_{n'} \to \sC_{n},
      \]
      and a homotopy coherent system of compatibilities between this data, where we write $\tr^*$ for the transition functors.
    \end{enumerate}
    In terms of such a choice, the morphism $u^*$ is represented by the colimit-preserving functors
    \[
      u^*_{n} : \sD'_{n'}
      \xrightarrow{i_*} \sD_{n'}
      \xrightarrow{v^*_{n}} \sC_{n}
      \xrightarrow{i^!} \sC'_{n}
    \]
    where we omit some decorations for simplicity.
    Note that we have canonical isomorphisms of functors
    \begin{align*}
      g_{n}^* u^*_{n}
      \simeq g_{n}^* i^! v^*_{n} i_*
      \simeq i^! f_{n}^* v^*_{n} i_*
      \simeq i^! \tr^* i_*
      \simeq i^! i_* \tr^*
      \simeq \tr^*\\
      u^*_{n} g_{n'}^*
      \simeq i^! v^*_{n} i_* g_{n'}^*
      \simeq i^! v^*_{n} f_{n'}^* i_*
      \simeq i^! \tr^* i_*
      \simeq i^! i_* \tr^*
      \simeq \tr^*,
    \end{align*}
    where we have used the assumptions that $i_*$ and $i^!$ commute with the vertical arrows.
    Note moreover that $u^*_{n}$ is compact: since $g_{n'}^*$ generates under colimits, this follows from the fact that $u^*_{n} g_{n'}^* \simeq \tr^*$ is compact.
    Thus the morphism $u^*$ is contained in the subcategory $\Pro(\Presc)$ and defines an inverse to $g^* = \{g_n^*\}_n$ in the latter.
  \end{proof}

  In the \emph{projectively} generated case (see \cite[Def.~5.5.8.23]{HTT}), and when the mapping pro-spaces are uniformly bounded, we can prove the following criterion for pro-equivalence.

  \begin{lem}\label{lem:proequiv criterion}
    Let $\{\sC_n\}_n$ and $\{\sD_n\}_n$ be cofiltered systems in $\Presc$ and let $\{f_n^* : \sC_n \to \sD_n \}_n$ be a cofiltered system of colimit-preserving compact projective\footnote{%
      A functor $f^* : \sC \to \sD$ is called compact projective if its right adjoint preserves sifted colimits.
    } functors which generate their codomains under colimits.
    Consider the following conditions:
    \begin{thmlist}
      \item\label{item:aisjdfnq}
      For every $n$ there exists a small set of compact projective generators $\{C_{n}^\alpha\}_\alpha$ of $\sC_n$ such that every transition functor $\sC_n \to \sC_m$ sends $C_n^\alpha \mapsto C_m^\alpha$ for every $\alpha$ and every $m<n$.

      \item
      For every pair of indices $\alpha$ and $\beta$, the map
      \[ \big\{ \Maps_{\sC_n}(C_n^\alpha, C_n^\beta) \big\}_{n} \to \big\{ \Maps_{\sD_n}(f_n^*(C_n^\alpha), f_n^*(C_n^\beta)) \big\}_{n} \]
      is an isomorphism of pro-spaces.

      \item
      There exists an integer $c\ge0$ such that the pro-spaces
      \[ \{\Maps_{\sC_n}(C_n^\alpha, C_n^\beta)\}_n \]
      are $c$-truncated for all $\alpha$ and $\beta$.
    \end{thmlist}
    Then $\{f_n^*\}_n$ induces a pro-equivalence.
  \end{lem}

  \begin{proof}
    For every $n$, consider the full subcategory $\sA_n \sub \sC_n$ spanned by finite direct sums of the objects $C_n^\alpha$ (as $\alpha$ varies).
    Then $\sA_n$ is an additive \inftyCat for which the inclusion $\sA_n \sub \sC_n$ extends to an equivalence $\sP_\Sigma(\sA_n) \simeq \sC_n$ by \cite[Prop.~5.5.8.25]{HTT}.
    Let $C_n \in \sC_n$ be the direct sum of the objects $C_n^\alpha$ (as $\alpha$ varies), and let $\sA^+_n \sub \sC_n$ be the full subcategory generated by $C_n$ under finite direct sums and direct summands.
    By \cite[Ex.~C.1.5.11]{SAG-20180204} this is an idempotent-complete additive \inftyCat equipped with an equivalence $\sP_\Sigma(\sA^+_n) \simeq \LMod^\cn_{A_n}$, where $A_n$ is the connective endomorphism algebra of $C_n$ (as in \cite[Rem.~C.2.1.9]{SAG-20180204}).
    Then we have inclusions $\sA_n \sub \sA^+_n$ which are closed under finite direct sums and hence give rise by \cite[Prop.~5.5.8.15]{HTT} to fully faithful colimit-preserving functors
    \begin{equation}\label{eq:lnxoaq}
      i_* : \sC_n \simeq \sP_\Sigma(\sA_n) \to \sP_\Sigma(\sA_n^+)
    \end{equation}
    whose right adjoints $i^!$ are restriction along $\sA_n \sub \sA_n^+$.
    Similarly let $D_n^\alpha \in \sD_n$ be the images of $C_n^\alpha$, $\sB_n \sub \sD_n$ the additive subcategory they generate, $D_n \in \sD_n$ their direct sum over $\alpha$, and $B_n$ the connective endomorphism algebra of $D_n$.
    Since the functors $f_n^*$ are compact projective and generate under colimits, the objects $D_n^\alpha$ form a small set of compact projective generators of $\sD_n$.
    Thus we similarly get functors $i_* : \sD_n \to \LMod^\cn_{B_n}$ fitting into commutative squares
    \[ \begin{tikzcd}
      \sC_n \ar{r}{i_*}\ar{d}{f_n^*}
        & \LMod^\cn_{A_n} \ar{d}{\phi_n^*}
      \\
      \sD_n \ar{r}{i_*}
        & \LMod^\cn_{B_n},
    \end{tikzcd} \]
    where $\phi_n^*$ is extension of scalars along the homomorphism $\phi_n : A_n \to B_n$ induced by $f_n^*$.
    Since $\{A_n\}_n$ and $\{B_n\}_n$ are bounded above by assumption, we see by \examref{exam:pro-equiv monogenic} that the functors $\phi_n^*$ induce a pro-equivalence.
    To conclude it remains to check the criteria of \lemref{lem:proequiv coloc} for the above square.

    The functors \eqref{eq:lnxoaq} preserve compact projective objects by construction, so we see that $i_*$ is compact projective.
    It remains to show that $i^!$ commutes with the vertical arrows.
    Since all functors in the square commute with sifted colimits, it will suffice to evaluate on the object $A_n \in \LMod^\cn_{A_n}$.
    By construction, the horizontal arrows send
    \begin{align*}
      A_n &\mapsto C_n,\\
      B_n &\mapsto D_n,
    \end{align*}
    whence the claim.
    Thus the conditions of \lemref{lem:proequiv coloc} are satisfied, and $\{f_n^*\}_n$ is a pro-equivalence.
  \end{proof}

  Under the conditions of \lemref{lem:proequiv criterion}, the \inftyCats $\sC_n$ and $\sD_n$ will never be stable, but only prestable.
  Nevertheless, by stabilization, we may deduce a following variant of \lemref{lem:proequiv criterion} for certain presentable stable \inftyCats.

  \begin{defn}\label{defn:weighted pstab}
    Let $\{\sC_i\}_{i\in\sI}$ be a diagram of \pstabs and compact colimit-preserving functors indexed by a small \inftyCat $\sI$.
    Suppose given full subcategories $(\sC_i)_{\ge 0} \sub \sC_i$ for every $i\in\sI$, closed under colimits and extensions, and collections $\{C^\alpha_i\}_{\alpha\in S}$ of compact projective objects of $(\sC_i)_{\ge0}$ indexed by some small set $S$.
    We will say that the collections $\{C^\alpha_i\}_{\alpha}$ form a set of \emph{projective generators} for $\{\sC_i\}_i$ if the following conditions hold:
    \begin{defnlist}
      \item
      For every $i \in \sI$, the objects $\{C^\alpha_i\}_{\alpha\in S}$ generate $(\sC_i)_{\ge0}$ under colimits and extensions, and the objects $\{\Sigma^{-n}(C^\alpha_i)\}_{n\ge0, \alpha}$ generate $\sC_i$ under colimits.
      
      \item
      For every morphism $i \to j$ in $\sI$, the induced functor $u_{i,j}^* : \sC_i \to \sC_j$ sends $C^\alpha_i \mapsto C^\alpha_j$ for every $\alpha \in S$, and its right adjoint $u_{i,j,*}$ sends $(\sC_j)_{\ge0}$ to $(\sC_i)_{\ge 0}$.
    \end{defnlist}
  \end{defn}

  \begin{rem}\label{rem:iohnpa}
    For $\sI = \Delta^0$, \defref{defn:weighted pstab} is essentially a presentable version of the notion of weighted \inftyCat as discussed in \ssecref{ssec:weight}.
    For $\sI = \Delta^1$, it corresponds to a weight-exact functor which generates under colimits.
  \end{rem}

  \begin{rem}\label{rem:ipajinp}
    Using \cite[Prop.~1.4.4.11, Lem.~7.2.2.6]{HA-20170918} we can recast \defref{defn:weighted pstab} in terms of t-structures: there are t-structures on the $\sC_i$---accessible, right-complete, and compatible with filtered colimits---with connective parts $(\sC_i)_{\ge0}$ generated under colimits and extensions by $\{C_i^\alpha\}_{\alpha}$, such that the functors $u_{i,j}^*$ and their right adjoints are right t-exact.
  \end{rem}

  \begin{rem}
    In case $\sI$ has an initial object $0$, we will only specify the generators $C^\alpha_0 \in \sC_0$, so that $C^\alpha_i \in \sC_i$ are implicitly defined as the images by the functor $\sC_0 \to \sC_i$ for every $i\in\sI$.
    If we have a cofiltered system of $\sI$-indexed diagrams $\{\sC_{n,i}\}_{n\in\Lambda,i\in\sI}$, where $\Lambda$ is cofiltered, then projective generators of $\{\sC_{n,i}\}_{n,i}$ (as a $\Lambda\times\sI$-indexed diagram) will be similarly specified by a small set $S$ and a collection of objects $\{C_{n,0}^\alpha\}_{\alpha\in S}$.
  \end{rem}

  We can now reformulate \lemref{lem:proequiv criterion} as follows:

  \begin{cor}\label{cor:proequiv criterion stable}
    Suppose given a cofiltered system $\{f_n^* : \sC_n \to \sD_n\}_n$ of \pstabs and compact colimit-preserving functors.
    Suppose there exists a small set $S$ and objects $\{C_n^\alpha\}_{\alpha\in S}$ of $\sC_n$ for every $n$, which projectively generate $\{f_n^*\}_n$ as a cofiltered system of $\Delta^1$-indexed diagrams.
    Then $\{f_n^*\}_n$ is a pro-equivalence in case the following conditions hold:
    \begin{thmlist}
      \item
      For every pair $\alpha$, $\beta$, the map
      \[ \big\{ \Maps_{\sC_n}(C_n^\alpha, C_n^\beta) \big\}_{n} \to \big\{ \Maps_{\sD_n}(f_n^*(C_n^\alpha), f_n^*(C_n^\beta)) \big\}_{n} \]
      is an isomorphism of pro-spaces.

      \item
      There exists an integer $c\ge0$ such that the mapping pro-spaces
      \[ \big\{ \Maps_{\sC_n}(C_n^\alpha, C_n^\beta) \big\}_{n} \]
      are $c$-truncated for all $\alpha$, $\beta$.
    \end{thmlist}
  \end{cor}

  We now prove our categorical pro-excision statement.

  \begin{thm}\label{thm:pro excision}
    Let $E$ be a localizing invariant.
    Let $\Delta$ be a pro-Milnor square of \pstabs which can be represented by a cofiltered system $\{\Delta_n\}_n$ of levelwise Milnor squares
    \[ \begin{tikzcd}
      \sA_n \ar{r}{f^*_n}\ar{d}{p^*_n}
        & \sB_n \ar{d}{q^*_n}
      \\
      \sA'_n \ar{r}{g^*_n}
        & \sB'_n
    \end{tikzcd} \]
    satisfying the following conditions:
    \begin{thmlist}
      \item
      There exists a small set $S$ and objects $\{A_n^\alpha\}_{\alpha\in S}$ of $\sA_n$ for every $n$, which projectively generate $\{\Delta_n\}_n$ as a cofiltered system of $(\Delta^1\times\Delta^1)$-indexed diagrams.

      \item
      There is an integer $c\ge0$ such that for all $\alpha$, $\beta$, the mapping pro-space
      \[ \big\{ \Maps_{\sB'_n}(g_n^*p_n^*(A_n^\alpha), g_n^*p_n^*(A_n^\beta)) \big\}_{n} \]
      is $c$-truncated.
    \end{thmlist}
    If $\Delta$ satisfies pro-base change, then the induced square of pro-objects
    \[
      \begin{tikzcd}
        \{\E(\sA_n)\}_n \ar{r}{f_n^*}\ar{d}{p_n^*}
          & \{\E(\sB_n)\}_n \ar{d}{q_n^*}
        \\
        \{\E(\sA'_n)\}_n \ar{r}{g_n^*}
          & \{\E(\sB'_n)\}_n
      \end{tikzcd}
    \]
    is cartesian.
  \end{thm}

  We will need the following lemma.

  \begin{lem}\label{lem:odot t-structure}
    Suppose given a precartesian square of \pstabs and compact colimit-preserving functors
    \[ \begin{tikzcd}
      \sA \ar{r}{f^*}\ar{d}{p^*}
        & \sB \ar{d}{q^*}
      \\
      \sA' \ar{r}{g^*}
        & \sB'.
    \end{tikzcd} \]
    Assume that the square is projectively generated by a collection $\{A^\alpha\}_\alpha$ of objects of $\sA$.
    Then the functors $q_0^* : \sB \to \sA'\odot_{\sA}^{\sB'} \sB$ and $b^* : \sA'\odot_{\sA}^{\sB'} \sB \to \sB'$ are both projectively generated (as $\Delta^1$-indexed diagrams).
    In particular, the collection $\{q_0^*f^*(A^\alpha)\}_\alpha$ forms a set of projective generators of the $\odot$-construction $\sA'\odot_{\sA}^{\sB'} \sB$.
  \end{lem}
  \begin{proof}
    We consider the t-structures defined in \remref{rem:ipajinp}.
    To simplify the notation set $\sQ := \sA' \odot_{\sA}^{\sB'} \sB$.
    By \cite[Prop.~1.4.4.11]{HA-20170918} there is a t-structure whose connective part $\sQ_{\ge0}$ is the cocomplete stable subcategory generated by $q_{0}^*(\sB_{\ge0})$.

    It follows from the assumption that each of the functors in the given square generates under colimits, and is compact projective (the latter by \cite[Lem.~7.2.2.6]{HA-20170918}, since it has right t-exact right adjoint by assumption).
    Hence by \lemref{lem:q_0 and g_0 generation} the functors $q_0^*$ and $b^*$ also generate under colimits and are compact.
    It remains to show that are right t-exact, and have right t-exact right adjoints.
    The functor $q_0^*$ is right t-exact by construction.
    Let us show $q_{0,*}$ is right t-exact.
    Since $q_0^*f^*$ generates under colimits and is right t-exact, it will suffice to show that $q_{0,*}q_{0}^*f^*$ is right t-exact.
    Since $\Delta$ is precartesian, \lemref{lem:base change Delta_0} implies that the latter functor is identified with the right t-exact functor $f^* p_{*} p^*$.

    Since $q_{0}^*$ generates under colimits and $b^*q_{0}^* \simeq q^*$ is right t-exact, it follows that $b^* : \sQ \to \sB'$ is also right t-exact.
    Finally, using the fact that $q_{0,*}$ is conservative and t-exact, it follows from the fact that $q_{0,*}b_* \simeq q_*$ is right t-exact that $b_*$ itself is right t-exact.
  \end{proof}

  The following key lemma shows that any pro-Milnor square as in \thmref{thm:pro excision} can be levelwise represented by squares coming from the $\odot$-construction.

  \begin{lem}\label{lem:levelwise Milnor}
    Let $\Delta$ be a pro-Milnor square as in \thmref{thm:pro excision}.
    If $\Delta$ satisfies pro-base change, then there exists an isomorphism
    \[ \{\Delta_{n,0}\}_n \to \{\Delta_n\}_n \]
    of commutative squares in $\Pro(\Presc)$, where every $\Delta_{n,0}$ is obtained from $\Delta_{n}$ via the $\odot$-construction.
  \end{lem}

  \begin{proof}
    Note that by \thmref{thm:odot}\ref{item:nalnfiowa}, the induced square $\Delta_{n,0}$
    \[ \begin{tikzcd}
      \sA_n \ar{r}\ar{d}
        & \sB_n \ar{d}
      \\
      \sA'_n \ar{r}
        & \sA'_n \odot_{\sA_n}^{\sB'_n} \sB_n
    \end{tikzcd} \]
    satisfies base change for every $n$.
    To simplify the notation, set $\sQ_n := \sA'_n \odot_{\sA_n}^{\sB'_n} \sB_n$.
    It remains only to show that the functors
    \[b_{n}^* : \sQ_n \to \sB'_n\]
    define a pro-equivalence as $n$ varies.
    Since $q_n^* \simeq b_n^* q_{n,0}^*$ generates $\sB'_n$ under colimits, so does $b_n^*$.
    By \lemref{lem:odot t-structure}, the functor $b_n^*$ is projectively generated for every $n$ (by the objects $B_n^\alpha = f_n^*(A_n^\alpha)$).
    Thus in order to apply \corref{cor:proequiv criterion stable} it will suffice to show that $b_n^*$ induce isomorphisms on mapping pro-spaces.
    For this we need a pro-version of \corref{cor:base change adjustment}.
    By \lemref{lem:base change adjustment} we have for every $n$ a commutative square of natural transformations
    \[ \begin{tikzcd}[matrix xscale=2]
      f_n^*p_{n,*} \ar{r}{\Ex_{\Delta_{n,0}}}\ar[equals]{d}
        & q_{n,0,*}g_{n,0}^* \ar{d}{\eta_{b_n}}
      \\
      f_n^*p_{n,*} \ar[swap]{r}{\Ex_{\Delta_{n}}}
        & q_{n,*}g_n^*.
    \end{tikzcd} \]
    By \lemref{lem:base change Delta_0}, the upper arrow is invertible for every $n$.
    Evaluating at the object $A_n'^\alpha := p_n^*(A_n^\alpha) \in \sA'_n$ for any $\alpha$, applying $\Maps_{\sB_n}(B_n^\beta, -)$ for any $\beta$, and passing to pro-systems yields a commutative diagram of pro-spaces
    \[ \begin{tikzcd}[matrix xscale=1.5]
      \{\Maps_{\sB_n}(B_n^\beta, f_n^*p_{n,*}(A_n'^\alpha))\}_n \ar{r}{\Ex_{\Delta_{n,0}}}\ar[equals]{d}
      & \{\Maps_{\sB_n}(B_n^\beta, q_{n,0,*}g_{n,0}^*(A_n'^\alpha))\}_n \ar{d}
      \\
      \{\Maps_{\sB_n}(B_n^\beta, f_n^*p_{n,*}(A_n'^\alpha))\}_n \ar[swap]{r}{\Ex_{\Delta_{n}}}
      & \{\Maps_{\sB_n}(B_n^\beta, q_{n,*}g_n^*(A_n'^\alpha))\}_n
    \end{tikzcd} \]
    where the upper arrow is invertible in $\Pro(\Spc)$.
    The claim is that the right-hand vertical arrow is invertible in $\Pro(\Spc)$.
    Since $\Delta$ satisfies pro-base change, it can be represented by some cofiltered system $\{\Delta'_n\}_n$ for which $\Delta'_n$ satisfies base change for every $n$.
    Choose an isomorphism $\{\Delta_n\}_n \simeq \{\Delta'_n\}_n$ (of squares in $\Pro(\Presc)$); re-indexing if necessary, we may assume that it is induced by levelwise morphisms $\Delta_n \to \Delta'_n$.
    By functoriality, such an isomorphism gives rise to an identification between the lower arrow in the diagram above and the analogous construction $\{\Ex_{\Delta'_n}\}_n$ for the squares $\Delta'_n$.
    Since $\Ex_{\Delta'_n}$ is invertible for every $n$, we conclude that the lower arrow above is also invertible.
  \end{proof}
  
  \begin{proof}[Proof of \thmref{thm:pro excision}]
    Combine Lemma~\ref{lem:levelwise Milnor} and \propref{prop:E(Delta_0)}.
  \end{proof}


\section{Localizing invariants of algebraic stacks}
\label{sec:stacky_applications}

  \begin{defn}
    Let $\sX$ be a derived algebraic stack which is perfect in the sense of \ssecref{ssec:compgen}.
    For instance, suppose $\sX$ is \emph{ANS} (affine diagonal and nice stabilizers, see \defref{defn:ANS} and \thmref{thm:ANS perfect}).
    Then we set
    \[ E(\sX) := E(\Qcoh(\sX)) \simeq E(\Perf(\sX)) \]
    for any localizing invariant $E$.
    For example, we have the nonconnective algebraic K-theory spectrum $\K(\sX)$.
  \end{defn}

  \begin{rem}
    Let $\sX$ and $\sY$ be perfect derived algebraic stacks.
    Then for any morphism $f : \sX \to \sY$, the functor $f^* : \Qcoh(\sY) \to \Qcoh(\sX)$ preserves perfect complexes and is therefore compact.
    In particular, there is an induced map
    \[ f^* : E(\sY) \to E(\sX) \]
    for any localizing invariant $E$.
  \end{rem}

\subsection{Étale excision}

  See \ssecref{ssec:compgen} for the notion of perfectness of stacks.

  \begin{thm}\label{thm:K excision}
    Let $f : \sX' \to \sX$ be a representable morphism of \emph{perfect} derived algebraic stacks.
    Suppose there exists a closed immersion $i : \sZ \to \sX$ with quasi-compact open complement such that $f$ is an isomorphism infinitely near $\sZ$, i.e., the induced morphism $\sX'^\wedge_{f^{-1}(\sZ)} \to \sX^\wedge_{\sZ}$ is invertible.
    Then for every localizing invariant $E$, the induced square
    \[ \begin{tikzcd}
      E(\sX) \ar{r}{j^*}\ar{d}{f^*}
        & E(\sX\setminus\sZ) \ar{d}
      \\
      E(\sX') \ar{r}
        & E(\sX'\setminus f^{-1}(\sZ))
    \end{tikzcd} \]
    is cartesian.
  \end{thm}
  \begin{proof}
    Apply the criterion of \cite[Thm.~18]{Tamme_2018} (cf. \cite[Cor.~13]{HoyoisEfimov}) to the cartesian square in \thmref{thm:Qcoh excision}.
    Let $j': \sU' \to \sX'$ denote the open immersion complementary to $f^{-1}(\sZ)$.
    Fully faithfulness of the functor $j'_* : \Qcoh(\sU') \to \Qcoh(\sX')$, i.e., invertibility of the unit $\id \to j'^* j'_*$, follows from the base change formula for the self-intersection $\sU' \fibprod^\bR_{\sX'} \sU' \simeq \sU'$.
  \end{proof}

  Specializing to étale neighbourhoods gives:

  \begin{cor}[Étale excision]\label{cor:K Nis}
    Suppose given a Nisnevich square of derived algebraic stacks of the form
    \[ \begin{tikzcd}
      \sU' \ar{r}{j'}\ar{d}{g}
        & \sX' \ar{d}{f}
      \\
      \sU \ar{r}{j}
        & \sX
    \end{tikzcd} \]
    where $\sX$ and $\sX'$ are \emph{perfect}.
    Then for every localizing invariant $E$, the induced square
    \[ \begin{tikzcd}
      E(\sX) \ar{r}{j^*}\ar{d}{f^*}
        & E(\sU) \ar{d}
      \\
      E(\sX') \ar{r}
        & E(\sU')
    \end{tikzcd} \]
    is cartesian.
  \end{cor}

  \begin{rem}
    Note that for affine Nisnevich squares (i.e., squares as above where $f$ is affine), \corref{cor:K Nis} can be deduced alternatively from \corref{cor:Qcoh Nis} and \thmref{thm:excision}.
    This weaker statement would in fact suffice for our purposes in this paper.
  \end{rem}

  \begin{rem}
    \corref{cor:K Nis} also holds for squares as above where $f$ is not assumed representable.
    This can also be deduced from \cite[Thm.~18]{Tamme_2018}, e.g. using Nisnevich descent for the presheaf of \inftyCats $\sX \mapsto \Qcoh(\sX)$ and \cite[Thm.~2.2.7]{KhanSAG}.
  \end{rem}

  \begin{rem}
    By \cite[Thm.~2.2.7]{KhanSAG}, \corref{cor:K Nis} implies that the presheaf of spectra $\sX \mapsto E(\sX)$ satisfies descent for the Grothendieck topology generated by Nisnevich squares.
    The discussion of \cite[Prop.~2.8]{hoyoiskrishna} goes through \emph{mutatis mutandis} in the derived setting to show that this topology is generated by a stacky variant of the usual notion of Nisnevich covers.
  \end{rem}

\subsection{Formal Milnor and finite excision}

  \begin{thm}\label{thm:K formal excision}
    Suppose given a square of \evcoconn derived algebraic stacks
    \[
      \begin{tikzcd}
        \sZ' \ar{d}{g}\ar{r}{i'}
        & \sX' \ar{d}{f}
        \\
        \sZ \ar{r}{i}
        & \sX.
      \end{tikzcd}
    \]
    with $\sX$ noetherian ANS.
    Assume that $i$ is a closed immersion, $f$ is affine, the square is cartesian on classical truncations, and the formally completed square \eqref{eq:formal cdh square} is cocartesian.
    Then for any localizing invariant $E$, the induced square
      \begin{equation*}
        \begin{tikzcd}
          \{E(\sX)\}\ar{r}\ar{d}
            & \form{E}(\sX^\wedge_{\sZ})\ar{d}
          \\
          \{E(\sX')\}\ar{r}
            & \form{E}(\sX'^\wedge_{\sZ'})
        \end{tikzcd}
      \end{equation*}
    is pro-cartesian.
  \end{thm}
  \begin{proof}
    By \corref{cor:K Nis}, \thmref{thm:ANS local structure}, \propref{prop:afsoub01} and \thmref{thm:ANS perfect}, we may reduce to the case $\sX = [X/G]$, where $G$ is an embeddable nice group scheme over an affine scheme $S$, and $X$ is an affine derived $S$-scheme with $G$-action.
    Now we may apply the criterion of \thmref{thm:pro excision} to the square
    \begin{equation*}
      \begin{tikzcd}
        \{\Qcoh(\sX)\}\ar{r}\ar{d}
          & \Qcohform(\sX^\wedge_{\sZ})\ar{d}
        \\
        \{\Qcoh(\sX')\}\ar{r}
          & \Qcohform(\sX'^\wedge_{\sZ'})
      \end{tikzcd}
    \end{equation*}
    which is pro-Milnor and satisfies pro-base change by \thmref{thm:Qcohform Ferrand}.
    The assumptions of \thmref{thm:pro excision} are verified by \remref{rem:f^* compact} and \propref{prop:X/G perf}.
  \end{proof}

  \begin{cor}[Formal Milnor excision]\label{cor:K formal Milnor excision}
    Suppose given a Milnor square of noetherian algebraic stacks of the form \eqref{eq:Milnor square of stacks}.
    Then for any localizing invariant $E$, the induced square
    \begin{equation*}
      \begin{tikzcd}
        \{E(\sX)\} \ar{r}{\hat{i}^*}\ar{d}{f^*}
          & \form{E}(\sX^\wedge_{\sZ}) \ar{d}{(f^\wedge)^*}
        \\
        \{E(\sX')\} \ar{r}{\hat{i'}^*}
          & \form{E}(\sX'^\wedge_{\sZ'}).
      \end{tikzcd}
    \end{equation*}
    is pro-cartesian.
  \end{cor}

  \begin{cor}[Formal finite excision]\label{cor:K formal finite excision}
    Suppose given a finite cdh square of \evcoconn derived stacks
    \begin{equation*}
      \begin{tikzcd}
        \sZ' \ar{r}\ar{d}
          & \sX' \ar{d}{f}
        \\
        \sZ \ar{r}{i}
          & \sX
      \end{tikzcd}
    \end{equation*}
    with $\sX$ noetherian ANS.
    Then for any localizing invariant $E$, the induced square
      \begin{equation*}
        \begin{tikzcd}
          \{E(\sX)\}\ar{r}\ar{d}
            & \form{E}(\sX^\wedge_{\sZ})\ar{d}
          \\
          \{E(\sX')\}\ar{r}
            & \form{E}(\sX'^\wedge_{\sZ'})
        \end{tikzcd}
      \end{equation*}
    is pro-cartesian.
  \end{cor}

  \begin{cor}[Formal nil-excision]\label{cor:K formal nil excision}
    Let $\sX$ be a \evcoconn noetherian ANS derived stack with classical truncation $\sX_\cl$, and let $\sZ \hook \sX$ be a closed immersion with $0$-truncated quasi-compact open complement.
    Then for any localizing invariant $E$, the induced square
      \begin{equation*}
        \begin{tikzcd}
          \{E(\sX)\}\ar{r}\ar{d}
            & \form{E}(\sX^\wedge_\sZ)\ar{d}
          \\
          \{E(\sX_\cl)\}\ar{r}
            & \form{E}((\sX_\cl)^\wedge_{\sZ_\cl})
        \end{tikzcd}
      \end{equation*}
    is pro-cartesian.
  \end{cor}

\subsection{Derived blow-ups}

  Given a \emph{quasi-smooth} closed immersion of derived algebraic stacks, one can form the derived blow-up $\tsX$ in the sense of \cite[4.1.6]{khan2018virtual}, which fits in a commutative square
  \[ \begin{tikzcd}
    \sD \ar{r}{i_\sD}\ar{d}{q}
      & \tsX \ar{d}{p}
    \\
    \sZ \ar{r}{i}
      & \sX,
  \end{tikzcd} \]
  where $q$ is the projection of the projectivized normal bundle $\sE = \P(\sN_{\sZ/\sX})$ and $i_\sD$ is a \emph{virtual Cartier divisor}, i.e., a quasi-smooth closed immersion of virtual codimension $1$.
  See \cite{khan2018virtual} for background on quasi-smoothness and virtual Cartier divisors.
  This square is not homotopy cartesian, but it is an abstract blow-up square in the sense of \defref{defn:cdh square}.
  The following result was proven in \cite[Thm.~A]{khan2018algebraic}:

  \begin{thm}\label{thm:derived blow-up}
    Let $i : \sZ \to \sX$ be a quasi-smooth closed immersion of derived algebraic stacks.
    Then for every localizing invariant $E$, there is a cartesian square
    \[ \begin{tikzcd}
      E(\sX) \ar{r}{i^*}\ar{d}{p^*}
        & E(\sZ) \ar{d}{q^*}
      \\
      E(\tsX) \ar{r}
        & E(\sD).
    \end{tikzcd} \]
  \end{thm}

  In this subsection, we derive a formal version of this statement.
  This is a special case of formal proper excision.

  \begin{prop}\label{prop:formal derived blow-up excision}
    Let the notation be as above and assume that $\sX$ is a noetherian ANS stack.
    Then the square
    \[ \begin{tikzcd}
      \{E(\sX)\} \ar{r}\ar{d}
        & \form{E}(\sX^\wedge_\sZ) \ar{d}
      \\
      \{E(\tsX)\} \ar{r}
        & \form{E}(\tsX^\wedge_\sD)
    \end{tikzcd} \]
    is pro-cartesian.
  \end{prop}

  \begin{rem}
      Note that if $\sX$ is ANS, then so is the derived blow-up $\tsX$ by \lemref{lem:rep over ANS}.
  \end{rem}

  \begin{rem}\label{rem:pojanqeff}
    If $\sX$ has the resolution property, the quasi-smooth closed immersion $i : \sZ \to \sX$ can be realized as the derived zero locus of some section $s$ of a vector bundle $\sV$ on $\sX$ (\remref{rem:qs res prop}).
    For every integer $n$, let $\sZ(n) \to \sX$ denote the derived zero locus of $s^{\otimes n}$.
    By \lemref{lem:pjpfmq}, the formal completion of $\sX$ can then be represented by \[ \sX^\wedge \simeq \{ \sZ(n) \}_n. \]
    Similarly, it follows from \remref{rem:formal base change} that the formal completion of the derived blow-up can be represented by \[\tsX^\wedge \simeq \{\sZ'(n)\}_n,\] where $\sZ'(n)$ denotes the derived base change $\sZ(n) \fibprod_\sX \tsX$.
  \end{rem}

  \begin{constr}\label{constr:alnlvlad}
    Assume that $\sX$ has the resolution property and let $\sZ(n)$ be as in \remref{rem:pojanqeff}.
    The derived blow-up square defining $\tsX$ is the derived base change of the classical blow-up square
    \[ \begin{tikzcd}
      \sD_{\sV} \ar{r}\ar{d}
          & \widetilde{\sV} \ar{d}
        \\
        \sX \ar{r}{0}
          & \sV
    \end{tikzcd} \]
    of the zero section $0 : \sX \to \sV$.
    Let $\sX(n)$ denote the $n$th infinitesimal thickening of the latter. Write $\sD_{\sV}(n)$ for the (classical) fibre product of $\sX(n)$ and
    $\widetilde{\sV}$ over $\sV$  (which is an effective Cartier divisor in $\widetilde{\sV}$)
    and let $\sZ'_{\sV}(n)$ denote the derived fiber product.
    The two commutative squares
    \begin{equation}\label{eq:V squares}
      \begin{tikzcd}
        \sD_{\sV}(n) \ar{r}\ar{d}
          & \widetilde{\sV} \ar{d}
        \\
        \sX(n) \ar{r}
          & \sV,
      \end{tikzcd}
      \qquad
      \begin{tikzcd}
        \sZ'_{\sV}(n) \ar{r}\ar{d}
          & \widetilde{\sV} \ar{d}
        \\
        \sX(n) \ar{r}
          & \sV
      \end{tikzcd}
    \end{equation}
    define by derived base change to $\sX$ the squares
    \begin{equation}\label{eq:X squares}
      \begin{tikzcd}
        \sD(n) \ar{r}\ar{d}
          & \tsX \ar{d}
        \\
        \sZ(n) \ar{r}
          & \sX,
      \end{tikzcd}
      \qquad
      \begin{tikzcd}
        \sZ'(n) \ar{r}\ar{d}
          & \tsX \ar{d}
        \\
        \sZ(n) \ar{r}
          & \sX.
      \end{tikzcd}
    \end{equation}
    The right-hand square is homotopy cartesian.
    The left-hand square is a derived blow-up square for $n=1$, and otherwise is a thickening of the latter.%
    \footnote{In fact, $\sD(n)$ can be described as the $n$-fold sum $n\sD$ of the virtual Cartier divisor $\sD$, but we do not need this here.}
  \end{constr}

  The ind-stack $\{\sD(n)\}$ gives another presentation of $\tsX^\wedge$:

  \begin{lem}\label{lem:divisor presentation}
    If $\sX$ is a noetherian algebraic stack, then the morphisms $\sD(n) \to \sZ'(n)$ induce an isomorphism
    \[ \{\sD(n)\}_n \to \{\sZ'(n)\}_n \simeq \tsX^\wedge \]
    of ind-stacks over $\tsX$.
  \end{lem}
  \begin{proof}
    By derived base change, it will suffice to show the claim for the analogous constructions over $\widetilde{\sV}$.
    Note that both squares in \eqref{eq:V squares} consist of classical stacks, with the sole exception of $\sZ'_\sV(n)$.
    In fact, the morphism $\sD_\sV(n) \to \sZ'_\sV(n)$ exhibits the domain as the classical truncation of the codomain.
    Thus the morphism \[\{\sD_\sV(n)\}_n \to \{\sZ'_\sV(n)\}_n\] over $\widetilde{\sV}$ can be identified with the morphism from the classical formal completion of $\widetilde{\sV}$ (in the classical stack $\sD_\sV$) to the formal completion in the sense of \defref{defn:formal completion}, and is invertible by \remref{rem:formal completion classical}.
  \end{proof}

  \begin{notation}
    To simplify the notation, we will compress commutative squares of ind-stacks of the form
    \[
      \begin{tikzcd}
        \sZ' \ar{r}\ar{d}
          & \sX' \ar{d}
        \\
        \sZ \ar{r}
          & \sX,
      \end{tikzcd}
    \]
    where the horizontal arrows can be represented by levelwise closed immersions, into morphisms of pairs \[ (\sX', \sZ') \to (\sX, \sZ). \]
    We will say that such a morphism is an \emph{$E$-equivalence} if it induces an isomorphism
    \[ \form{E}(\sX, \sZ) \to \form{E}(\sX', \sZ'), \]
    where $\form{E}(\sX, \sZ)$ is defined as the fibre of $\form{E}(\sX) \to \form{E}(\sZ)$.
  \end{notation}

  To show \propref{prop:formal derived blow-up excision}, we will need to compare the pairs
  \[ \{(\tsX, \sZ'(n))\}_n \quad\mrm{and}\quad \{(\sX, \sZ(n))\}_n, \]
  up to $E$-equivalence.
  By \thmref{thm:derived blow-up}, we have $E$-equivalence of the pair $\{(\sX, \sZ(n))\}_n$ with its derived blow up $\{(\tsX(n), \sD(n,n))\}_n$, which is $E$-equivalent to the pair $\{(\tsX(n), \sZ'(n,n))\}_n$ by \lemref{lem:divisor presentation}, where $\sZ'(n,n))$ is the derived pullback of $\sZ(n)$ in $\tsX(n)$.
  The goal is therefore to understand the relation between $\{(\tsX(n), \sZ'(n,n))\}_n$ and $\{(\tsX, \sZ'(n))\}_n$.
  For this purpose it will be convenient to introduce the following bi-indexed construction.

  \begin{constr}\label{constr:biindex}
    For every pair of natural numbers $n$ and $k$, consider the following two commutative squares:
    \begin{equation}\label{eq:V bi squares}
      \begin{tikzcd}
        \sD_{\sV}(n,k) \ar{r}\ar{d}
          & \widetilde{\sV}(k) \ar{d}
        \\
        \sX(n) \ar{r}
          & \sV,
      \end{tikzcd}
      \qquad
      \begin{tikzcd}
        \sZ'_{\sV}(n,k) \ar{r}\ar{d}
          & \widetilde{\sV}(k) \ar{d}
        \\
        \sX(n) \ar{r}
          & \sV
      \end{tikzcd}
    \end{equation}
    Here $\widetilde{\sV}(k)$ is the blow-up of $\sV$ centred in $\sX(k)$.
    The left-hand square is classically cartesian and the right-hand square is homotopy cartesian.
    As above, all stacks are underived except for $\sZ'_\sV(n,k)$.

    Now by derived base change to $\sX$ we get the squares
    \begin{equation}\label{eq:X bi squares}
      \begin{tikzcd}
        \sD(n,k) \ar{r}\ar{d}
          & \tsX(k) \ar{d}
        \\
        \sZ(n) \ar{r}
          & \sX,
      \end{tikzcd}
      \qquad
      \begin{tikzcd}
        \sZ'(n,k) \ar{r}\ar{d}
          & \tsX(k) \ar{d}
        \\
        \sZ(n) \ar{r}
          & \sX.
      \end{tikzcd}
    \end{equation}
    Note that we have
    \[ \sD(n, 1) = \sD(n), \quad \sZ'(n, 1) = \sZ'(n) \]
    for every $n$.
    We regard $\sD(n,k)$ and $\sZ'(n,k)$ as ind-stacks indexed by the poset of pairs $(n,k) \in \bN \times \bN$, where $(n,k) \le (n', k')$ iff $n \le n'$
    and $k \le k'$.
  \end{constr}

  \begin{proof}[Proof of \propref{prop:formal derived blow-up excision}]
    By \corref{cor:K Nis}, \thmref{thm:ANS local structure}, \propref{prop:afsoub01} and \thmref{thm:ANS perfect}, we may reduce to the case $\sX = [X/G]$, where $G$ is an embeddable nice group scheme over an affine scheme $S$, and $X$ is an affine $S$-scheme with $G$-action.
    Then by \propref{prop:X/G res prop}, $\sX$ has the resolution property so we are in the situation of \remref{rem:pojanqeff} and \constrref{constr:alnlvlad}.

    We need to show that the morphism of pairs
    \begin{equation}\label{eq:pzxockp}
      \{(\tsX, \sZ'(n))\}_n \to \{(\sX, \sZ(n))\}_n
    \end{equation}
    is an $E$-equivalence.
    Passing to the bi-indexed system constructed in \constrref{constr:biindex},
    note that for $k=n$, the square
    \[
      \begin{tikzcd}
        \sD(n,n) \ar{r}\ar{d}
          & \tsX(n) \ar{d}
        \\
        \sZ(n) \ar{r}
          & \sX
      \end{tikzcd}
    \]
    is a derived blow-up square and therefore by \thmref{thm:derived blow-up},
    \[
    \{(\tsX(n), \sD(n,n))\}_n \to  \{(\sX, \sZ(n))\}_n
    \]
    induces a levelwise $E$-equivalence.
    For each $k$, \lemref{lem:divisor presentation} gives an identification of ind-stacks
     \[ \{\sD(n,k)\}_n \to \{\sZ'(n,k)\}_n, \]
    which induces an equivalence
    \[ \{\sD(n,n)\}_n \to \{\sZ'(n,n)\}_n. \]
    Therefore we have an $E$-equivalence
    \begin{equation} \label{eq:absdp}
    \{(\tsX(n), \sZ'(n,n))\}_n \simeq \{(\tsX(n), \sD(n,n))\}_n \to  \{(\sX, \sZ(n))\}_n.
    \end{equation}

    The inclusion of posets $\Delta_{\bN} \subseteq \bN \times \bN \supseteq \bN \times \{1\}$, induces morphisms
    \[\{(\tsX(m), \sZ'(m,m))\}_m \to \{(\tsX(k), \sZ'(n,k))\}_{n,k} \leftarrow \{(\tsX(n), \sZ'(n,1))\}_n \simeq \{(\tsX(n), \sZ'(n))\}_n,\]
    where the left arrow is an equivalence by cofinality and we shall see that the right arrow is an equivalence by finite formal excision.
     As one can see by base change from $\sV$, there is for every $k$ a \emph{finite} morphism
    \[ \tsX \to \tsX(k) \]
    which fits in a commutative diagram
    \[ \begin{tikzcd}
      \sZ'(n) \ar{r}\ar{d}
        & \tsX \ar{d}
      \\
      \sZ'(n, k) \ar{r}\ar{d}
        & \tsX(k) \ar{d}
      \\
      \sZ(n) \ar{r}
        & \sX
    \end{tikzcd} \]
    where the squares are homotopy cartesian for each $n,k$.
    This provides a factorization
    \begin{equation}\label{eq:zlxchoash}
      \{(\tsX, \sZ'(n))\}_{n} \to \{(\tsX(n), \sZ'(n,n))\}_{n} \to \{(\sX, \sZ(n))\}_{n},
    \end{equation}
   where the first arrow is an $E$-equivalence by finite excision (\corref{cor:K formal finite excision}) and the second arrow is an $E$-equivalence by \eqref{eq:absdp}.
  \end{proof}

\subsection{Formal proper excision}

  The following is the main result of this section.
  It generalizes \corref{cor:K formal finite excision} and \propref{prop:formal derived blow-up excision} to arbitrary proper cdh squares.
  The proof is essentially the same as that in \cite[5.3.4]{khan2018algebraic}.

  \begin{thm}\label{thm:K formal proper excision}
    Suppose given a proper cdh square of algebraic stacks
    \begin{equation*}
      \begin{tikzcd}
        \sZ' \ar{r}\ar{d}
          & \sX' \ar{d}{f}
        \\
        \sZ \ar{r}{i}
          & \sX
      \end{tikzcd}
    \end{equation*}
    with $\sX$ noetherian and ANS.
    Then for any localizing invariant $E$, the induced square
      \begin{equation*}
        \begin{tikzcd}
          \{E(\sX)\}\ar{r}\ar{d}
            & \form{E}(\sX^\wedge_{\sZ})\ar{d}
          \\
          \{E(\sX')\}\ar{r}
            & \form{E}(\sX'^\wedge_{\sZ'})
        \end{tikzcd}
      \end{equation*}
    is pro-cartesian.
  \end{thm}
  \begin{proof}
    Let us first demonstrate the claim in the case where $f$ is the projection of the blow-up $\sX' = \on{Bl}_{\sZ}\sX$.
    By Nisnevich descent (\corref{cor:K Nis}), \thmref{thm:ANS local structure} and \propref{prop:afsoub01}, we may assume that $\sX = [X/G]$ where $G$ is an embeddable nice group scheme over an affine scheme $S$, and $X$ is an affine derived $S$-scheme with $G$-action.
    Since $\sX$ has the resolution property (\propref{prop:X/G res prop}), $i : \sZ \to \sX$ is the classical truncation of a quasi-smooth closed immersion $\widetilde{\sZ} \to \sX$ (\constrref{constr:ohuhoj}).
    Let $\widetilde{\sX} \to \sX$ denote the derived blow-up of the latter.
    Since the formal completion $\sX^\wedge$ only depends on the classical truncation $\widetilde{\sZ}_\cl \simeq \sZ$, the square in question factors as in the diagram below:
    \[ \begin{tikzcd}
      \{E(\sX)\} \ar{r}\ar{d}
        & \form{E}(\sX^\wedge) \ar{d}
      \\
      \{E(\widetilde{\sX})\} \ar{r}\ar{d}
        & \form{E}(\widetilde{\sX}^\wedge) \ar{d}
      \\
      \{E(\sX')\} \ar{r}
        & \form{E}(\sX'^\wedge).
    \end{tikzcd} \]
    The upper square is cartesian by \propref{prop:formal derived blow-up excision}.
    Since $\sX' \to \tsX$ is a closed immersion which is an isomorphism away from $\sZ$, the lower square is also cartesian by \corref{cor:K formal finite excision}.
    This concludes the proof in the case of a blow-up square.

    Now consider the case of an arbitrary proper morphism $f$.
    By \corref{cor:K formal finite excision} it is safe to replace $\sX$ by the schematic closure of $\sX\setminus\sZ$ and thereby assume that $\sX\setminus\sZ$ is schematically dense in $\sX$ and $\sX'\setminus\sZ'$ is schematically dense in $\sX'$.
    In this case Rydh's stacky generalization of Raynaud--Gruson (see \cite{rydhflatification}, \cite[Cor.~2.4]{hoyoiskrishna}) to $f : \sX' \to \sX$ yields a proper morphism $f' : \sX'' \to \sX'$ such that $f \circ f' : \sX'' \to \sX$ is a sequence of $(\sX\setminus\sZ)$-admissible blow-ups.
    In particular, $f'$ sits in a proper cdh square $Q'$ over the original square $Q$.
    Applying the construction again to $f'$ yields a third proper cdh square $Q''$ over $Q'$.
    Then it suffices to show the claim for the two squares $Q' \circ Q$ and $Q'' \circ Q'$, so we have reduced to the case where $f$ is a sequence of $(\sX\setminus\sZ)$-admissible blow-ups.
    By induction, we may as well assume it is a $(\sX\setminus\sZ)$-admissible blow-up.
    Using \corref{cor:K formal finite excision} and the same argument as in \cite[Claim~5.3]{Kerz_2017}, one finally reduces to the case of a blow-up considered above.
  \end{proof}

  \begin{rem}\label{rem:ashfu01bp}
    \thmref{thm:K formal proper excision} also holds ``with supports'' in any closed substack $\sY \sub \sX$.
    That is, one can replace $E(\sX)$ with $E(\sX~\mrm{on}~\sY)$, $\form{E}(\sX^\wedge_\sZ)$ with $\form{E}(\sX^\wedge_\sZ~\mrm{on}~\sY^\wedge_{\sY\cap\sZ})$, etc.
    Here $E(\sX~\mrm{on}~\sY)$ is $E$ applied to the kernel of the restriction $\Perf(\sX) \to \Perf(\sX\setminus\sY)$ as in \cite{ThomasonTrobaugh}.
    Note that since $E$ is localizing, there are exact triangles
    \[
      E(\sX~\mrm{on}~\sY) \to E(\sX) \to E(\sX\setminus\sY).
    \]
    Therefore, the ``with supports'' variant of \thmref{thm:K formal proper excision} follows immediately from the ``without supports'' one.
  \end{rem}


\section{Negative K-theory}
\label{sec:weightstructures}

\ssec{Nil-invariance of negative K-groups}
\label{ssec:weight}

  In this subsection we prove a nil-invariance result for sufficiently negative K-groups of ANS stacks (\corref{nilinvariance_general}).

  We make use of the formalism of weight structures of Bondarko.
  Let $\sC$ be an additive \inftyCat which is projectively generated in the sense of \cite[Def.~5.5.8.23]{HTT}.
  Then the full subcategory $\sA$ of compact projective objects is an idempotent-complete additive \inftyCat for which the inclusion $\sA \subset \sC$ extends to an equivalence $\sP_\Sigma(\sA) \simeq \sC$ by \cite[Prop.~5.5.8.25]{HTT}.
  By \cite[Prop.~C.1.5.7]{SAG-20180204}, $\sC$ is prestable and thus embeds fully faithfully into its stabilization $\Spt(\sC)$.
  In this situation, the full subcategory $\sD = \Spt(\sC)^\omega$ of compact objects admits a \emph{weight structure} in the sense of \cite{BondarkoWeights}.
  This weight structure is bounded, its heart $\sD^{w=0}$ is $\sA$, 
  and the subcategory $\sD^{w\ge 0} \sub \sD$ of connective objects is $\sC^\omega$.
  Moreover, every bounded weight structure arises in this manner: in fact, the \inftyCat of weighted \inftyCats and weight-exact functors is equivalent to the \inftyCat of idempotent-complete additive \inftyCats by \cite[Prop.~3.3]{Vovastheoremoftheheart} (see also \cite[Props.~3.1.4, 3.1.5]{SosniloRegular}).
  We refer the reader to \cite[1.3]{Vovastheoremoftheheart}, \cite[3.1]{SosniloRegular}, or \cite{ElmantoSosnilo} for an $\infty$-categorical account of the theory of weight structures, originally developed in \cite{BondarkoWeights}.

  \begin{exam}\label{exam:panoujnq}
    Let $A$ be a connective $\sE_\infty$-ring.
    Then the \inftyCat $\Mod_A^\cn$ of connective $A$-modules is projectively generated by $A$.
    Thus there exists a canonical weight structure on the stable \inftyCat $\Perf_A$ whose heart is the full subcategory of finite projective $A$-modules.
  \end{exam}

  The next example, a slight generalization of \cite[Theorem~3.4.3]{SosniloRegular}, will play in an important role in what follows.

  \begin{exam} \label{exam:rfmalrcdsc}
    Let $R$ be a commutative ring, $G$ an embeddable linearly reductive group scheme over $R$, and $A$ a derived commutative ring over $R$.
    Then by \propref{prop:X/G perf}, the \inftyCat $\Qcoh([\Spec(A)/G])_{\ge 0}$ of connective $G$-equivariant $A$-modules is projectively generated and there exists a canonical weight structure on $\Qcoh([\Spec(A)/G])$ whose heart is the full subcategory spanned by objects of the form $p^*(\sE)$, where $p : [\Spec(A)/G] \to BG$ is the projection and $\sE \in \Qcoh(BG)$ is a finite projective $G$-equivariant $R$-module.
  \end{exam}

  \begin{prop}\label{nilinvariance}
    Let $R$ be a commutative ring and $G$ an embeddable linearly reductive group scheme over $R$.
    For any $G$-equivariant nilpotent extension $A \to B$ of connective $\sE_{\infty}$-algebras over $R$ with $G$-actions 
    (i.e., a $\pi_0$-surjection with nilpotent kernel), the induced map
    $$ \K_{-n}([\Spec(A)/G]) \to \K_{-n}([\Spec(B)/G]) $$
    is an isomorphism for every $n\ge 0$.
  \end{prop}

  \begin{proof}
   \examref{exam:rfmalrcdsc} together with \lemref{lem:Vect nil-invariance} shows that the morphism $A \to B$ induces an equivalence of the homotopy categories of the heart of the weight structures on the \inftyCats $\Qcoh([\Spec(A)/G])$ and $\Qcoh([\Spec(B)/G])$.
   The required isomorphism on negative K-groups therefore follows from \cite[Theorem~4.3]{Vovastheoremoftheheart}.
  \end{proof}

  This further implies the following nil-invariance statement for a large class of stacks, which will be crucial for the proof of Weibel's conjecture.

  \begin{cor}\label{nilinvariance_general}
    Let $\sX$ be a ANS derived stack of Nisnevich cohomological dimension $d$.
    Then for any surjective closed immersion $i : \sX' \to \sX$, the map
      $$i^* : \K_{-n}(\sX) \to \K_{-n}(\sX')$$
    is an isomorphism for every $n > d$.
  \end{cor}
  \begin{proof}
    By \corref{cor:K Nis}, we may regard $\K(-)$ and $\K(- \times_{\sX} \sX')$ as Nisnevich sheaves of spectra on the small \'{e}tale site of $\sX$.
    Then the fibre $\sF$ of the morphism $\K(-) \to \K(- \times_\sX \sX')$ is also a Nisnevich sheaf.
    The claim is that $\pi_{-n}(\sF(\sX)) = 0$ for $n>d$.
    Considering the descent spectral sequence
    $$\on{H}^p_{\Nis}(\sX, \pi_q^{\Nis}(\sF)) \Rightarrow \pi_{q-p}(\sF(\sX)),$$
    it will suffice to show that the left-hand side is trivial for $q<0$ and for $p>d$.
    By \thmref{thm:ANS local structure} and \propref{prop:afsoub01}, there exists an embeddable linearly reductive group scheme $G$ over an affine scheme $S$, and affine derived schemes $U_i$ over $S$ with $G$-action, together with étale morphisms $[U_i/G] \to \sX$ which generate a Nisnevich covering.
    By \propref{nilinvariance}, the spectrum $\sF([U_i/G])$ is connective for all $i$.
    Hence the homotopy sheaves $\pi^{\Nis}_q(\sF)$ vanish for $q<0$.
  \end{proof}

\ssec{Killing by blow-ups}

  The killing lemma, proven in \cite[Prop.~5]{KerzStrunkKH}, says that for any negative K-theory class one can find a suitable (sequence of) blow-ups along which the inverse image vanishes.
  It was generalized to stacks, using Rydh's generalization of Raynaud--Gruson flatification, in \cite[Prop.~7.3]{hoyoiskrishna}.
  The following is a ``with supports'' variant of the killing lemma, which we will require for our proof of the Weibel conjecture.

  \begin{prop}\label{prop:fklkdajjk}
    Let $f: \sX \to \sY$ be a smooth morphism of finite type where $\sY$ is a reduced noetherian ANS stack and $\sX$ satisfies the resolution property.
    Let $\sZ \sub \sX$ be a closed substack. 
    Then for any integer $i > 0$ and any class $\alpha \in \K_{-i}(\sX ~\text{on}~ \sZ)$, there exists a sequence of blow-ups $q : \sY' \to \sY$ with nowhere dense centres such that $q_\sX^*(\alpha) = 0$ in $\K_{-i}(\sX' ~\text{on}~ \sZ')$, where $\sX' := \sX \fibprod_\sY \sY'$, $\sZ' := \sZ \fibprod_\sX \sX'$, and $q_\sX: \sX' \to \sX$ is the projection.
  \end{prop}

  The proof will use the following standard lemma.
  For a noetherian algebraic stack $\sX$, let $\Coh(\sX)$ be the derived \inftyCat of coherent complexes on $\sX$ (see e.g. \cite[Defn.~1.5]{asfnjw} where it is denoted $\bD_{\mrm{coh}}(\sX)$).
  Given a closed substack $\sZ$, $\Coh(\sX~\mrm{on}~\sZ)$ denotes the full subcategory spanned by complexes supported set-theoretically on $\abs{\sZ}$.

  \begin{lem} \label{lem:lsdkdfj}
    Let $\sX$ be a noetherian algebraic stack and $\sZ$ a closed substack.
    Then for any open immersion $j: \sU \to \sX$, the restriction functor
    $$j^* :\Coh(\sX~\mrm{on}~\sZ) \to \Coh(\sU~\mrm{on}~\sZ \cap \sU)$$
    is essentially surjective.
  \end{lem}
  \begin{proof}
    Since $j^*$ is t-exact and the t-structures are bounded, it is enough to show that it induces an essentially surjective functor on the hearts.
    That is, it is enough to show that every coherent sheaf $\sF$ on $\sU$ supported on $\sZ\cap\sU$ extends to a coherent sheaf $\widetilde{\sF}$ on $\sX$ supported on $\sZ$. 
    Let $i : \sZ \to \sX$ and $i_\sU : \sZ\cap\sU \to \sU$ denote the inclusions.
    We may write $\sF \simeq i_{\sU,*} (\sG)$ for some coherent sheaf $\sG$ on $\sZ\cap\sU$.
    Now $\sG$ can be extended to a coherent sheaf $\widetilde{\sG}$ on $\sZ$ (see e.g. \cite[Cor.~15.5]{laumon2018champs}).
    By the base change formula, $i_*(\widetilde{\sG})|_\sU \simeq i_{\sU,*}(\widetilde{\sG}|_{\sZ\cap\sU}) \simeq \sF$.
    That is, $\widetilde{\sF} = i_*(\widetilde{\sG})$ is an extension of $\sF$ as desired.
  \end{proof}

  \begin{proof}[Proof of \propref{prop:fklkdajjk}]
    By inductive application of the Bass fundamental theorem (see e.g. \cite[Thm.~2.21]{asfnjw}, which holds with supports just as in \cite[Thm.~7.5]{ThomasonTrobaugh}), there is a canonical surjection
    $$
    \Coker (\K_0 (\sX \times \A^i ~\text{on}~ \sZ \times \A^i) \xrightarrow{j^*}  \K_0 (\sX \times \G_m^i ~\text{on}~ \sZ \times \G_m^i)) \to \K_{-i}(\sX ~\text{on}~ \sZ).
    $$
    Therefore it suffices to show that for any $\alpha \in \K_0 (\sX \times \G_m^i ~\text{on}~ \sZ \times \G_m^i)$, there exists a sequence of blow-ups $q: \sY' \to \sY$ with nowhere dense centers such that the image of $\alpha$ in $\K_0 (\sX' \times \G_m^i ~\text{on}~ \sZ' \times \G_m^i)$ lifts to a class in $\K_0 (\sX \times \A^i ~\text{on}~ \sZ \times \A^i)$.
    By definition, we may write $\alpha = [\sF]$ where $\sF$ is a perfect complex on $\sX \times \G_m^i$ supported on $\sZ \times \G_m^i$.

    By \lemref{lem:lsdkdfj}, we can extend $\sF$ to some $\sE \in \Coh(\sX \times \A^i~\mrm{on}~\sZ \times \A^i)$.
    Since $\sX \times \A^i$ has the resolution property (as it is affine over $\sX$), we may assume that $\sE$ is represented by a chain complex $\sE_\bullet$ of finite locally free sheaves with $\sE_n = 0$ for $n\ll 0$. 
    Since its restriction $\sF \simeq j^*(\sE)$ is perfect, say of Tor-amplitude $\le a$, we may replace $\sE$ by its truncation $\tau_{\le a}(\sE)$ (which still restricts to $\sF$) so that it may be represented by a chain complex $\sE_\bullet$ with the following properties:
    \begin{enumerate}
      \item $\sE_n = 0$ for $n\ll 0$ or $n>a$;
      \item $\sE_n$ is finite locally free for all $n<a$;
      \item $\sE_a$ is coherent and $j^*(\sE_a)$ is finite locally free.
    \end{enumerate}
    
    By Rydh's stacky generalization of Raynaud--Gruson (see \cite[Thm.~4.2]{rydhflatification}), we can argue as in the proof of \cite[Prop.~7.3]{hoyoiskrishna} to produce a sequence of blow-ups $q: \sY' \to \sY$ such that the strict transform $\sE'_a$ of $\sE_a$ on $\sX' \times \A^i$ is flat over $\sX'$.
    For every $n$, the strict transform $\sE'_n$ of $\sE_n$ on $\sX' \times \A^i$ is the cokernel of the inclusion $\sH^0_{\sD\fibprod_\sY\sX\times\A^i}(q_\sX^*(\sE_n)) \hook q_\sX^*(\sE_n)$, where $\sD \sub \sY'\times\A^i$ is the exceptional divisor.
    Thus we may regard $\sE'_\bullet$ as a chain complex with the following properties:

    \begin{enumerate}[label={(\alph*)}]
      \item
      $\sE'_a$ is of finite tor-amplitude on $\sX' \times \A^i$, since it is flat over $\sY'$ (see e.g. \cite[Lem.~7.2]{hoyoiskrishna}).

      \item
      For every $n<a$ we have $\sE'_n = q_\sX^*(\sE_n)$, since $\sE_n$ is already flat over $\sX' \times \A^i$.

      \item
      For every $n$, we have $\sE'_n|_{\sX'\times\G_m^i} = q_\sX^*(\sF_n)$, because $\sF_n$ is already flat over $\sX'$.
    \end{enumerate}

    Since each term of $\sE'_\bullet$ is of finite Tor-amplitude on $\sX'\times\A^i$, $\sE'_\bullet$ represents a perfect complex $\sE' \in \Perf(\sX'\times\A^i)$.
    Since the chain complex $q_\sX^* (\sE_\bullet)$ has homology supported on $\sZ' \times \A^i$, the same holds for its quotient $\sE'_\bullet$, hence in particular $\sE' \in \Perf(\sX'\times\A^i~\mrm{on}~\sZ' \times \A^i)$.
    Finally, since $\sE'|_{\sX'\times\G_m^i} \simeq q_\sX^*(\sF)$ in $\Perf(\sX'\times\G_m^i~\mrm{on}~\sZ' \times \G_m^i)$, we find that the class $[\sE']\in\K_0 (\sX' \times \A^i ~\mrm{on}~ \sZ' \times \A^i)$ lifts $q_\sX^*(\alpha)$ as claimed.
  \end{proof} 

\ssec{Weibel's conjecture (I)}

  \begin{thm}\label{weibel}
    Let $\sX$ be a noetherian ANS stack of fppf-covering dimension $d$ (see \defref{defn:dim}).
    Then the negative K-groups $\K_{-i}(\sX~\mrm{on}~\sY)$ vanish for all $i>d$.
  \end{thm}

  \begin{lem}
    Let $\sX$ be a reduced noetherian ANS stack with the resolution property of fppf-covering dimension $d$ (see \defref{defn:dim}).
    Then for any closed substack $\sY \sub \sX$, the negative K-groups $\K_{-i}(\sX~\mrm{on}~\sY)$ vanish for all $i>d$.
  \end{lem}
  \begin{proof}
    We argue by induction on $d$.
    For any element $\gamma \in \K_{-i}(\sX~\mrm{on}~\sY)$, there exists by \propref{prop:fklkdajjk} a sequence of blow-ups $f: \sX' \to \sX$ with nowhere dense centers such that $f^*(\gamma) = 0$ in $\K_{-i}(\sX'~\mrm{on}f^{-1}(\sY))$.
    This fits in a proper cdh square
    $$
    \begin{tikzcd}
    \sZ' \arrow[r]\arrow[d]& \sX'\arrow[d, "f"]\\
    \sZ\arrow[r]& \sX
    \end{tikzcd}
    $$
    where $\sZ \sub \sX$ is any nowhere dense closed substack for which $f$ is an isomorphism over $\sX\setminus\sZ$.
    Let $\sZ(n)$ and $\sZ'(n)$ denote the $n$th infinitesimal thickenings of $\sZ$ and $\sZ'$, respectively.
    By \thmref{thm:K formal proper excision} (and \remref{rem:ashfu01bp}) and Remark~\ref{rem:formal completion classical}, we get a long exact sequence
    $$\cdots \to \{\K_{-i+1}(\sZ'(n))~\mrm{on}~f^{-1}(\sY)\}_n \to \K_{-i}(\sX~\mrm{on}~\sY) \to \{\K_{-i}(\sZ(n))~\mrm{on}~\sY\}_n \oplus \K_{-i}(\sX'~\mrm{on}~f^{-1}(\sY)) \to \cdots$$
    of pro-abelian groups.
    Now $\sZ(n)$ and $\sZ'(n)$ are reduced noetherian ANS stacks satisfying the resolution property (\lemref{lem:rep over ANS} and \lemref{lem:pisnfipj}) 
    and of fppf-covering dimension $<d$, so $\{\K_{-i+1}(\sZ'(n))~\mrm{on}~f^{-1}(\sY)\}_n$ and $\{\K_{-i}(\sZ(n))~\mrm{on}~\sY\}_n$ both vanish by the induction hypothesis.
    It follows that $f^*: \K_{-i}(\sX) \to \K_{-i}(\sX')$ is injective and hence that $\gamma = 0$.
  \end{proof}

  \begin{proof}[Proof of \thmref{weibel}]
    We again argue by induction on $d$.
    Suppose $d=0$.
    By \corref{cor:K Nis} and \thmref{thm:ANS perfect}, there is a convergent Nisnevich-descent spectral sequence:
    $$\on{H}^p_{\Nis}(\sX,\pi_q^{\Nis}(\K)) \Rightarrow \K_{q-p}(\sX),$$
    where $\pi_q^{\Nis}(\K)$ denotes the Nisnevich sheaf associated with $\K_q$.
    It follows from \propref{coh_dimension} and the previous case that $\on{H}^p_{\Nis}(\sX,\pi_q^{\Nis}(\K))$ vanishes for
    all $q<0$ and $p>0$ and therefore also $\K_{-i}(\sX) = 0$ for $i > 0$.

    Now suppose $d \ge 1$.
    By \thmref{thm:ANS local structure} and \propref{prop:X/G res prop}, there is a finite sequence of open immersions
     $\initial = \sU_0
      \hook \sU_1
      \hook \cdots
      \hook \sU_n = \sX$, and Nisnevich squares
      \[
      \begin{tikzcd}
        \sW_j \ar{r}\ar{d}
          & \sV_j \ar{d}
        \\
        \sU_{j-1} \ar{r}
          & \sU_j,
      \end{tikzcd}
    \]
    where each $\sW_j$ and $\sV_j$ satisfy the resolution property and have fppf-covering dimension $\le d$.
    We prove by induction on $j$ that for $i > d$, $\K_{-i}(\sU_j)$ vanishes.
    For $j = 0$, this follows from the previous case.
    By \propref{coh_dimension} and Corollary~\ref{nilinvariance_general}, we may assume that $\sU_j$ is reduced.
    Choose $\gamma \in \K_{-i}(\sU_j)$.
    By induction on $j$, and the previous case
    for stacks with resolution property, the groups
    $\K_{-l}(\sW_j)$, $\K_{-l}(\sV_j)$ and $\K_{-l}(\sU_{j-1})$
    vanish for all $l > d$.
    By Nisnevich descent (\corref{cor:K Nis}),
    we have a long exact sequence:
    $$
    \cdots \to \K_{-i+1}(\sW_j) \xrightarrow{\partial} \K_{-i}(\sU_j) \to \K_{-i}(\sU_{j-1}) \oplus \K_{-i}(\sV_j) \to \cdots.
    $$
    By induction hypothesis on $j$, we deduce that $\gamma = \partial(\alpha)$ for some $\alpha \in \K_{-i+1}(\sW_j)$.
    By applying the killing lemma (\propref{prop:fklkdajjk}) to the \'{e}tale morphism $\sW_j \to \sU_j$,
    we can find a sequence of blow-ups $f: \sU'_j \to \sU_j$ with nowhere dense centers
    such that for the induced map $f_W: \sW'_j := \sU'_j \times_{\sU_j} \sW_j \to \sW_j$, $f_W^*(\alpha) = 0$ in $\K_{-i+1}(\sW'_j)$.
    Since $\sU_j'$ is again a noetherian ANS stack (\lemref{lem:rep over ANS}), we conclude that $f^*(\gamma) = f^*(\partial(\alpha)) = \partial(f_W^*(\alpha)) = 0$.
    Now as in the first case, by using \thmref{thm:K formal proper excision} and the induction hypothesis on $d$, we conclude that $\gamma = 0$.
  \end{proof}

\ssec{Weibel's conjecture (II)}

  \begin{thm}\label{thm:weibel2}
    Let $\sX$ be a noetherian ANS stack of smooth-covering dimension $d$ (see \defref{defn:dim}). Then for any vector bundle $\pi: \sE \to \sX$, the map
    $$ \pi^*: \K_{-i}(\sX) \to \K_{-i}(\sE)$$
    is an isomorphism for all $i \ge d$.
  \end{thm}

  \begin{proof}
    The zero section induces a retraction of $\pi^*: \K(\sX) \to \K(\sE)$, so it is enough to show that $\pi^*$ is surjective.
    If $\sF$ denotes the cofiber of the morphism $\pi^*: \K(-) \to \K(- \times_{\sX} \sE)$ on the small \'{e}tale site of $\sX$, then it suffices to show that $\sF_{-i}(\sX)$ vanishes for all $i \ge d$.
    By \propref{nilinvariance}, $\sF_{-i}$ is nil-invariant for all $i \ge 0$ for any $[X/G]$, where $X$ is an affine scheme and
    $G$ is an embeddable linearly reductive group scheme.
    Therefore by \corref{cor:K Nis}, \thmref{thm:ANS local structure} and \propref{prop:afsoub01}, $\sF_{-i}$ is nil-invariant for all $i \ge d$ (arguing as in the proof of \corref{nilinvariance_general}).
    Thus we can assume that $\sX$ is reduced.
    
    Suppose $d=0$.
    Then there exists a smooth surjection $u: X \to \sX$ by a $0$-dimensional noetherian scheme $X$.
    Since $\sX$ is reduced, $X$ is also reduced and hence regular.
    Since perfectness is fppf-local, we find that every cohomologically bounded pseudocoherent complex on $\sX$ is perfect (see e.g. the proof of \cite[Lem.~5.6]{hoyoiskrishna}), and the same for $\sE$.
    Thus the map $\pi^* : \K(\sX) \to \K(\sE)$ is identified with the inverse image $\pi^* : \on{G}(\sX) \to \on{G}(\sE)$, which is invertible by \cite[Thm.~3.5]{asfnjw}.
    Now assume $d > 0$ and the statement is known for smooth-covering dimension $< d$.
  
  \emph{Case 1: $\sX$ has the resolution property.}
    If $\sX$ has the resolution property, then by \lemref{lem:rep over ANS}, \thmref{thm:ANS perfect} and \lemref{lem:pisnfipj}, $\sE$ is a perfect stack with the resolution property.
    By \propref{prop:fklkdajjk}, for any $\gamma \in \K_{-d}(\sE)$, there exists a sequence of blow-ups $f: \sX' \to \sX$ with nowhere dense centers such that $\gamma$ goes to $0$ in $\K_{-d}(\sX' \times_{\sX} \sE)$.
    Choose $\sZ \sub \sX$ a nowhere dense closed substack such that $f$ is an isomorphism over $\sX\setminus\sZ$
    and let $\sZ' = \sZ \times_{\sX} \sX'$. 
    The $n$th infinitesimal thickenings $\sZ(n)$ satisfy the induction hypothesis by Lemmas~\ref{lem:rep over ANS} and \ref{lem:pisnfipj}.
    Combining formal proper excision (\thmref{thm:K formal proper excision}) with Remark~\ref{rem:formal completion classical}, the isomorphism $\K_{-d}(\sE \times_{\sX}\sZ(n)) \simeq \K_{-d}(\sZ(n)) \simeq 0$ (by induction hypothesis), and \thmref{weibel}, we get a commutative diagram with exact rows:
    \[
      \begin{tikzcd}
        \{\K_{-d+1}(\sZ'(n))\}_n \ar{r} \ar{d}{\pi_{Z'}^*} & \K_{-d}(\sX) \ar{r} \ar{d}{\pi^*} &  \K_{-d}(\sX') 
        \ar{d}{\pi_{Z}^*}\\
        \{\K_{-d+1}(\sE \times_{\sX} \sZ'(n))\}_n \ar{r} & \K_{-d}(\sE) \ar{r} & \K_{-d}(\sE \times_{\sX} \sX').
      \end{tikzcd}
    \]
    Since $\pi_{Z'}^*$ is also an isomorphism by induction hypothesis, this implies that $\gamma$ is in the image of $\pi^*$.

  \emph{Case 2: $\sX$ is arbitrary.}
    In general, there exists by \thmref{thm:ANS local structure} and \propref{prop:X/G res prop} a finite sequence of open immersions
     $\initial = \sU_0
      \hook \sU_1
      \hook \cdots
      \hook \sU_n = \sX$ and Nisnevich squares
    \[
      \begin{tikzcd}
        \sW_j \ar{r}\ar{d}
          & \sV_j \ar{d}{p} 
        \\
        \sU_{j-1} \ar{r}
          & \sU_j,
      \end{tikzcd}
    \]
    where each $\sV_j$ satisfies the resolution property.
    We prove by induction on $j$ that for $i \ge d$, $\sF_{-i}(\sU_j)$ vanishes.
    For $j = 0$, this follows from Case 1.
    Assuming the groups
    $\sF_{-l}(\sU_{j-1})$
    vanish for all $l \ge d$, we will show that $\sF_{-i}(\sU_j)$ also vanishes.
    
    Choose $\gamma \in \sF_{-i}(\sU_j)$ for some $i \ge d$.
    We will show that $\gamma = 0$.
    By localization for $\sF$, we get an exact sequence of homotopy groups 
    \[
      \sF_{-i}(\sU_j ~\text{on}~ \sZ_j) \rightarrow \sF_{-i}(\sU_j) \to \sF_{-i}(\sU_{j-1}),
    \]
    where $\sZ_j$ denotes the (reduced) complement of the open substack $\sU_{j-1} \subseteq \sU_j$.
    Since $\sF_{-i}(\sU_{j-1})$ vanishes by induction hypothesis, $\gamma$ lifts to a class $\sF_{-i}(\sU_j ~\text{on}~ \sZ_j)$ (which we also denote $\gamma$).
    Let $\widetilde{\gamma}$ denote its image by $p^*: \sF_{-i}(\sU_j ~\text{on}~ \sZ_j) \to \sF_{-i}(\sV_j ~\text{on}~ p^{-1}(\sZ_j))$.
    
    By applying \propref{prop:fklkdajjk}\footnote{%
      To be precise, we use the statement for $\sF$ in place of $\K$, which holds since there is a natural splitting $\K(- \fibprod_\sX \sE) \simeq \K(-) \oplus \sF(-)$.
    } to $\widetilde{\gamma}$ and the smooth morphism $\sV_j \to \sU_j$, we can find a sequence of blow-ups $q: \sU'_j \to \sU_j$ with nowhere dense centres such that $\widetilde{\gamma}$ vanishes after inverse image along the base change $q_\sV : \sV'_j := \sU'_j \times_{\sU_j} \sV_j \to \sV_j$.
    The cartesian square
    \[\begin{tikzcd}
      \sV'_j \ar{r}{q_\sV}\ar{d}
      & \sV_j \ar{d}{p}
      \\
      \sU'_j \ar{r}{q}
        & \sU_j
    \end{tikzcd}\]
    gives rise to a commutative square
    \[\begin{tikzcd}
      \sF_{-i}(\sU_j ~\mrm{on}~ \sZ_j) \ar{r}\ar{d}{q^*}
      & \sF_{-i}(\sV_j ~\mrm{on}~ \sZ_j) \ar{d}{q_\sV^*}
      \\
      \sF_{-i}(\sU'_j ~\mrm{on}~ \sZ_j) \ar{r}
      & \sF_{-i}(\sV'_j ~\mrm{on}~ \sZ_j),
    \end{tikzcd}\]
    where the horizontal arrows are invertible by excision (and we implicitly base change $\sZ_j$ where necessary).
    Therefore, we get that $q^*(\gamma)$ vanishes in $\sF_{-i}(\sU'_j ~\mrm{on}~ \sZ_j)$ and in particular in $\sF_{-i}(\sU'_j)$.

    By construction, $q$ is an isomorphism over $\sU_j \setminus \sD$ for some nowhere dense closed substack $\sD \sub \sU_j$.
    Let $\sD' = \sD \fibprod_{\sU_j} \sU_j'$. 
    Using \thmref{thm:K formal proper excision}, we have an exact sequence of pro-abelian groups
    \[
      \{\sF_{-i+1}(\sD'(n))\}_n
      \to \sF_{-i}(\sU_j)
      \to \sF_{-i}(\sU'_j) \oplus \{\sF_{-i}(\sD(n))\}_n,
    \]
    where $\sF_{-i+1}(\sD'(n))$ and $\sF_{-i}(\sD(n))$ vanish
    as they satisfy the induction hypothesis on $d$ (by Lemmas~\ref{lem:rep over ANS} and \ref{lem:pisnfipj}).
    But since $q^*(\gamma)$ vanishes in $\sF_{-i}(\sU_j')$, we have $\gamma = 0$ in $\sF_{-i}(\sU_j)$ as desired.
  \end{proof}

  \begin{rem}
    The argument in the proof of \thmref{thm:weibel2} can also be used to generalize \thmref{weibel} to any stack $\sX$ that is \emph{smooth} and affine over a noetherian ANS stack of fppf-covering dimension $d$.
  \end{rem}

\appendix

\section{Algebraic stacks}
\label{sec:stacks}

\ssec{ANS stacks}
\label{ssec:ANS}

  \begin{defn}\label{defn:group schemes}
    Let $G$ be an affine fppf group scheme over an affine scheme $S$.
    \begin{defnlist}
      \item
      We say that $G$ is \emph{linearly reductive} if direct image along the morphism $BG \to S$ is t-exact (i.e., it is cohomologically affine).

      \item
      We say that $G$ is \emph{nice} if it is an extension of a finite étale group scheme, of order prime to the characteristics of $S$, by a group scheme of multiplicative type.

      \item
      We say that $G$ is \emph{embeddable} if it is a closed subgroup of $\mrm{GL}_{S}(\sE)$ for some finite locally free sheaf $\sE$ on $S$.
    \end{defnlist}
    Nice group schemes are linearly reductive by \cite[Rem.~2.2]{AlperHallRydhLocal}.
  \end{defn}

  \begin{defn}\label{defn:ANS}
    A derived algebraic stack $\sX$ is called \emph{ANS} if it has affine diagonal and nice stabilizers.
  \end{defn}

  \begin{exam}
    In characteristic zero any reductive group $G$ (such as $\mrm{GL}_{n,S}$) is linearly reductive.
    In characteristic $p>0$, any linearly reductive group is nice \cite[Thm.~18.9]{AlperHallRydhLocal}.
  \end{exam}

  \begin{exam}\label{apdsfni}
    Let $G$ be a finite étale group scheme over a field $k$.
    If $G$ has order prime to the characteristic of $k$, then $G$ is nice and embeddable.
    It follows that any separated Deligne--Mumford stack over $k$ is ANS as long as it is tame (i.e., has all stabilizers of order prime to the characteristic).
  \end{exam}

  \begin{exam}
    Any algebraic stack with affine diagonal that is tame in the sense of \cite[Def.~3.1]{AbramovichOlssonVistoli} is ANS.
    This generalizes \examref{apdsfni}.
  \end{exam}

  \begin{exam}
    Tori are embeddable group schemes of multiplicative type (hence nice).
    Thus if $T$ is a torus over an affine scheme $S$ acting on an algebraic space $X$ over $S$ with affine diagonal, then the quotient $[X/T]$ is ANS.
    (However, it is typically not tame in the sense of \cite{AbramovichOlssonVistoli}.)
  \end{exam}

  \begin{lem}\label{lem:rep over ANS}
    Let $\sX$ be an ANS derived stack.
    Let $f : \sX' \to \sX$ be a representable morphism with affine diagonal.
    Then $\sX'$ is ANS.
  \end{lem}
  \begin{proof}
    Since $f$ is representable, the stabilizers of $\sX'$ are subgroups of those of $\sX$.
  \end{proof}

  The following is the main result of \cite{ahhr} in the classical case.
  The generalization to derived stacks is immediate.

  \begin{thm}[Alper--Hall--Halpern-Leistner--Rydh]\label{thm:ANS local structure}
    Let $\sX$ be an ANS derived stack.
    Then there exists a finite sequence of open immersions
    \[
      \initial = \sU_0
      \hook \sU_1
      \hook \cdots
      \hook \sU_n = \sX,
    \]
    an embeddable nice group scheme $G$ over an affine scheme $S$,
    and Nisnevich squares
    \[
      \begin{tikzcd}
        \sW_i \ar{r}\ar{d}
          & \sV_i \ar{d}
        \\
        \sU_{i-1} \ar{r}
          & \sU_i
      \end{tikzcd}
    \]
    where $\sV_i$ is étale and affine over $\sU_i$ and quasi-affine over $BG$.
  \end{thm}
  \begin{proof}
    This follows by combining \cite[Thm.~6.3]{ahhr}\footnote{%
      See \cite[Cor.~17.3]{AlperHallRydhLocal} and \cite[Thm.~3.2]{AbramovichOlssonVistoli} for documented special cases of this result.
    } with \cite[Prop.~2.9]{hoyoiskrishna}.
    See \cite[Thm.~2.12]{sixstack} for details.
  \end{proof}

  \begin{prop}\label{prop:afsoub01}
    Let $\sX = [X/G]$ be the quotient of a quasi-compact separated derived algebraic space $X$ with action of a nice group scheme $G$ over an affine scheme $S$. 
    Then $\sX$ admits a scallop decomposition of the form $(\sU_i, \sV_i, u_i)_i$, where $\sV_i$ is of the form $[V_i/G]$ for some affine derived schemes $V_i$ over $S$ with $G$-action, and $u_i$ is an affine morphism for each $i$.
  \end{prop}
  \begin{proof}
    By generalized Sumihiro (see \cite[Theorem~2.14]{sixstack}), $\sX$ admits an affine Nisnevich cover $u: \sV \twoheadrightarrow \sX$ where $\sV$ is of the form $[V/G]$ with $V$ an affine scheme over $S$ with $G$-action.
    The desired scallop decomposition is obtained by a $G$-equivariant version of the construction in the proof of \cite[Lem.~5.7.5]{RaynaudGruson} or \cite[Prop.~3.2.2.4]{SAG-20180204}, which goes through \emph{mutatis mutandis}:

    For every $i\ge 0$, define $\sU^i \sub \sX$ as the substack of points where the fibre of $u$ has $\ge i$ geometric points.
    We have $\sU^1 = \sX$ (since $u$ is surjective) and $\sU^{n+1} = \initial$ for some large enough $n$ (since $\sX$ is quasi-compact).
    This gives a finite filtration of $\sX$ by quasi-compact opens $\sU_i := \sU^{n+1-i}$.

    Consider the fibre powers $V^i$ of $V$ over $X$ and $\sV^i = [V^i/G]$ of $\sV$ over $\sX$, respectively.
    Since $u$ is affine, so is each $V^i$.
    Since $V \to X$ is affine and étale, the ``big diagonal'' $\Delta^i \sub V^i$ is an open and closed subscheme.
    The permutation action of the symmetric group $\Sigma_i$ on $V^i$ is free away from $\Delta^i$, and commutes with the factorwise $G$-action on $V^i$.
    Thus we can write
    \[ \sV_i := [(V^i \setminus \Delta^i) / G \times \Sigma_i] \simeq [W_i/G], \]
    where $W_i = [(V^i \setminus \Delta^i)/ \Sigma_i]$, as a quotient of an affine scheme by a free action of a finite group, is an affine scheme.
    Now one checks, exactly as in \cite[Prop.~3.2.2.4]{SAG-20180204}, that the canonical morphisms $\sV_i \to \sX$ factor through affine étale morphisms $u_i : \sV_i \to \sU_i$, and that the resulting construction $(\sU_i, \sV_i, u_i)_i$ is indeed a scallop decomposition.
  \end{proof}

\ssec{The resolution property}

  \begin{defn}
    Let $\sX$ be a derived algebraic stack.
    We say that $\sX$ has the \emph{resolution property} if for every discrete coherent sheaf $\sF$ of finite type on $\sX$, there exists a finite locally free sheaf $\sE$ and a surjection $\sE \twoheadrightarrow \sF$.
  \end{defn}

  The following construction is one of the pleasant consequences of the resolution property.

  \begin{constr}\label{constr:ohuhoj}
    Let $i : \sZ \to \sX$ be a closed immersion of derived stacks.
    If $i$ is almost of finite presentation (e.g. $\sX$ is noetherian), then the ideal $\sI \sub \pi_0(\sO_\sX)$ defining $\sZ_\cl$ in $\sX_\cl$ is of finite type.
    Thus if $\sX$ admits the resolution property, there exists a surjection $\sE \twoheadrightarrow \sI$ from a finite locally free sheaf $\sE$ on $\sX$.
    The induced morphism $s: \sE \to \sI \to \sO_\sX$ can be viewed as a section of the vector bundle
    \[ \bV_\sX(\sE) = \Spec_\sX(\Sym_{\sO_\sX}(\sE)), \]
    and its derived zero locus defines a quasi-smooth closed immersion $\widetilde{i} : \widetilde{\sZ} \to \sX$ whose $0$-truncation is $i : \sZ \to \sX$ and which fits in a homotopy cartesian square
    \[ \begin{tikzcd}
      \widetilde{\sZ} \ar{r}{\widetilde{i}}\ar{d}
        & \sX \ar{d}{s}
      \\
      \sX \ar{r}{0}
        & \bV_\sX(\sE).
    \end{tikzcd} \]
    Repeating this construction with the section $s^{\otimes n}$, for any $n> 0$, gives a tower of infinitesimal thickenings
    \[
      \sZ
      \hook \widetilde{\sZ} = \widetilde{\sZ}(1)
      \hook \widetilde{\sZ}(2)
      \hook \cdots . \]
  \end{constr}

  \begin{rem}\label{rem:qs res prop}
    If $\sX$ is a derived stack with the resolution property, then any quasi-smooth closed immersion $i : \sZ \to \sX$ fits in a homotopy cartesian square
    \[ \begin{tikzcd}
      \sZ \ar{r}{i}\ar{d}
        & \sX \ar{d}{s}
      \\
      \sX \ar{r}{0}
        & \bV_\sX(\sE)
    \end{tikzcd} \]
    where $\sE$ is a finite locally free sheaf on $\sX$.
    This follows by a variant of the proof of \cite[Prop.~2.3.8]{khan2018virtual}.
  \end{rem}

  We discuss some examples of derived stacks with the resolution property.
  First, recall that in the classical setting, the property is stable under affine morphisms:

  \begin{lem}\label{lem:pisnfipj}
    Let $f : \sX \to \sY$ be a quasi-affine morphism of (classical) algebraic stacks.
    If $\sY$ has the resolution property, then so does $\sX$.
  \end{lem}
  \begin{proof}
    See \cite[Lem.~7.1]{Hall_2017} or \cite[Prop.~1.8(v)]{Gross_2017}.
  \end{proof}

  The classifying stack $BG$ has the resolution property for embeddable linearly reductive group schemes $G$, so one finds that affine quotient stacks $[X/G]$ also admit the resolution property (see \cite[Rmk.~2.5]{AlperHallRydhLocal}).
  We prove the following derived generalization of this statement:

  \begin{prop}\label{prop:X/G res prop}
    Let $S$ be an affine scheme, $G$ an embeddable linearly reductive group scheme over $S$, and $X$ a derived affine $S$-scheme with $G$-action.
    Then the derived stack $\sX = [X/G]$ admits the resolution property.
  \end{prop}

  The proof of \propref{prop:X/G res prop} will require the following lemma.
  We write $\D_{\mrm{lfr}}(\sY)$ for the full subcategory of $\Qcoh(\sY)$ spanned by the finite locally free sheaves, for any derived algebraic stack $\sY$.

  \begin{lem}\label{lem:Vect nil-invariance}
    Let the notation be as in \propref{prop:X/G res prop}.
    For any integer $n\ge 0$, let $i : \tau_{\le n}(\sX) \to \sX$ be the inclusion of the $n$-truncation.
    Then we have:
    \begin{thmlist}
      \item\label{item:ouqlnou}
      For any finite locally free $\sE \in \D_{\mrm{lfr}}(\sX)$ and any $\sF \in \Qcoh(\sX)$, the canonical map of $n$-truncated spaces
      \begin{equation*}
        \tau_{\le n}\Maps(\sE, \sF)
        \to \Maps(\tau_{\le n}(\sE), \tau_{\le n}(\sF))
      \end{equation*}
      is invertible.

      \item\label{item:iahfnsn}
      The induced functor of $(n+1)$-categories
      \[ i^* : \tau_{\le n+1}\D_{\mrm{lfr}}(\sX) \to \D_{\mrm{lfr}}(\tau_{\le n}(\sX)) \]
      is an equivalence.
    \end{thmlist}
    In particular, the functor of ordinary categories
    \[ \on{h}\D_{\mrm{lfr}}(\sX) \to \D_{\mrm{lfr}}(\sX_\cl) \]
    is an equivalence (where $\on{h}$ denotes the homotopy category).
  \end{lem}
  \begin{proof}
    Note that the map in the first claim is induced by the morphism in $\Qcoh(\sX)$
    \begin{equation}\label{eq:oaisflas}
      \tau_{\le n}\uHom_{\sO_\sX}(\sE, \sF)
      \to i_*\uHom_{\sO_{\tau_{\le n}(\sX)}}(\tau_{\le n}(\sE), \tau_{\le n}(\sF)),
    \end{equation}
    where $\uHom$ denotes the internal Hom, by applying in succession the functors of direct image along $f : \sX \to BG$ (which is t-exact since $f$ is affine), direct image along $BG \to S$ (which is t-exact since $G$ is linearly reductive), and (derived) global sections (which is t-exact since $S$ is affine).
    Therefore it will suffice to show that \eqref{eq:oaisflas} is invertible.
    By fpqc descent, this can be checked after inverse image along the smooth surjection $X \to \sX$.
    Since $i$ is representable, $i_*$ satisfies base change and we are thus reduced to the affine case, which is well-known (see e.g. \cite[Claim~4.3]{kdescent}).

    Consider now claim~\ref{item:iahfnsn}.
    By \ref{item:ouqlnou} the functor in question is fully faithful (on finite locally frees).
    For essentially surjectivity, we may assume $n=0$ (so that $\tau_{\le n}(\sX) = \sX_\cl$).
    Since $BG$ has the resolution property \cite[Rmk.~2.5]{AlperHallRydhLocal}, \lemref{lem:pisnfipj} implies that for every finite locally free sheaf $\sE \in \Qcoh(\sX_\cl)$, there exists a finite locally free sheaf $\sF \in \Qcoh(BG)$ and a surjection $g^*(\sF) \twoheadrightarrow \sE$ where $g : \sX_\cl \to BG$.
    Certainly $g^* (\sF) \in \Qcoh(\sX_\cl)$ lifts to $f^*(\sF) \in \Qcoh(\sX)$, so we are reduced to show that if $\sE \twoheadrightarrow \sF$ is surjection of locally free sheaves on $\sX_{\cl}$ and $\sE = i^*\widetilde {\sE}$ for some locally free $\widetilde{\sE} \in \Qcoh(\sX)$, then $\sF$ also extends to a locally free sheaf $\widetilde{\sF}$ on $\sX$.
    Since $G$ is linearly reductive, $\sF$ is projective by \cite[Lem.~2.17]{Hoy17} and the surjection $\sE \twoheadrightarrow \sF$ splits.
    The resulting map $e_0 : \sE \twoheadrightarrow \sF \to \sE$ is an idempotent with image $\sF$.
    By claim~\ref{item:ouqlnou}, we can extend $e_0$ to an idempotent endomorphism $e$ of $\widetilde{\sE}$.
    If we set
    \begin{align*}
      \widetilde{\sF} := \colim (\widetilde{\sE} \xrightarrow{e} \widetilde{\sE} \xrightarrow{e} \cdots),\\
      \widetilde{\sF}_1 := \colim(\widetilde{\sE} \xrightarrow{\id-e} \widetilde{\sE} \xrightarrow{\id-e} \cdots),
    \end{align*}
    then the induced morphism $\phi: \widetilde{\sE} \to \widetilde{\sF} \oplus \widetilde{\sF}_1$ is invertible.
    Indeed, by fpqc descent we may pull back to $X$ and thereby assume $\sX$ is affine, in which case it is clear that $\phi$ is an isomorphism on homotopy groups.
    In particular, it follows that $\widetilde{\sF}$ is locally free, and clearly $i^* (\widetilde{\sF}) \simeq \sF$ by construction.
  \end{proof}

  \begin{proof}[Proof of \propref{prop:X/G res prop}]
    By \cite[Rmk.~2.5]{AlperHallRydhLocal}, the classical truncation $[X_\cl/G]$ has the resolution property.
    Hence the claim follows from \lemref{lem:Vect nil-invariance}.
  \end{proof}

\ssec{Compact generation}
\label{ssec:compgen}

  \begin{defn}\label{defn:perfect stack}
    A derived algebraic stack $\sX$ is \emph{perfect} if the stable \inftyCat $\Qcoh(\sX)$ is compactly generated by its full subcategory $\Perf(\sX)$ of perfect complexes.
  \end{defn}

  In this subsection we prove the following result, which is a derived generalization of \cite[Prop.~14.1]{AlperHallRydhLocal}.

  \begin{thm}\label{thm:ANS perfect}
    Let $\sX$ be an ANS derived stack.
    Then $\sX$ is perfect.
  \end{thm}

  \begin{lem}\label{lem:BG perf}
    Let $G$ be an embeddable linearly reductive group scheme over an affine scheme $S$.
    Then the classifying stack $BG$ is perfect.
    Moreover, $\Qcoh(BG)$ is compactly generated by finite representations of $G$ (i.e., finite projectives).
  \end{lem}
  \begin{proof}
    Since $G$ is affine, $BG$ has affine diagonal.
    By \propref{prop:X/G res prop}, $BG$ has the resolution property.
    Hence the claim follows from \cite[Prop.~8.4]{Hall_2017}.
  \end{proof}

  The following shows that, for nice enough quotients of affine derived schemes, the derived \inftyCat is not only compactly generated, but projectively generated in the sense of \defref{defn:weighted pstab}.

  \begin{prop}\label{prop:X/G perf}
    Let $G$ be an embeddable linearly reductive group scheme over an affine scheme $S$.
    Then for any affine derived scheme $X$ with $G$-action, the quotient stack $[X/G]$ is perfect and even \emph{crisp} in the sense of \cite{Hall_2017}.
    Moreover, if $p : [X/G] \to BG$ is the projection, then the collection of objects $\{p^*(\sE)\}$ forms a small set of compact projective generators for $\Qcoh([X/G])_{\ge 0}$, as $\sE \in \Qcoh(BG)$ varies over finite locally free $G$-modules on $S$.
  \end{prop}
  \begin{proof}
    Since $p$ is affine, the functor $p^* : \Qcoh(BG) \to \Qcoh([X/G])$ is compact and generates its codomain under colimits (\ref{rem:f^* compact}).
    This already implies that $\Qcoh([X/G])$ is compactly generated by the objects of the stated form.

    We now demonstrate the stronger property of crispness.
    By \propref{prop:X/G res prop}, $[X/G]$ admits the resolution property.
    Since $[X/G]$ has affine diagonal, \cite[Prop.~8.4]{Hall_2017} then implies that $[X/G]$ is crisp.
    Note that \emph{loc. cit.} only discusses classical stacks, but we only need the argument from the last paragraph of the proof, which immediately generalizes to the derived setting.

    Finally let us show that the object $p^*(\sE) \in \Qcoh([X/G])_{\ge 0}$ is projective for every finite representation $\sE \in \Qcoh(BG)$.
    By \cite[Lem.~7.2.2.6]{HA-20170918} it will suffice to show that the functor $\sMaps(p^*(\sE), -) : \Qcoh(BG) \to \Spt$ is t-exact, where $\sMaps(-,-)$ denotes the mapping spectrum functor in the stable \inftyCat $\Qcoh(BG)$.
    We have a canonical isomorphism
    \[ \sMaps(p^*(\sE), -) \simeq \Gamma\left(S, (p_*\uHom(\sE, -))^G\right). \]
    Note that $\uHom(\sE, -)$ is t-exact because $\sE$ is projective in $\Qcoh(S)$, $p_*$ is t-exact since $p$ is affine, the $G$-invariants functor $(-)^G$ is identified with the direct image functor $\Qcoh(BG) \to \Qcoh(S)$ and hence is t-exact because $G$ is linearly reductive, and the (derived) global sections functor $\Gamma(S, -)$ is t-exact since $S$ is affine.
  \end{proof}

  \begin{proof}[Proof of \thmref{thm:ANS perfect}]
    By \thmref{thm:ANS local structure}, we can in particular find an affine étale surjection onto $\sX$ from a finite coproduct of quotient stacks of the form $[X/G]$, where $G$ is a nice group scheme over an affine scheme $S$ and $X$ is an affine derived $S$-scheme with $G$-action.
    It will now suffice to show that the property of crispness can be detected by affine étale surjections.
    Indeed, we note that the proof given in \cite[Thm.~C]{Hall_2017} for the classical case generalizes to our setting, following Example~9.4 of \emph{loc. cit}.
  \end{proof}

\ssec{Dimension of algebraic stacks}
\label{ssec:dim}

Recall several useful notions of dimension from \cite{hoyoiskrishna}.

\begin{defn}\label{defn:dim}
  Let $\sX$ be a noetherian stack.
  \begin{defnlist}
    \item
    The \emph{Krull dimension} $\dim(\sX)$ of $\sX$ is the dimension of the underlying topological space $\abs{\sX}$.

    \item
    The \emph{blow-up dimension} $\bldim(\sX)$ is the supremum over integers $n\ge 0$ for which there exists a sequence of maps
    $$\sX_n \to \cdots \to  \sX_0 = \sX$$
    such that $\sX_i$ is a nonempty nowhere dense closed substack in an iterated blow-up of $\sX_{i-1}$ for all $i>0$.
    If $\sX$ is empty, then $\bldim(\sX) = -1$ by convention.

    \item
    The \emph{fppf-covering dimension} $\covdimfppf(\sX)$ is the minimal integer $-1 \le n \le \infty$ such that there exists an fppf morphism $X \to \sX$ where $X$ is a noetherian scheme of Krull dimension $n$.

    \item
    The \emph{covering dimension} $\covdimsm(\sX)$ (more precisely, \emph{smooth-covering dimension}) is the minimal integer $-1 \le n \le \infty$ such that there exists a smooth surjection $X \to \sX$ where $X$ is a noetherian scheme of Krull dimension $n$.
  \end{defnlist}
\end{defn}

\begin{rem}
  In general, one has the inequalities $\dim(\sX) \le \bldim(\sX) \le \covdimfppf (\sX) \le \covdimsm(\sX)$.
  For quasi-Deligne--Mumford stacks these are all equal to the usual dimension as defined in \cite[Tag~0AFL]{StacksProject}.
  See \cite[Lemma~7.8]{hoyoiskrishna}.
\end{rem}

\begin{exam}
  Let $G$ be an fppf group scheme over an algebraic space $S$.
  If $G$ acts on an algebraic space $X$ over $S$, then the quotient stack $\sX = [X/G]$ is of fppf-covering dimension $\le \dim(X)$.
\end{exam}

In this subsection we will show that the Nisnevich cohomological dimension of an algebraic stack $\sX$, which we denote by $\on{cd}_{\Nis}(\sX)$,
is bounded by its fppf-covering dimension:

\begin{prop}\label{coh_dimension}
  Let $\sX$ be a noetherian stack and let $\sF$ be a sheaf of abelian groups on the Nisnevich site of $\sX$. Then
  $$
  \on{H}^i_{\Nis}(\sX, \sF) = 0
  $$
  for all $i > \covdimfppf (\sX)$.
  In other words, $\on{cd}_{\Nis} (\sX) \le \covdimfppf (\sX)$.
\end{prop}

Recall that the Nisnevich topology on the category of noetherian algebraic stacks is generated by a cd-structure \cite[Sect.~2C]{hoyoiskrishna}, which is clearly complete and regular.
To establish \propref{coh_dimension} we will show that this cd-structure is bounded with respect to a density structure.

\begin{constr}
  For any noetherian algebraic stack $\sX$, let $\Stk^{\mrm{DM}}_{/\sX}$ be the category of algebraic stacks $\sY$ over $\sX$ for which the structural morphism $\sY \to \sX$ is representable by Deligne--Mumford stacks.
  Choose an fppf covering $\sS \to \sX$ from a Deligne--Mumford stack $\sS$.
  We define a density structure $D^{\sS}_i(-)$ in the sense of \cite[Definition 2.20]{Voe10} on the category $\Stk^{\mrm{DM}}_{/\sX}$.
  For $\sY \in \Stk^{\mrm{DM}}_{/\sX}$, let $q: \sS_{\sY} \to \sY$ be the fppf covering of $\sY$ given by the base change of $\sS \to \sX$ along $\sY$.
  For $i \ge 0$, let $D^{\sS}_i(\sY)$ denote the class
  of open substacks $\sU \hook \sY$ such that the open substack $\sU \times_{\sY} \sS_{\sY} \hook \sS_{\sY}$ defines an element of the class
  $D_i(\sS_{\sY})$, where $D_*(-)$ denotes the density structure on the category of Deligne--Mumford stacks defined in \cite[Definition~4.4]{KrishnaOstvaer}.
  It is easy to check that the classes $D^{\sS}_i(-)$ define a density structure which is locally of finite dimension and the dimension of $\sX$ with respect to the density structure is equal to the Krull dimension $\dim(\sS)$.
\end{constr}

Recall that for a Deligne--Mumford stack $\sX$ and $i \geq 0$, $D_i(\sX)$ is the collection of open substacks $\sU \hook \sX$ such that 
for every irreducible, reduced closed substack $\sZ$ of $\sX$ with $\sZ \times_{\sX} \sU$ the empty stack, 
there exists a sequence $\sZ = \sZ_0 \subsetneq \sZ_1 \subsetneq \cdots \subsetneq \sZ_i$ of irreducible, reduced closed substacks of $\sX$.

\begin{lem}\label{lem:intersecting_substacks}
Let $\sX$ be a Deligne--Mumford stack and $\sU \in D_i(\sX)$.
For any open substack $\sV$ of $\sX$, we have $\sU \cap \sV \in D_i(\sV)$.
\end{lem}
\begin{proof}
By \cite[Lemma~4.5]{KrishnaOstvaer} it suffices to show that $\abs{\sU \cap \sV} \in D_i(\abs{\sV})$.
Now we can use the same argument as in the proof of \cite[Lemma~2.5]{Voe10-2}.
\end{proof}

\begin{prop} \label{prop:cd-bdd}
The Nisnevich cd-structure on the category $\Stk^{\mrm{DM}}_{/\sX}$ is bounded with respect to the density structure $D^{\sS}_*(-)$.
\end{prop}

\begin{proof}
We need to show that every Nisnevich square is reducing with respect to the density structure.
Consider a Nisnevich square $Q$ in $\Stk^{\mrm{DM}}_{/\sX}$:

\begin{equation} \label{eqn:Nis-cd}
\begin{tikzcd}
\sW \arrow[hookrightarrow, r, "j_V"] \arrow[d] & \sV \arrow[d, "p"]\\
\sU \arrow[r, hookrightarrow, "j"] & \sY,
\end{tikzcd}
\end{equation}
where $p$ is \'{e}tale, $j$ is an open immersion and the induced morphism
$p^{-1} (\sY \setminus \sU) \to (\sY \setminus \sU)$ is invertible.
Choose $\sW_0 \in D^S_{i-1}(\sW)$, $\sU_0 \in D^S_i(\sU)$ and $\sV_0 \in D^S_i(\sV)$.
To show that the above square $Q$ is reducing with respect to the density structure,
we need to prove that there exists a Nisnevich square $Q'$ in $\Stk^{\mrm{DM}}_{/\sX}$
\begin{equation} \label{eqn:Nis-cd-2}
\begin{tikzcd}
\sW' \arrow[hookrightarrow, r] \arrow[d] & \sV'\arrow[d]\\
\sU' \arrow[r, hookrightarrow] & \sY',
\end{tikzcd}
\end{equation}
and a morphism $Q' \to Q$ such that $\sW' \to \sW$ factors through $\sW' \to \sW_0$, $\sU' \to \sU$ through
$\sU' \to \sU_0$, $\sV' \to \sV$ through $\sV' \to \sV_0$,
and $\sY' \in D^S_i(\sY)$ (see \cite[Definition 2.21]{Voe10}).
Applying Lemma \ref{lem:cd-bdd} to the morphism $j \coprod p$, we can find $\sY_0 \in D^{\sS}_i(\sY)$ such that $j^{-1}(\sY_0) \subseteq \sU_0$ and
$p^{-1}(\sY_0) \subseteq \sV_0$.
Therefore by base changing \eqref{eqn:Nis-cd} along $\sY_0 \hook \sY$ and then replacing $\sY$ by $\sY_0$
we are reduced to the case when $\sU = \sU_0$ and $\sV = \sV_0$ in \eqref{eqn:Nis-cd}. Note that
$\sW_0\times_{\sY} \sY_0$ is in $D_{i-1}^{\sS}(\sU_0 \times_{\sY} \sV_0)$ by Lemma~\ref{lem:intersecting_substacks}.

Let $\sZ = \sW \setminus \sW_0$ and $\sC = \sY \setminus \sU$ and set $\sW' = \sW_0$, $\sU' = \sU$, $\sV' = \sV \setminus \cl_{\sV}(\sZ)$ and
$\sY' = \sY \setminus (\sC \cap \cl_{\sY}(p\circ j_V(\sZ)))$ in \eqref{eqn:Nis-cd-2} to obtain the Nisnevich square $Q'$ with a natural morphism to $Q$
given by inclusions. Now $\sY' \times_{\sY} \sS_{\sY} \to \sS_{\sY} \in D_i(\sS_{\sY})$ by the proof of \cite[Proposition~4.9]{KrishnaOstvaer}.
Therefore $Q' \to Q$ satisfies all the required properties.
\end{proof}

\begin{lem} \label{lem:cd-bdd}
Let $f: \sW \to \sY$ be an \'{e}tale surjection of noetherian stacks.
Then for any $i \ge 0$ and $\sW_0 \in D^{\sS}_i(\sW)$ there exists $\sY_0 \in D^{\sS}_i(\sY)$ such that
$f^{-1}(\sY_0) \subset \sW_0$.
\end{lem}
\begin{proof}
Let $\sS_{\sW} = \sS \times_{\sX} \sW$, $\sS_{\sW_0} = \sS \times_{\sX} \sW_0$ and let $f_S: \sS_{\sW} \to \sS_{\sY}$
denote the base change of $f: \sW \to \sY$ along
the fppf covering $q: \sS_{\sY} \to \sY$ and $\widetilde{q}: \sS_{\sW} \to \sW$  denote the base change of $q$ along $f$. Then by \cite[Lemma~4.7]{KrishnaOstvaer} applied to $f_S$
there exists $\widetilde{\sY} \in D_i(\sS_{\sY})$ such that
$f_S^{-1}(\widetilde{\sY}) \subset \sS_{\sW_0}$. Let $\sY_0 = q(\widetilde{\sY})$, then $\sY_0 \in D_i(\sY)$
since $q$ is open (see \cite[Proposition~5.6]{laumon2018champs}), $\widetilde{\sY} \subseteq \sY_0 \times_{\sY} \sS_{\sY}$ and
$\widetilde{\sY} \in D_i(\sS_{\sY})$.
Moreover $f^{-1}(\sY_0) = f^{-1}(q(\widetilde{\sY})) \subseteq \widetilde{q}(f_S^{-1}(\widetilde{\sY})) \subseteq \widetilde{q}(\sS_{\sW_0}) = \sW_0$.
\end{proof}


\section{Formal stacks}
\label{sec:formal}

\ssec{Formal completion}
\label{ssec:formal}

  We briefly review some elements of formal derived algebraic geometry.
  Good references are \cite[Sect.~2.1]{HLP-h} and \cite{GR-DGIndSch}.
  The following definition is \cite[Def.~2.1.1]{HLP-h}.

  \begin{defn}[Formal completion]\label{defn:formal completion}
    Let $i : \sZ \to \sX$ be a closed immersion of derived stacks with quasi-compact open complement.
    The \emph{formal completion} of $\sX$ in $\sZ$ is the derived prestack $\sX^\wedge_\sZ$ whose $R$-points, for any \scr $R$, are the $R$-points $x : \Spec(R) \to \sX$ which factor \emph{set-theoretically} through the underlying topological space $\abs{\sZ} \subseteq \abs{\sX}$.
    By definition, $i : \sZ \to \sX$ factors as
      \begin{equation*}
        i : \sZ \to \sX^\wedge_\sZ \xrightarrow{\hat{i}} \sX,
      \end{equation*}
    where the first arrow induces an isomorphism on reductions, and the second arrow is a monomorphism.
    When there is no risk of ambiguity, we will write simply $\sX^\wedge$ for $\sX^\wedge_\sZ$.
  \end{defn}

  \begin{rem}\label{rem:formal base change}
    Note that the formal completion $\sX^\wedge_\sZ$ only depends on the underlying topological space $\abs{\sZ}$.
    In particular, if we have a commutative square
    \[ \begin{tikzcd}
      \sZ' \ar{d}{g}\ar{r}{i'}
      & \sX' \ar{d}{f}
      \\
      \sZ \ar{r}{i}
      & \sX
    \end{tikzcd} \]
    where $\abs{\sZ'}$ is the set-theoretic inverse image $f^{-1}(\abs{\sZ})$, then the formal completion $\sX'^\wedge$ is the derived base change of $\sX^\wedge$.
    That is, we have a homotopy cartesian square
    \[ \begin{tikzcd}
      \sX'^\wedge_{\sZ'} \ar[hookrightarrow]{r}\ar{d}{f^\wedge}
        & \sX' \ar{d}{f}
      \\
      \sX^{\wedge}_{\sZ} \ar[hookrightarrow]{r}
        & \sX.
    \end{tikzcd} \]
  \end{rem}

  \begin{rem}\label{rem:formal colim}
    The formal completion of a derived algebraic stack $\sX$ along a closed immersion $i : \sZ \to \sX$ is always an \emph{ind-algebraic stack}.
    More precisely, one has the following canonical isomorphism (see \cite[Prop.~6.5.5]{GR-DGIndSch}):
      \begin{equation*}
        \{\widetilde{\sZ}\}_{\widetilde{\sZ} \to \sX} \to \sX^\wedge_\sZ,
      \end{equation*}
    where the source is the ind-system indexed by the filtered \inftyCat of closed immersions $\widetilde{\sZ} \to \sX$ that induce an isomorphism $\widetilde{\sZ}_\red \simeq \sZ_\red$ on reductions.
    Note that the transition morphisms are surjective closed immersions.
  \end{rem}

  In the affine case, we can give the following more familiar (but less canonical) description.

  \begin{exam}\label{exam:oizchvnawl}
    Let $A$ be a \scr and $I \subseteq \pi_0(A)$ an ideal, and consider the formal completion of $\Spec(A)^\wedge$ in the vanishing locus of $I$.
    Choosing generators $f_1,\ldots,f_m$ for the ideal $I$, there is an equivalence
    \begin{equation*}
      \Spec(A)^\wedge \simeq \{\Spec(A\modmod(f_1^n,\ldots,f_m^n))\}_n.
    \end{equation*}
    See \cite[Proof of Prop.~8.1.2.1]{SAG-20180204} or \cite[Prop.~2.1.2]{HLP-h}.
  
    If $A$ is an ordinary commutative ring and is noetherian, then we recover the classical formal completion, i.e., the formal spectrum of $A$ as an $I$-adic ring: 
    \begin{equation*}
      \Spec(A)^\wedge \simeq \{\Spec(A/(f_1^n,\ldots,f_m^n))\}_n.
    \end{equation*}
    See \cite[Lem.~17.3.5.7]{SAG-20180204}, \cite[Prop.~2.1.4]{HLP-h}, or \cite[Proof of Prop.~6.8.2]{GR-DGIndSch}.
  \end{exam}

  More generally in the presence of the resolution property, we have:

  \begin{lem}\label{lem:pjpfmq}
    Let $i : \sZ \to \sX$ be a closed immersion almost of finite presentation between derived algebraic stacks for which $\sX$ has the resolution property.
    Let $\widetilde{\sZ}(n)$ be as in \constrref{constr:ohuhoj}.
    Then there is an isomorphism of ind-stacks
    \[ \sX^\wedge_\sZ \simeq \{\widetilde{\sZ}(n)\}. \]
  \end{lem}
  \begin{proof}
    This is a simple cofinality argument using the description of the formal completion given in \remref{rem:formal colim}.
  \end{proof}

  \begin{rem}\label{rem:formal completion classical}
    If $\sX$ is a \emph{classical} stack, then the formal completion (in any closed substack) in the sense of \defref{defn:formal completion} coincides with the classical formal completion, as long as $\sX$ is \emph{noetherian}.
    This follows from \examref{exam:oizchvnawl}, see \cite[Cor.~2.1.5]{HLP-h}.
  \end{rem}

  For the reader's convenience, we spell out \lemref{lem:pjpfmq} in the equivariant (quotient stack) case.

  \begin{constr}
    Let $G$ be a group scheme over a commutative ring $R$, and let $A$ be a derived commutative $R$-algebra with $G$-action.
    Let $M$ be a locally free $G$-equivariant $A$-module and $s : M \to A$ a homomorphism of $G$-equivariant $A$-modules.
    The \emph{derived quotient} of $A$ by $s$ is the $G$-equivariant $A$-algebra formed by attaching a cell $s \simeq 0$, i.e., by the cocartesian square in $G$-equivariant \scrs
    \[ \begin{tikzcd}
      \Sym_A(M) \ar{r}{0}\ar{d}{s}
        & A \ar{d}
      \\
      A \ar{r}
        & A\modmod s,
    \end{tikzcd} \]
    where the map $0$ is induced by adjunction from the zero map $M \to A$, and similarly for $s$.

    Similarly, for any $n>0$, we also write $A\modmod s^n$ for the same construction where $s$ is replaced by $s^{\otimes n} : M^{\otimes n} \to A^{\otimes n} \simeq A$ (the $n$-fold derived tensor product taken over $A$).
  \end{constr}

  \begin{lem}\label{lem:equivariant formal completion}
    Let $G$ be a group scheme over a commutative ring $R$ and $A \twoheadrightarrow B$ a homomorphism of $G$-equivariant derived commutative $R$-algebras which is surjective on $\pi_0$.
    Assume that the quotient stack $[\Spec(\pi_0(A))/G]$ admits the resolution property, so that there exists a locally free $G$-equivariant $A$-module $M$ and $s : M \to A$ whose image on $\pi_0$ is equal to the kernel of $\pi_0(A) \twoheadrightarrow \pi_0(B)$.
    Then we have the following presentation of the formal completion:
    \[ [\Spec(A)/G]^\wedge \simeq \{ [\Spec(A\modmod s^n)/G] \}_{n>0}. \]
  \end{lem}

  \begin{lem}\label{lem:classical equivariant formal completion}
    Let $G$ be a group scheme over a commutative ring $R$ and $A \twoheadrightarrow B$ a surjective homomorphism of $G$-equivariant commutative $R$-algebras with kernel $I$.
    Assume that the quotient stack $[\Spec(A)/G]$ admits the resolution property, so that there exists a locally free $G$-equivariant $A$-module $M$ and $s : M \to A$ whose image is $I$.
    Then we have the following presentation of the formal completion:
    \[ [\Spec(A)/G]^\wedge \simeq \{ [\Spec(A/I^n)/G] \}_{n>0}. \]
  \end{lem}

\ssec{Quasi-coherent sheaves}
\label{ssec:formal/qcoh}

  For formal stacks (or more generally ind-stacks) such as the formal completion, there is a natural pro-$\infty$-categorical refinement of the stable \inftyCat of quasi-coherent sheaves.

  \begin{constr}
    If $\sC$ denotes the \inftyCat of derived algebraic stacks and representable morphisms, then by \remref{rem:f^* compact} the assignment $\sX \mapsto \Qcoh(\sX)$ can be regarded as a functor $\sC^\op \to \Presc$ valued in the \inftyCat of presentable \inftyCats and compact colimit-preserving functors.
    Now consider its Ind-extension
      \[ \Ind(\sC)^\op \simeq \Pro(\sC^\op) \to \Pro(\Presc). \]
    Any derived ind-algebraic stack $\sX$ can be regarded an ind-object in $\sC$ and hence gives rise to a canonical pro-\inftyCat which we denote $\Qcohform(\sX)$.
  \end{constr}

  \begin{exam}
    Let $\sX$ be a derived algebraic stack.
    Then $\Qcohform(\sX)$ is the constant pro-\inftyCat $\{\Qcoh(\sX)\}$.
  \end{exam}

  \begin{exam}
    If $\sX$ is a derived ind-algebraic stack represented by a filtered system $\{\sX_n\}_n$, then $\Qcohform(\sX)$ is represented by the cofiltered system $\{ \Qcoh(\sX_n) \}_n$.
  \end{exam}

  \begin{exam}
    Let $A$ be a \scr and $I \subseteq \pi_0(A)$ an ideal, and consider the formal completion $\Spec(A)^\wedge$ in the vanishing locus of $I$.
    Then choosing generators $f_1,\ldots,f_m$ for the ideal $I$, we have an equivalence
      \begin{equation*}
        \Qcohform(\Spec(A)^\wedge) \simeq \{\Qcoh(A\modmod(f_1^n,\ldots,f_m^n))\}_n
      \end{equation*}
    by \examref{exam:oizchvnawl}.
    If $A$ is a noetherian commutative ring, then we have also
    \begin{equation*}
      \Qcohform(\Spec(A)^\wedge) \simeq \{\Qcoh(A/(f_1^n,\ldots,f_m^n))\}_n
    \end{equation*}
    again by \examref{exam:oizchvnawl}.
  \end{exam}


\section{Weak pro-Milnor squares}
\label{sec:weak}

In this section we continue to use the language of weight structures as in \secref{sec:weightstructures}.
For convenience, the term \emph{weighted \inftyCat} will refer to an essentially small stable \inftyCat with a bounded weight structure.
We write $\sC^{w=0}$ for the weight-heart of a weighted \inftyCat $\sC$.

\ssec{Connected invariants}

  \begin{defn}\leavevmode
    \begin{defnlist}
      \item
      We say that a weight-exact functor $f : \sC \to \sD$ between weighted $\infty$-categories is \emph{thickly surjective} 
      if every object $Y \in \sD^{w=0}$ is a direct summand of $f(X)$ for some object $X \in \sC^{w=0}$.
      In other words, if the induced functor $f^*$ on Ind-completions generates its codomain under colimits.
      
      \item
      Let $\sC$ and $\sD$ be weighted \inftyCats and $k\ge 0$ an integer.
      A thickly surjective weight-exact functor $f : \sC \to \sD$ is \emph{$k$-connective} if it induces $k$-connective maps
      \[ \Maps_{\sC^{w=0}}(X,Y) \to \Maps_{\sD^{w=0}}(f(X),f(Y)) \]
      for all $X$ and $Y$ in $\sC^{w=0}$.
      In other words, if the induced functor of $(k+1)$-categories
      \[ \tau_{\le k+1}(\sC) \to \tau_{\le k+1}(\sD) \]
      is fully faithful (and hence an equivalence) when restricted to the weight-hearts.
    \end{defnlist}
  \end{defn}

  \begin{exam}
    If $A \to B$ is a $k$-connective map of connective $\Einfty$-rings, then the induced functor $\Perf_A \to \Perf_B$ is $k$-connective with respect to the weight structures of \examref{exam:panoujnq}.
  \end{exam}

  The following definition is a variant of \cite[Def.~2.5]{LandTamme}.

  \begin{defn}\label{connected_invariant}
    Let $E$ be a spectrum-valued functor on the \inftyCat of small stable \inftyCats.
    We say that $E$ is \emph{connected} if for any $k$-connective functor $\sC \to \sD$ of weighted \inftyCats, the induced map of spectra $E(\sC) \to E(\sD)$ is $(k+1)$-connective.
  \end{defn}

  \begin{exam}\label{ktheory_is_connected}
    Connective K-theory is an example of a connected invariant.
    This follows from \cite[Cor.~5.16]{fontes2018weight} and the fact that plus-construction sends $k$-connected maps of spaces into
    $k+1$-connected maps (cf. \cite[Lem.~2.4]{LandTamme}). 
    Moreover, nonconnective K-theory is also an example by \cite[Thm.~4.3]{Vovastheoremoftheheart}.
  \end{exam}

  \begin{rem}\label{connected_and_connective_invariants}
    Let $E$ be a localizing invariant.
    If $E$ is connected, then it is $1$-connective in the sense of \cite[Def.~2.5]{LandTamme}.
    The converse holds if $E$ commutes with filtered colimits.
    If $E$ does not commute with filtered colimits, one can still show a weaker statement: if $E$ is $2$-connective in the sense of \cite{LandTamme}, then it is connected.
    This follows from the fact that any weighted \inftyCat can be functorially realized as the kernel of a weight-exact localization functor $\Perf_{A(\sC)} \to \Perf_{B(\sC)}$ for some map of connective $\sE_1$-rings $A(\sC) \to B(\sC)$; see the proof of \lemref{lem:proequiv criterion}.
    (See \cite[Sect.~8.1]{BondarkoWeights} or \cite{BondarkoSosnilo} for the theory of weight-exact localizations.)
    Any $n$-connective functor $\sC \to \sD$ clearly induces an $n$-connective map on $A(\sC) \to A(\sD)$, and it also induces an $n$-connective map
    $B(\sC) \to B(\sD)$
    by the universal properties of localization and of truncation.
  \end{rem}

\ssec{Weak pro-Milnor squares}

  The starting point for our definition of weak pro-Milnor squares is the following classical definition (see e.g. \cite[Sect.~4]{ArtinMazur}):

  \begin{defn}
    Let $\{f_n : X_n \to Y_n\}_n$ be a morphism of cofiltered systems of spectra.
    We say that $f$ is \emph{pro-$k$-connective} if the induced map
    \[ \{\tau_{\le k}(X_n)\}_n \to \{\tau_{\le k}(Y_n)\}_n \]
    is an isomorphism in $\Pro(\Spt)$ for every $k$.
    It is a \emph{weak pro-equivalence} if it is pro-$k$-connective for all $k$.
  \end{defn}

  Note that the same definition makes sense for the \inftyCat of $\sE_1$-rings, for instance, in place of $\Spt$.
  The following can be viewed as a many-object generalization (see \examref{exam:panoujnq}).
  A \emph{weighted pro-\inftyCat} is a pro-object in the \inftyCat of weighted \inftyCats and weight-exact functors.

  \begin{defn}
    Let $\{f_n : \sC_n \to \sD_n\}_n$ be a cofiltered system of weight-exact functors between weighted \inftyCats and $k\ge 0$ an integer.
    We say that it is \emph{pro-$k$-connective} if the induced morphism of pro-\inftyCats
    \[ \{\tau_{\le k+1}(\sC_{n}^{w=0})\}_n \to \{\tau_{\le k+1}(\sD_{n}^{w=0})\}_n \]
    is invertible.
    If it is $k$-connective for all $k$, then it is called a \emph{weak pro-equivalence} of weighted pro-\inftyCats.
  \end{defn}

  This enables us to define weak pro-Milnor squares of weighted \inftyCats.

  \begin{defn}\label{defn:weakly pro-precartesian}
    Let $\{\Delta_n\}_n$ be a cofiltered system of commutative squares of weighted \inftyCats and weight-exact functors of the form
    \[ \begin{tikzcd}
      \sA_n \ar{r}{f_n}\ar{d}{p_n}
      & \sB_n \ar{d}{q_n}
      \\
      \sA'_n \ar{r}{g_n}
      & \sB'_n.
    \end{tikzcd} \]
    For every $n$, let $\sA_n^+ \sub \sA'_n \times_{\sB'_n} \sB_n$ denote the full subcategory of the pullback (taken in the \inftyCat of weighted \inftyCats) generated under finite colimits, finite limits, and retracts by the essential image of $\sA_n \to \sA'_n \times_{\sB'_n} \sB_n$.
    Note that $\sA_n^+$ inherits a weight structure from $\sA'_n \times_{\sB'_n} \sB_n$.
    We say that $\Delta$ is \emph{$k$-pro-precartesian} if the functors $\sA_n \to \sA_n^+$ induce a pro-$k$-connective functor on weighted pro-\inftyCats.
    We say it is \emph{weakly pro-precartesian} if it is $k$-pro-precartesian for all $k$.
  \end{defn}

  \begin{defn}
    Let $\{\Delta_n\}_n$ be a cofiltered system of commutative squares of weighted \inftyCats and weight-exact functors as above.
    We say that $\{\Delta_n\}_n$ is \emph{$k$-pro-Milnor} if it is pro-precartesian and each of the functors $f_n^*$, $g_n^*$, $p_n^*$, $q_n^*$ is thickly surjective.
    It is \emph{weakly pro-Milnor} if it is $k$-pro-Milnor for all $k$.
  \end{defn}

  \begin{constr}\label{constr:aobfuqp}
    Suppose given a commutative square
    \[ \begin{tikzcd}
      \sA \ar{r}{f}\ar{d}{p}
      & \sB \ar{d}{q}
      \\
      \sA' \ar{r}{g}
      & \sB'
    \end{tikzcd} \]
    of weighted \inftyCats.
    Write $\widehat{\sA}$, $\widehat{\sB}$, etc. for the Ind-completions and consider the $\odot$-construction $\widehat{\sQ} := \widehat{\sA'} \odot_{\widehat{\sA}}^{\widehat{\sB'}} \widehat{\sB}$.
    By \lemref{lem:odot t-structure} there is a weight structure on $\sQ := \widehat{\sQ}^\omega$ such that all the functors in the induced square
    \[ \begin{tikzcd}
      \sA \ar{r}{f}\ar{d}{p}
      & \sB \ar{d}{q_0}
      \\
      \sA' \ar{r}{g_0}
      & \sQ,
    \end{tikzcd} \]
    as well as the functor $b : \sQ \to \sB'$, are weight-exact.
    We call $\sQ$ the \emph{weighted $\odot$-construction}, and denote it by
    \[ \sA' \odot_{\sA}^{\sB'} \sB. \]
  \end{constr}

  \begin{defn}
    Let $\{\Delta_n\}_n$ be a cofiltered system of commutative squares of weighted \inftyCats and weight-exact functors of the form
    \begin{equation}\label{eq:weak pro Milnor}
      \begin{tikzcd}
        \sA_n \ar{r}{f_n}\ar{d}{p_n}
        & \sB_n \ar{d}{q_n}
        \\
        \sA'_n \ar{r}{g_n}
        & \sB'_n.
      \end{tikzcd}
    \end{equation}
    If $\Delta$ is a pro-$k$-Milnor square, then we say it \emph{pro-$k$-base change} if the functors $\sA'_n \odot_{\sA_n}^{\sB'_n} \sB_n \to \sB'_n$ induce a pro-$k$-connective functor on weighted pro-\inftyCats.
    If $\Delta$ is a weak pro-Milnor square, then we say it \emph{weakly satisfies pro-base change} if satisfies pro-$k$-base change for all $k$.
  \end{defn}

\ssec{Weak pro-excision}\label{weightedproexcision}

For connected invariants, we have the following analogue of \thmref{thm:pro excision}:

\begin{thm}[Weak pro-excision]\label{thm:proexcision_stcat}
  Let $E$ be a connected localizing invariant.
  Suppose given a weak pro-Milnor square of weighted pro-\inftyCats of the form \eqref{eq:weak pro Milnor} weakly satisfying pro-base change.
  Then the induced square of pro-spectra
  \[ \begin{tikzcd}
    \{E(\sA_n)\}_n \ar{r}\ar{d}
    & \{E(\sB_n)\}_n \ar{d}
    \\
    \{E(\sA'_n)\}_n \ar{r}
    & \{E(\sB'_n)\}_n
  \end{tikzcd} \]
  is weakly pro-cartesian.
  That is, the morphisms $E(\sA_n) \to E(\sA'_n) \fibprod_{E(\sB'_n)} E(\sB_n)$ induce a weak pro-equivalence of pro-spectra.
\end{thm}

\begin{lem}\label{lem:weakstrongconnectivity}
  Let $\{f_n : \sC_n \to \sD_n\}_n$ be a cofiltered system of weight-exact functors between idempotent-complete weighted \inftyCats.
  If it is pro-$k$-connective, then it is isomorphic to a pro-system of levelwise $k$-connective functors.
\end{lem}
\begin{proof}
  For every $n$, consider the commutative square
  \[ \begin{tikzcd}
    \sC_n \ar{r}{f_n}\ar{d}
    & \sD_n \ar{d}
    \\
    \tau_{\le k+1}(\sC_n) \ar{r}{f_n}
      & \tau_{\le k+1}(\sD_n)
  \end{tikzcd} \]
  By assumption, the lower arrow induces an isomorphism of pro-objects as $n$ varies.
  Since passage to underlying pro-objects commutes with finite limits, it follows that the base changes $\sE_n \to \sD_n$ also induce an isomorphism of pro-objects.
   Thus it will suffice to show that for every $n$, the induced functor $f'_n: \sC_n \to \sE_n$ is $k$-connective.
  By construction of $\sE_n$, the functor $\sE_n \to \tau_{\le k+1}(\sC_n)$ induces an equivalence on homotopy categories of the weight-hearts.
  Since the same holds for $\sC_n \to \tau_{\le k+1}(\sC_n)$, it follows that $\sC_n \to \sE_n$ is thickly surjective.

  Now for any two objects $X$ and $Y$ in $\sC_n^{w=0}$, consider the commutative triangle
  \[ \begin{tikzcd}
    \Maps_{\sC_n}(X,Y) \arrow[r]\arrow[dr]
    & \tau_{\le k}\Maps_{\sE_n}(f'_n(X),f'_n(Y))\arrow[d] \\
    & \tau_{\le k}\Maps_{\sC_n}(X,Y).
  \end{tikzcd} \]
  Note that the vertical and diagonal maps have fibres
  \[
    \tau_{\ge k+1} \Maps_{\sD_n}(f_n(X), f_n(Y))
    ~\text{and}~
    \tau_{\ge k+1} \Maps_{\sC_n}(X, Y),
  \]
  respectively.
  These are both $(k+1)$-connective, so it follows from the octahedral axiom that the horizontal map also has $k$-connective fibre.
\end{proof}

\begin{cor}\label{cor:proconnected}
  Let $E$ be a connected invariant.
  Then $E$ sends pro-$k$-connective maps of idempotent-complete weighted \inftyCats to $(k+1)$-connective maps of pro-spectra.
  In particular, it sends weak pro-equivalences to weak pro-equivalences.
\end{cor}

\begin{proof}[Proof of \thmref{thm:proexcision_stcat}]
  Since $\{\Delta_n\}_n$ is weakly pro-Milnor, it is weakly pro-equivalent to the pro-system induced by the squares
  \[ \begin{tikzcd}
    \sA_n \ar{r}\ar{d}
      & \sB_n \ar{d}
    \\
    \sA'_n \ar{r}
      & \sA'_n \odot_{\sA_n}^{\sB'_n} \sB_n.
  \end{tikzcd} \]
  By \thmref{thm:excision}, $E$ sends the above square to a cartesian square of spectra for every $n$.
  Thus the claim follows from \corref{cor:proconnected}.
\end{proof}

\bibliographystyle{alphamod}

{\small
\bibliography{references}
}

\end{document}